\documentclass[12pt]{article}
\usepackage{mathrsfs}
\usepackage{}
\usepackage{amsfonts}
\usepackage{pifont}
\usepackage[english]{babel}
\RequirePackage{fix-cm}

\usepackage[pagewise]{lineno}
\usepackage{array}
\usepackage{longtable}
\usepackage{amsmath}
\usepackage{amsthm}
\usepackage{amssymb}
\usepackage{graphics}
\usepackage{graphicx}
\usepackage{subfigure}
\usepackage{epsfig}
\usepackage{delarray}
\usepackage{epstopdf}
\usepackage{color}
\usepackage{rotating}

\usepackage{graphics}
\usepackage{graphicx}
\usepackage{subfigure}
\usepackage{epsfig}
\usepackage{delarray}
\usepackage{multirow}

\begin{document}
\begin{sloppypar}

\newcommand{\tl}[1]{\multicolumn{1}{l}{#1}} %������
\renewcommand{\appendixname}{\appendix~{\thechapter}~}
\renewcommand{\theequation}{\thesection.\arabic{equation}}
\newtheorem{lm}{Lemma}
\newtheorem{defi}{Definition}
\newtheorem{prob}{Problem}
\newtheorem{thm}{Theorem}
\newtheorem{pro}{Proposition}
\newtheorem{exmp}{Example}
\newtheorem{rmk}{Remark}
\newtheorem{cor}{Corollary}
\newtheorem{con}{Conjecture}
\newcommand\T{\rule{-0.5pt}{2.6ex}}
\newcommand\B{\rule[-1.2ex]{0pt}{0pt}}
\vskip 0.5cm
\baselineskip 14pt
\parskip 8pt
\allowdisplaybreaks[4]
%%%%%%%%%%%%%%%%%%%
%%************************************

\title{Resonance properties and chaotic dynamics of a three-dimensional discrete logistic ecological system within the neighborhoods of bifurcation points\thanks{This work is supported
by the National Science Foundation of China (\#11701476, \#12171337), the National Science Foundation of
Sichuan Province(\#2023NSFSC0064).}}
%%%%%%%%%%%%

\author{{\sc Yujiang Chen}$^{a}$,~{\sc Lin Li}$^{b}$,~{\sc Lingling Liu}$^{c}$,~{\sc Zhiheng Yu}$^a$\footnote{Corresponding to: yuzhiheng9@163.com (Zhiheng Yu) 
}
\\
$^a${\small School of Mathematics, Southwest Jiaotong University,}
\\
{\small Chengdu, Sichuan 611756, P. R. China}
\\
$^b${\small College of Data Science, Jiaxing University,}
\\
{\small Jiaxing, Zhejiang 314001, P. R. China}
\\
$^c${\small School of Sciences, Southwest Petroleum University,}
\\
{\small Chengdu, Sichuan 610500, P. R. China}
}
\date{}

\maketitle
\vskip -1.2cm
\begin{abstract}
In this paper, we delve into the dynamical properties of a class of three-dimensional logistic ecological models. By using the complete discriminant theory of polynomials, we first give a topological classification for each fixed point and investigate the stability of corresponding system near the fixed points. Then employing the bifurcation and  normal form theory, we discuss all possible codimension-1 bifurcations near the fixed points, i.e., transcritical, flip, and Neimark-Sacker bifurcations, and further prove that the system can undergo codimension-2 bifurcations, specifically 1:2, 1:3, 1:4 strong resonances and weak resonance Arnold tongues. Additionally, chaotic behaviors in the sense of Marotto are rigorously analyzed. Numerical simulations are conducted to validate the theoretical findings and illustrate the complex dynamical phenomena identified.

\vskip 0.2cm
{\bf AMS(2010) Subject Classifications:} 37G10; 39A28; 58K50; 68W30

\end{abstract}

%%%%%%%%%%%%%%%%%%%

\section{Introduction}

The study of population dynamics in ecosystems is a highly regarded field. By investigating the population dynamics of species, we can understand trends in species variation, and then make effective control and management of species numbers. Interactions between populations have a significant impact on population dynamics. In ecosystems, interactions among populations can be categorized into predation, symbiosis and competition (\cite{Holland,Huisman1,Huisman2,Kaiser,Lotka,Murray,Volterra}). These interactions produce complex dynamical properties such as bifurcations and chaos (\cite{Arancibia,Chuang,Collet,May,Mira}).

The earliest model describing the interactions between two species (predator and prey) in ecosystem is the Lotka-Volterra model by Lotka and Volterra (\cite{Lotka,Volterra}). In 1965, Holling introduced a more realistic predator-prey model, which proposed functional responses of different species to simulate predation phenomena (\cite{Holling}). Since then, more and more scholars have studied the dynamic properties of the Lotka-Volterra model (\cite{Danca,Kasarinoff,Kuang,Zhu}).

In $2019$,  Vidiella et al. (see \cite{Vidiella}) proposed a discrete-time logistic ecological system
\begin{eqnarray}
\left\{
\begin{array}{l}
x_{n+1}=\mu x_n(1-x_n-y_n), \\
y_{n+1}=\beta x_n y_n,
\end{array}
\right.
\label{eq1.1}
\end{eqnarray}
where $x_n$ and $y_n$ represent the population density of prey and predators respectively, $\mu$ represents prey reproduction rate, and $\beta$ represents prey population growth rate. In \cite{Vidiella}, the authors analyzed the global stability of co-extinction fixed points of predators and prey, provided conditions for predator extinction and prey survival, and identified the existence of chaos phenomena in certain parameter regions. These results offer important insights into understanding the stability and dynamical behaviors of system \eqref{eq1.1}. Furthermore, in ecosystems food chains typically comprise multiple trophic levels,
such as the three-tier structure of ``prey-predator-top predator'' (see \cite{Elsadany}). The introduction of top predators not only directly regulates the population size of intermediate predators but may also indirectly affect prey dynamics through resource competition or behavioral regulation, creating complex hierarchical coupling effects. For instance, in marine ecosystems, the interactions between phytoplankton (prey), copepods (predators), and fish (top predators) often lead to dramatic population fluctuations due to environmental disturbances, potentially causing local extinctions(\cite{Uye}). However, the complexity of these multidimensional dynamics is difficult to fully capture with traditional two-dimensional models, highlighting the need for the development of a more realistic three-dimensional system theory framework (\cite{Hastings}).

%where predators that feed on basal species serving as prey for another predator. The aforementioned predator-prey model \eqref{eq1.1} can be extended to include a "top predator" species that preys on the population of predator (see \cite{Elsadany}). This result in a food chain composed of three species, whose corresponding discrete-time model exhibits more complex dynamical properties (\cite{Hastings}).

In $2020$, Alsed\`{a} et al. (see \cite{Alseda}) developed a three-dimensional discrete-time ecological system model based on their previous model \eqref{eq1.1}, considering interactions among three species
\begin{eqnarray}
\left\{
\begin{array}{l}
x_{n+1}=\mu x_n(1-x_n-y_n-z_n), \\
y_{n+1}=\beta y_n (x_n-z_n), \\
z_{n+1}=\lambda y_n z_n.
\end{array}
\right.
\label{eq1.2}
\end{eqnarray}
Here, $x_n$ and $y_n$ also represent the population density of prey and predators respectively, and $z_n$ represents the population density of top predators, $\mu$ denotes the reproduction rate of prey, $\beta$ represents the effective growth rate of predator $y_n$, and $\lambda$ denotes the growth rate of top predator $z_n$ due to the presence of specie $y_n$.
System \eqref{eq1.2} introduced a new top predator specie into the model, which consumes other predators and interferes with prey growth. This new model suggests that top predators consume other predators, while predators prey on prey. Furthermore, top predators exhibit negative interactions with prey growth due to predation or competition.
%Besides, there is category competition among prey.
%For instance, in the northern part of the Scandinavian Peninsula, certain insect species, by consuming butterfly eggs, emerge earlier than the latter. Other species, such as spiders, can also serve as top predators.
%The authors in \cite{Alseda} partially analyzed dynamical properties of this three-dimensional model and conducted numerical simulations on some codimension-1 bifurcations and chaos phenomena.
The authors in \cite{Alseda} provided a complete description of local stability of the fixed points for system \eqref{eq1.2}, and found that the increase of predator pressure on prey leads to chaos through supercritical Neimark-Sacker bifurcation. Then period-doubling bifurcation of the invariant curve occurs. The numerical bifurcation diagram and Lyapunov exponents were also presented to identify the period and chaotic state.
In this paper, based on the work in \cite{Alseda}, we improve the proofs of all codimension-1 bifurcations near the fixed points and further present the codimension-2 bifurcation sequences as the parameter varies around the strong resonance regions, from which an Arnold tongue corresponding to the rotation number $1/5$ is obtained. Then, by the theory of normal form, we prove theoretically the Arnold tongues in the weak resonances such that the system possesses two periodic orbits on the stable invariant closed curve generated from the Neimark-Sacker bifurcation.
Specifically, through transcritical bifurcation analysis, the critical reproduction rate for prey populations transitioning from extinction to persistent survival is determined, providing quantitative metrics for assessing population resilience. The study reveals that period-doubling phenomena induced by Flip bifurcations (e.g., period-2 orbits) correspond to alternating outbreaks and declines in population abundance, while invariant circles generated through Neimark-Sacker bifurcations reflect sustained oscillation patterns in predator-prey systems. The $1:2$, $1:3$ and $1:4$ strong resonance phenomena demarcate abrupt transition zones of dynamical behaviors in parameter space, which may correspond to critical transitions from ordered to chaotic states in ecosystems, providing a theoretical foundation for early warning signals. Within weak resonance regions, Arnold tongues delineate parameter ranges supporting periodic solutions, guiding parameter regulation in ecological management.
Additionally, we further consider the chaotic behaviors in the sense of Marotto when the boundary fixed point has a cubic characteristic polynomial.

The structure of this paper is as follows. In Section 2, we employ the complete polynomial discriminant system (\cite{YANG1,YANG2}) obtaining the topological classifications of three fixed points. Under the cases of non-hyperbolic, we present all codimension-1 bifurcations of system \eqref{eq1.2} near these fixed points. Specifically, Section 3.1 is devoted to demonstrate the presence of a transcritical bifurcation near the fixed point $O$ (the origin) by using center manifold theorem (\cite{Collet}). In Section 3.2, it is proved that both transcritical and flip bifurcations occur near the fixed point $E_1$. The occurrence of similar bifurcations as $E_1$ near the fixed point $E_2$, is established in Section 3.3. In Section 3.4, we prove that a Neimark-Sacker bifurcation may occur near $E_2$ under certain parametric conditions.
In the fourth and fifth sections, we explore codimension-2 bifurcations of the system, i.e., 1:2, 1:3, 1:4 strong resonances and weak resonance Arnold tongues (\cite{Whitley}), followed from the Neimark-Sacker bifurcation.
%It discusses cases where eigenvalues intersect the unit circle at the unit root $\lambda:= exp(2\pi\,{\bf i}\,n/m)$, where $n=m$ is an irreducible proper fraction, $m \geq 5$. In narrowly defined cusp-shaped parameter regions, a pair of m-periodic orbits exists within the invariant circle,
Section 6 focuses on the chaotic behaviors of the system in the sense of Marotto. By using the definition of snap-back repellers to establish the parametric conditions in theory, we show the appearance of Marotto chaotic behaviors in system \eqref{eq1.2}.
Finally, we present numerical simulations to system \eqref{eq1.2} in Section 7, i.e., codimension-1 and codimension-2 bifurcations including flip bifurcation diagrams, invariant circles generated by the Neimark-Sacker bifurcation, all the 1:2, 1:3, 1:4 strong resonances and Arnold tongues phenomena, as well as the chaotic behaviors in the sense of Marotto.

%%%%%%%%%%%%%%%%%%%%%%%%%%%%%%%%%%%%%%%%%%%%%%%%%%%%%%%%%%%%%%%%%%%%%%%%%%%%%%%%%%%%%%%%%%%%%%%%%%%%%%%%%%%%%%%%%%%%%%%%%%%%%%%%%%%%%
\section{Qualitative properties of fixed points}
\allowdisplaybreaks[4]
\setcounter{equation}{0}
System \eqref{eq1.2} is regarded as a three-dimensional mapping $F: {{\mathbb{R}}^3_0}\rightarrow {{\mathbb{R}}^3_0}$, given by
\begin{eqnarray}
F\left(
\begin{array}{l}
x \\
y \\
z
\end{array}
\right)
=
\left(
\begin{array}{c}
\mu x \left(1-x-y-z\right)\\
\beta y \left(x-z\right)\\
\lambda y z
\end{array}
\right),
\label{eq2.1}
\end{eqnarray}
where ${\mathbb{R}}^{3}_0:=\{(x, y ,z)\in {\mathbb{R}}^{3}|x\geq0,y\geq0,z\geq0,x+y+z\leq1\}$ and $\mu, \beta, \lambda\in \mathbb{R}_+^{3}$.

In this section, we discuss the stability of the fixed points of system \eqref{eq2.1} and list the topological classifications of the first three fixed points.
\begin{pro}
For parameter $\Lambda:=(\lambda,\mu,\beta)\in \mathbb{R}_+^{3}$, mapping \eqref{eq2.1} has at most four fixed points, i.e.,  a fixed point $O:(0,0,0)$, which always exists,
two boundary fixed points $E_1:(1-1/\mu,0,0)$ if {$\mu>1$}
and
$E_2:\left({1}/{\beta},1-1/\mu-1/\beta,0\right)$
if {$\beta\geq \mu/(\mu-1)$}, {$\mu>1$},
and a positive one
$$E_3:\left(\frac{1}{2}(1-\frac{1}{\mu}-\frac{1}{\lambda}+\frac{1}{\beta}),\frac{1}{\lambda},
\frac{1}{2}(1-\frac{1}{\mu}-\frac{1}{\lambda}-\frac{1}{\beta})
\right),$$
which exists in the case {$\beta>\lambda\mu/(\lambda\mu - \lambda - \mu)$}, $\lambda>\mu/(\mu-1)$ and $\mu>1$.
The topological classifications for $O$, $E_1$ and $E_2$ are respectively detailed in Tables \ref{Table1}-\ref{Table3}.
%Furthermore, the interior fixed point $E_3$ exhibits sink dynamics when the parameter $\Lambda$ satisfies certain conditions.
\label{pro2.1}
\end{pro}

\begin{table*}[htbp]
 \centering
 %\scriptsize
 \caption{Topological types of the fixed point $O$}
\label{Table1}
{\fontsize{12pt}{\baselineskip}\selectfont\begin{tabular}{lllll}
\hline\noalign{\smallskip}
\multicolumn{1}{c}{\multirow{2}*{Parameters ($\Lambda\in \mathbb{R}_+^{3}$)}}  & \multicolumn{1}{c}{Properties} & \multirow{2}*{Cases}\\ \noalign{\smallskip}\cline{2-2}\noalign{\smallskip}
\multicolumn{1}{c}{}& ~~~~~~~~$O$ &  \\ \noalign{\smallskip}\hline\noalign{\smallskip}
    $0<\mu<1$               & stable node   & $\mathbf{D}_{1}$\\
    $\mu=1$                 & non-hyperbolic& $\mathbf{L}_{1}$\\
    $\mu>1$                 & saddle point  &$\mathbf{D}_{2}$\\
 \noalign{\smallskip}\hline
 \end{tabular}}
\end{table*}

\begin{table*}[htbp]
 \centering
 %\scriptsize
 \caption{Topological types of the fixed point $E_1$}
\label{Table2}
{\fontsize{12pt}{\baselineskip}\selectfont\begin{tabular}{lllll}
\hline\noalign{\smallskip}
\multicolumn{2}{c}{\multirow{2}*{Parameters ($\Lambda\in \mathbb{R}_+^{3}$)}}  & \multicolumn{1}{c}{Properties} & \multirow{2}*{Cases}\\ \noalign{\smallskip}\cline{3-3}\noalign{\smallskip}
\multicolumn{2}{c}{}& ~~~~~~~~$E_1$ &  \\ \noalign{\smallskip}\hline\noalign{\smallskip}
$1<\mu<3$  & $0<\beta<\frac{\mu}{\mu-1}$               & stable node   & $\mathcal {D}_{1}$\\
                 & $\beta=\frac{\mu}{\mu-1}$                 & non-hyperbolic& $\mathcal {L}_{1}$\\
                 & $\beta>\frac{\mu}{\mu-1}$                 & saddle point  &$\mathcal {D}_{31}$\\
$\mu=3$       & $\beta>0$ &non-hyperbolic &$\mathcal {L}_{2}$\\
$\mu>3$     & $0<\beta<\frac{\mu}{\mu-1}$ & saddle point&$\mathcal {D}_{32}$\\
                 & $\beta=\frac{\mu}{\mu-1}$ & non-hyperbolic & $\mathcal {L}_{3}$\\
                 & $\beta>\frac{\mu}{\mu-1}$ & unstable node   & $\mathcal {D}_{2}$\\
 \noalign{\smallskip}\hline
 \end{tabular}}
\end{table*}

%%%%%%%%%%%%%%%%%%%%%%%%%%%%%%%%%%%%
\begin{center}
{\rm Table 3: Topological types of the fixed point $E_2$}
\end{center}
{\rm\footnotesize
\begin{longtable}{p{20mm}p{30mm}p{35mm}p{35mm}p{10mm}}
%\begin{table*}[htbp]
 %\centering
 %\scriptsize
 %\caption{Topological types of fixed point $E_2$.}
%\label{Table3}
%{\fontsize{12pt}{\baselineskip}\selectfont\begin{tabular}{lllll}
\hline\noalign{\smallskip}
\multicolumn{3}{c}{\multirow{2}*{Conditions ($\Lambda\in \mathbb{R}_+^{3}$)}}  & \multicolumn{1}{c}{Properties} & \multirow{2}*{Cases}\\
\noalign{\smallskip}\cline{4-4}\noalign{\smallskip}
\multicolumn{3}{c}{}& ~~~~~~~~$E_2$ &  \\
\noalign{\smallskip}\hline\noalign{\smallskip}
$1<\beta\leq\frac{3}{2}$ & $\mu>\frac{\beta}{\beta-1}$ & $0<\lambda<\frac{\mu \beta}{\mu \beta-\mu-\beta}$ & saddle point & $\mathfrak{D}_{31}$\\
             & &$\lambda=\frac{\mu  \beta}{\beta  \mu -\beta -\mu}$& non-hyperbolic& $\mathfrak{L}_{11}$\\
             & &$\lambda>\frac{\mu  \beta}{\beta  \mu -\beta -\mu}$& saddle point  &$\mathfrak{D}_{32}$\\
$\frac{3}{2}<\beta\leq2$ & $\frac{\beta}{\beta -1}<\mu <-\frac{3 \beta}{\beta-3}$ & $0<\lambda<\frac{\mu \beta}{\mu \beta-\mu-\beta}$ & stable node & $\mathfrak{D}_{11}$\\
             & &$\lambda=\frac{\mu  \beta}{\beta  \mu -\beta -\mu}$& non-hyperbolic& $\mathfrak{L}_{21}$\\
             & &$\lambda>\frac{\mu  \beta}{\beta  \mu -\beta -\mu}$& saddle point  &$\mathfrak{D}_{33}$\\                         &$\mu=-\frac{3 \beta}{\beta -3}$& $\lambda>0$ &non-hyperbolic &$\mathfrak{L}_{22}$\\
             &$\mu>-\frac{3 \beta}{\beta -3}$& $0<\lambda<\frac{\mu  \beta}{\mu  \beta -\mu -\beta}$ &saddle point &$\mathfrak{D}_{34}$\\
             & &$\lambda=\frac{\mu  \beta}{\beta  \mu -\beta -\mu}$& non-hyperbolic& $\mathfrak{L}_{13}$\\
             & &$\lambda>\frac{\mu  \beta}{\beta  \mu -\beta -\mu}$& saddle point  &$\mathfrak{D}_{35}$\\
$2<\beta<\frac{9}{4}$ & $\frac{\beta}{\beta -1}<\mu \le \gamma_1$ & $0<\lambda<\frac{\mu \beta}{\mu \beta-\mu-\beta}$ & stable node & $\mathfrak{D}_{12}$\\
             & &$\lambda=\frac{\mu  \beta}{\beta  \mu -\beta -\mu}$& non-hyperbolic& $\mathfrak{L}_{14}$\\
             & &$\lambda>\frac{\mu  \beta}{\beta  \mu -\beta -\mu}$& saddle point  &$\mathfrak{D}_{36}$\\
             &$\gamma_1<\mu <\gamma_2$& $0<\lambda<\frac{\mu  \beta}{\mu  \beta -\mu -\beta}$ &stable focus-node&$\mathfrak{D}_{51}$\\
             & &$\lambda=\frac{\mu  \beta}{\beta  \mu -\beta -\mu}$& non-hyperbolic& $\mathfrak{L}_{15}$\\
             & &$\lambda>\frac{\mu  \beta}{\beta  \mu -\beta -\mu}$& saddle-focus  &$\mathfrak{D}_{61}$\\
             &$\gamma_2\leq\mu <-\frac{3 \beta}{\beta -3}$& $0<\lambda<\frac{\mu  \beta}{\mu  \beta -\mu -\beta}$ &stable node&$\mathfrak{D}_{13}$\\
             & &$\lambda=\frac{\mu  \beta}{\beta  \mu -\beta -\mu}$& non-hyperbolic& $\mathfrak{L}_{16}$\\
             & &$\lambda>\frac{\mu  \beta}{\beta  \mu -\beta -\mu}$& saddle point&$\mathfrak{D}_{37}$\\
             &$\mu=-\frac{3 \beta}{\beta -3}$& $\lambda>0$ &non-hyperbolic &$\mathfrak{L}_{23}$\\
             &$\mu >-\frac{3 \beta}{\beta -3}$& $0<\lambda<\frac{\mu  \beta}{\mu  \beta -\mu -\beta}$ &saddle point&$\mathfrak{D}_{38}$\\
             & &$\lambda=\frac{\mu  \beta}{\beta  \mu -\beta -\mu}$& non-hyperbolic& $\mathfrak{L}_{17}$\\
             & &$\lambda>\frac{\mu  \beta}{\beta  \mu -\beta -\mu}$& saddle point&$\mathfrak{D}_{39}$\\
$\beta=\frac{9}{4}$ & $\frac{\beta}{\beta -1}<\mu \le \gamma_1$ & $0<\lambda<\frac{\mu \beta}{\mu \beta-\mu-\beta}$ & stable node & $\mathfrak{D}_{14}$\\
             & &$\lambda=\frac{\mu  \beta}{\beta  \mu -\beta -\mu}$& non-hyperbolic& $\mathfrak{L}_{18}$\\
             & &$\lambda>\frac{\mu  \beta}{\beta  \mu -\beta -\mu}$& saddle point  &$\mathfrak{D}_{41}$\\
             &$\gamma_1<\mu <9$& $0<\lambda<\frac{\mu  \beta}{\mu  \beta -\mu -\beta}$ &stable focus-node&$\mathfrak{D}_{52}$\\
             & &$\lambda=\frac{\mu  \beta}{\beta  \mu -\beta -\mu}$& non-hyperbolic& $\mathfrak{L}_{19}$\\
             & &$\lambda>\frac{\mu  \beta}{\beta  \mu -\beta -\mu}$& saddle-focus  &$\mathfrak{D}_{62}$\\
             &$\mu=9$& $\lambda>0$ &non-hyperbolic &$\mathfrak{L}_{24}$\\
             &$\mu >9$& $0<\lambda<\frac{\mu  \beta}{\mu  \beta -\mu -\beta}$ &saddle point&$\mathfrak{D}_{42}$\\
             & &$\lambda=\frac{\mu  \beta}{\beta  \mu -\beta -\mu}$& non-hyperbolic& $\mathfrak{L}_{1-10}$\\
             & &$\lambda>\frac{\mu  \beta}{\beta  \mu -\beta -\mu}$& saddle point&$\mathfrak{D}_{43}$\\
$\frac{9}{4}<\beta<3$ & $\frac{\beta}{\beta -1}<\mu \le \gamma_1$ & $0<\lambda<\frac{\mu \beta}{\mu \beta-\mu-\beta}$ & stable node & $\mathfrak{D}_{15}$\\
             & &$\lambda=\frac{\mu  \beta}{\beta  \mu -\beta -\mu}$& non-hyperbolic& $\mathfrak{L}_{1-11}$\\
             & &$\lambda>\frac{\mu  \beta}{\beta  \mu -\beta -\mu}$& saddle point  &$\mathfrak{D}_{44}$\\
             &$\gamma_1<\mu <\frac{\beta}{\beta -2}$& $0<\lambda<\frac{\mu  \beta}{\mu  \beta -\mu -\beta}$ &stable focus-node&$\mathfrak{D}_{53}$\\
             & &$\lambda=\frac{\mu  \beta}{\beta  \mu -\beta -\mu}$& non-hyperbolic& $\mathfrak{L}_{1-12}$\\
             & &$\lambda>\frac{\mu  \beta}{\beta  \mu -\beta -\mu}$& saddle-focus  &$\mathfrak{D}_{63}$\\
             &$\mu=\frac{\beta}{\beta -2}$& $\lambda>0$ &non-hyperbolic& $\mathfrak{L}_{31}$\\
             &$\frac{\beta}{\beta -2}<\mu <\gamma_2$& $0<\lambda<\frac{\mu  \beta}{\mu  \beta -\mu -\beta}$ &saddle-focus&$\mathfrak{D}_{64}$\\
             & &$\lambda=\frac{\mu  \beta}{\beta  \mu -\beta -\mu}$& non-hyperbolic& $\mathfrak{L}_{1-13}$\\
             & &$\lambda>\frac{\mu  \beta}{\beta  \mu -\beta -\mu}$& unstable focus-node&$\mathfrak{D}_{71}$\\
             &$\gamma_2\leq\mu <-\frac{3 \beta}{\beta -3}$& $0<\lambda<\frac{\mu  \beta}{\mu  \beta -\mu -\beta}$ &saddle point&$\mathfrak{D}_{45}$\\
             & &$\lambda=\frac{\mu  \beta}{\beta  \mu -\beta -\mu}$& non-hyperbolic& $\mathfrak{L}_{1-14}$\\
             & &$\lambda>\frac{\mu  \beta}{\beta  \mu -\beta -\mu}$& unstable node&$\mathfrak{D}_{21}$\\
             &$\mu=-\frac{3 \beta}{\beta -3}$& $\lambda>0$ &non-hyperbolic &$\mathfrak{L}_{25}$\\
             &$\mu >-\frac{3 \beta}{\beta -3}$& $0<\lambda<\frac{\mu  \beta}{\mu  \beta -\mu -\beta}$ &saddle point&$\mathfrak{D}_{46}$\\
             & &$\lambda=\frac{\mu  \beta}{\beta  \mu -\beta -\mu}$& non-hyperbolic& $\mathfrak{L}_{1-15}$\\
             & &$\lambda>\frac{\mu  \beta}{\beta  \mu -\beta -\mu}$& saddle point&$\mathfrak{D}_{47}$\\
$\beta\geq3$ & $\frac{\beta}{\beta -1}<\mu \le \gamma_1$ & $0<\lambda<\frac{\mu \beta}{\mu \beta-\mu-\beta}$ & stable node & $\mathfrak{D}_{16}$\\
             & &$\lambda=\frac{\mu  \beta}{\beta  \mu -\beta -\mu}$& non-hyperbolic& $\mathfrak{L}_{1-16}$\\
             & &$\lambda>\frac{\mu  \beta}{\beta  \mu -\beta -\mu}$& saddle point  &$\mathfrak{D}_{48}$\\
             &$\gamma_1<\mu <\frac{\beta}{\beta-2}$& $0<\lambda<\frac{\mu  \beta}{\mu  \beta -\mu -\beta}$ &stable focus-node&$\mathfrak{D}_{54}$\\
             & &$\lambda=\frac{\mu  \beta}{\beta  \mu -\beta -\mu}$& non-hyperbolic& $\mathfrak{L}_{1-17}$\\
             & &$\lambda>\frac{\mu  \beta}{\beta  \mu -\beta -\mu}$& saddle-focus  &$\mathfrak{D}_{65}$\\
             &$\mu=\frac{\beta}{\beta-2}$& $\lambda>0$ &non-hyperbolic &$\mathfrak{L}_{32}$\\
             &$\frac{\beta}{\beta-2}<\mu <\gamma_2$& $0<\lambda<\frac{\mu  \beta}{\mu  \beta -\mu -\beta}$ &saddle-focus&$\mathfrak{D}_{66}$\\
             & &$\lambda=\frac{\mu  \beta}{\beta  \mu -\beta -\mu}$& non-hyperbolic& $\mathfrak{L}_{1-18}$\\
             & &$\lambda>\frac{\mu  \beta}{\beta  \mu -\beta -\mu}$& unstable focus-node&$\mathfrak{D}_{72}$\\
             &$\mu \geq\gamma_2$& $0<\lambda<\frac{\mu  \beta}{\mu  \beta -\mu -\beta}$ &saddle point&$\mathfrak{D}_{49}$\\
             & &$\lambda=\frac{\mu  \beta}{\beta  \mu -\beta -\mu}$& non-hyperbolic& $\mathfrak{L}_{1-19}$\\
             & &$\lambda>\frac{\mu  \beta}{\beta  \mu -\beta -\mu}$& unstable node&$\mathfrak{D}_{22}$\\
 \noalign{\smallskip}\hline
 %\end{tabular}}
%\end{table*}
\label{Table3}
\end{longtable}}

\begin{proof}
By solving the following equations
\begin{eqnarray}
\left\{\begin{array}{ll}
x=\mu x \left(1-x-y-z\right),  \\
y=\beta y \left(x-z\right),  \\
z=\lambda y z,
\end{array}\right.
\label{eq2.2}
\end{eqnarray}
we get the fixed points of system \eqref{eq2.1}. More concretely,
a solution $(x,y,z)=(0,0,0)$ of \eqref{eq2.2} always exists since $\lambda>0,\beta>0,\mu>0$;
equations \eqref{eq2.2} has a solution $(x,y,z)=((\mu-1)/\mu,0,0)$ if $\mu>1$;
if $\mu\geq \beta/(\beta-1)$ and $\beta>1$, \eqref{eq2.2} has a solution
$$(x,y,z)=\left(\frac{1}{\beta},1-\frac{1}{\mu}-\frac{1}{\beta},0\right).$$
Besides, the point
$$(x,y,z)=\left(\frac{1}{2}(1-\frac{1}{\mu}-\frac{1}{\lambda}+\frac{1}{\beta}),\frac{1}{\lambda},
\frac{1}{2}(1-\frac{1}{\mu}-\frac{1}{\lambda}-\frac{1}{\beta})\right)$$
is also a solution of \eqref{eq2.2} if $\mu>1$ , $\lambda>\mu/(\mu-1)$ and $\beta>\lambda\mu/(\lambda\mu - \lambda - \mu)$.
Thus, system \eqref{eq2.1} has four fixed points $O, E_1, E_2$ and $E_3$, i.e.,
\begin{eqnarray*}
&&O:(0, 0, 0) , E_1:\left(1-\frac{1}{\mu},0,0\right) , E_2:\left(\frac{1}{\beta},1-\frac{1}{\mu}-\frac{1}{\beta},0\right)\\
&&E_3:\left(\frac{1}{2}(1-\frac{1}{\mu}-\frac{1}{\lambda}+\frac{1}{\beta}),\frac{1}{\lambda},
\frac{1}{2}(1-\frac{1}{\mu}-\frac{1}{\lambda}-\frac{1}{\beta})\right).
\end{eqnarray*}

In what follows, we discuss their topological classifications. The Jacobian matrices at $O$, $E_1$ and $E_2$ are given by
$$
JF(O)=\left(
        \begin{array}{ccc}
          \mu & 0 & 0 \\
          0 & 0 & 0 \\
          0 & 0 & 0
        \end{array}
      \right),
JF(E_1)\!=\!\left(\!\!
        \begin{array}{ccc}
          -\mu +2 & -\mu +1 & -\mu +1 \\
          0 & {\beta  \left(\mu -1\right)}/{\mu} & 0 \\
          0 & 0 & 0
        \end{array}
     \!\! \right),
$$
$$
JF(E_2)\!=\!\left(\!\!
        \begin{array}{ccc}
          {(\beta -\mu)}/{\beta} & -{\mu}/{\beta} & -{\mu}/{\beta} \\
          {(\beta  \mu -\beta -\mu)}/{\mu} & 1 & -{(\beta  \mu -\beta -\mu)}/{\mu} \\
          0 & 0 & {\lambda  \left(\beta  \mu -\beta -\mu \right)}/{\beta  \mu}
        \end{array}
     \!\! \right),
$$
respectively.
Clearly, the eigenvalues of $JF(O)$ are $0, 0$ and $\mu$.
Hence, the fixed point $O$ is non-hyperbolic if and only if $\mu=1$.
Besides, the fixed point $O$ is hyperbolic and its topological types are listed as follows.
\begin{description}
\item[(i)] If $0<\mu<1$, the fixed point $O$ is a stable node (referred to case $\mathbf{D}_{1}$ in Table \ref{Table1}).
\item[(ii)] If $\mu>1$, the fixed point $O$ is a saddle point (referred to case $\mathbf{D}_{2}$ in Table \ref{Table1}).
\end{description}

Next, we discuss the qualitative properties of $E_1$. One can check that the matrix $JF(E_1)$ has eigenvalues $0, -\mu + 2$ and $\beta(\mu-1)/\mu$, which are all real. Then we have
\begin{description}
\item[(i)] If $E_1$ is a stable node, then $|-\mu + 2|<1$, $|\beta(\mu-1)/\mu|<1$, and we obtain the corresponding semi-algebraic system
    \begin{eqnarray*}
    &&\!\!\!\!\!\!\!\!\!\!\!\!\!\!PS_1:=\left\{\mu>1,\beta>0,-\mu+2<1,-\mu+2>-1,{\beta  \left(\mu -1\right)}/{\mu}<1,{\beta  \left(\mu -1\right)}/{\mu}>-1\right\}.
    \end{eqnarray*}
    Solving $PS_1$, we get the following parameter conditions
    \begin{description}
    \item[$\bullet$]~$1<\mu<3,0<\beta<{\mu}/{(\mu -1)}$ (referred to case $\mathcal {D}_{1}$ in Table \ref{Table2}).
    \end{description}
\item[(ii)] If $E_1$ is an unstable node, then $|-\mu + 2|>1$ and $|\beta(\mu-1)/\mu|>1$, i.e.,
    \begin{description}
    \item[$\bullet$]~$\mu>3 , \beta>{\mu}/{(\mu -1)}$ (referred to case $\mathcal {D}_{2}$ in Table \ref{Table2}).
    \end{description}

\item[(iii)] If $E_1$ is a saddle point, then either $|-\mu + 2|>1$, $|\beta(\mu-1)/\mu|<1$, or $|-\mu + 2|<1$, $|\beta(\mu-1)/\mu|>1$, i.e.,
    \begin{description}
    \item[$\bullet$]~$1<\mu<3,\beta>{\mu}/{(\mu -1)}$
    \item[$\bullet$]~$\mu>3,0<\beta<{\mu}/{(\mu -1)}$
    \end{description}
    (referred to cases $\mathcal {D}_{31}$ and $\mathcal {D}_{32}$ in Table \ref{Table2}, respectively).
\item[(iv)] If $E_1$ is non-hyperbolic, then $|-\mu + 2|=1$ or $|\beta(\mu-1)/\mu|=1$. Solving the corresponding semi-algebraic systems, we get
    \begin{description}
    \item[$\bullet$]~$1<\mu<3 ,\beta={\mu}/{(\mu -1)},$
    \item[$\bullet$]~$\mu=3, \beta>0,$
    \item[$\bullet$]~$\mu>3, \beta={\mu}/{(\mu -1)}$
    \end{description}
    (referred to cases $\mathcal {L}_{1}, \mathcal {L}_{2}$ and $\mathcal {L}_{3}$ in Table \ref{Table2}, respectively).
\end{description}

Finally, we consider the topological types of the fixed point $E_2$. The characteristic polynomial of $JF(E_2)$ is given by
\begin{eqnarray*}
\mathcal{P}_{E_2}(t)&:=&\left(t-\frac{\lambda  \left(\beta  \mu -\beta -\mu \right)}{\beta \mu}\Big)\Big(
-t^{2}+\frac{2 \beta -\mu}{\beta} t-\frac{\mu  \left(\beta -2\right)}{\beta}\right).
\end{eqnarray*}
Hence, one real eigenvalue of $JF(E_2)$ is $\lambda(\beta\mu-\beta-\mu)/\beta\mu$, and the other two eigenvalues can be determined from the zeros of the polynomial
\begin{eqnarray*}
&&{P}_{E_2}(t):=\frac{\mathcal{P}_{E_2}(t)}{t-\frac{\lambda \left(\beta  \mu -\beta -\mu \right)}{\beta \mu}}=-t^{2}+\frac{2 \beta -\mu}{\beta} t-\frac{\mu  \left(\beta -2\right)}{\beta}.
\end{eqnarray*}
Solving ${P}_{E_2}(t)=0$, we obtain
\begin{eqnarray}
t_{1,2}=\frac{2 \beta -\mu \pm\sqrt{-4 \beta^{2} \mu +4 \beta^{2}+4 \beta  \mu +\mu^{2}}}{2 \beta}.
\label{solu-1}
\end{eqnarray}
This follows that
\begin{description}
\item[(i)] If $E_2$ is a stable node, then $|\lambda(\beta\mu - \beta - \mu)/\beta\mu|<1$ and $|t_{1,2}|<1$, i.e., $|\lambda(\beta\mu - \beta - \mu)/\beta\mu|<1$, ${P_{E_2}}(1)<0$, ${P_{E_2}}(-1)<0$, $\Delta\geq0$ and $-1<t_*<1$, where $t_*:=(2\beta - \mu)/2\beta$, denotes the axis symmetry of $P_{E_2}(t)$, $\Delta$ stands for the discriminant of $P_{E_2}(t)$. Thus, the corresponding semi-algebraic system is
    \begin{eqnarray*}
    &&\!\!\!\!\!\!\!\!\!\!PS_2:=\left\{\mu>0,\beta>0,\lambda>0,{(\beta\mu-\beta-\mu)}/{\beta\mu}\ge 0,{(-\beta  \mu +\beta +\mu)}/{\beta}<0\right., \\
    &&~~~~~~~{(3 \mu-\mu\beta -3\beta)}/{\beta}<0,{(-4 \mu \beta^{2}+4\beta^{2}+4 \beta  \mu +\mu^{2})}/{\beta^{2}}\ge 0,{(2 \beta -\mu)}/{2 \beta}<1,\\
    &&~~~~~~~\left.{(2 \beta -\mu)}/{2 \beta}>-1,{(\lambda \left(\beta  \mu -\beta -\mu \right))}/{\beta \mu}<1,{(\lambda\left(\beta  \mu -\beta -\mu \right))}/{\beta \mu}>-1\right\}.
    \end{eqnarray*}
   Employing the polynomial complete discrimination system theory to solve $PS_2$, we obtain the following parameter conditions :
    \begin{description}
    \item[$\bullet$]~${3}/{2}<\beta\leq2,{\beta}/{(\beta -1)}<\mu<-{3 \beta}/{(\beta -3)},0<\lambda<{\mu  \beta}/{(\mu  \beta -\mu -\beta)}$,
    \item[$\bullet$]~$2<\beta\leq{9}/{4},{\beta}/{(\beta -1)}<\mu\leq\gamma_1,0<\lambda<{\mu  \beta}/{(\mu  \beta -\mu -\beta)}$,
    \item[$\bullet$]~$2<\beta<{9}/{4},\gamma_2\leq\mu<-{3 \beta}/{(\beta -3)},0<\lambda<{\mu  \beta}/{(\mu  \beta -\mu -\beta)}$,
    \item[$\bullet$]~${9}/{4}<\beta<3,{\beta}/{(\beta -1)}<\mu\leq\gamma_1,0<\lambda<{\mu  \beta}/{(\mu  \beta -\mu -\beta)}$,
    \item[$\bullet$]~$\beta\geq3 ,{\beta}/{\beta -1}<\mu\leq\gamma_1,0<\lambda<{\mu  \beta}/{(\mu  \beta -\mu -\beta)}$,
    \end{description}
    where
    \begin{eqnarray*}
    \gamma_1:={2 \beta^{2}-2 \beta -2 \sqrt{\beta^{4}-2 \beta^{3}}},\\
    \gamma_2:={2 \beta^{2}-2 \beta +2 \sqrt{\beta^{4}-2 \beta^{3}}}
    \end{eqnarray*}
    (referred to cases $\mathfrak{D}_{1i}$ in Table \ref{Table3}, where $i= 1,2,3,4,5,6$).
\item[(ii)] If $E_2$ is an unstable node, then $|\lambda(\beta\mu - \beta - \mu)/\beta\mu|>1$ and $|t_{1,2}|>1$. Moreover, $|t_{1,2}|>1$ implies that one of the following conditions holds:
    \begin{description}
    \item[$\lozenge$]~${P_{E_2}}(1)>0$, ${P_{E_2}}(-1)>0$, $\Delta>0$,
    \item[$\lozenge$]~${P_{E_2}}(1)<0$, ${P_{E_2}}(-1)<0$, $\Delta\geq0 $, $t_*<-1$,
    \item[$\lozenge$]~${P_{E_2}}(1)<0$, ${P_{E_2}}(-1)<0$, $\Delta\geq0 $, $t_*>1$.
    \end{description}
    Using the same idea as done in case (i), we get
    \begin{description}
    \item[$\bullet$]~$9/4<\beta<3,\gamma_2\leq\mu<-3 \beta/(\beta -3),\lambda>\mu  \beta/(\mu  \beta -\mu -\beta)$,\\
    \item[$\bullet$]~$\beta\geq3,\mu\geq\gamma_2,\lambda>\mu  \beta/(\mu  \beta -\mu -\beta)$
    \end{description}
    (referred to cases $\mathfrak{D}_{2i}$ in Table \ref{Table3}, where $i= 1,2$).
\item[(iii)] If $E_2$ is a saddle point, then $\lambda(\beta\mu - \beta - \mu)/\beta\mu$, $t_1$ and $t_2$ are all real. This implies that one of the following conditions is satisfied:
    \begin{description}
    \item[$\lozenge$]~$|\lambda(\beta\mu - \beta - \mu)/\beta\mu|>1$, ${P_{E_2}}(1)<0$, ${P_{E_2}}(-1)<0$, $\Delta\geq0$, $-1<t_*<1$,
    \item[$\lozenge$]~$|\lambda(\beta\mu - \beta - \mu)/\beta\mu|<1$, ${P_{E_2}}(1)>0$, ${P_{E_2}}(-1)<0$, $\Delta\geq0 $,
    \item[$\lozenge$]~$|\lambda(\beta\mu - \beta - \mu)/\beta\mu|<1$, ${P_{E_2}}(1)<0$, ${P_{E_2}}(-1)>0$, $\Delta\geq0 $,
    \item[$\lozenge$]~$|\lambda(\beta\mu - \beta - \mu)/\beta\mu|>1$, ${P_{E_2}}(1)>0$, ${P_{E_2}}(-1)<0$, $\Delta\geq0 $,
    \item[$\lozenge$]~$|\lambda(\beta\mu - \beta - \mu)/\beta\mu|>1$, ${P_{E_2}}(1)<0$, ${P_{E_2}}(-1)>0$, $\Delta\geq0 $,
    \item[$\lozenge$]~$|\lambda(\beta\mu - \beta - \mu)/\beta\mu|<1$, ${P_{E_2}}(1)>0$, ${P_{E_2}}(-1)>0$, $\Delta>0$,
    \item[$\lozenge$]~$|\lambda(\beta\mu - \beta - \mu)/\beta\mu|<1$, ${P_{E_2}}(1)<0$, ${P_{E_2}}(-1)<0$, $\Delta\geq0 $, $t_*<-1$,
    \item[$\lozenge$]~$|\lambda(\beta\mu - \beta - \mu)/\beta\mu|<1$, ${P_{E_2}}(1)<0$, ${P_{E_2}}(-1)<0$, $\Delta\geq0 $, $t_*>1$.
    \end{description}
    Hence, similar to the discussion in the above two cases, we obtain
    \begin{description}
    \item[$\bullet$]~$1<\beta\le {3}/{2},\mu>{\beta}/{(\beta -1)},0<\lambda<{\mu  \beta}/{(\mu  \beta -\mu -\beta)}$,
    \item[$\bullet$]~$1<\beta\le {3}/{2},\mu >{\beta}/{\beta-1},\lambda>{\mu  \beta}/{(\mu  \beta -\mu -\beta)}$,
    \item[$\bullet$]~${3}/{2}<\beta \le 2,{\beta}/{(\beta -1)}<\mu <-{3 \beta}/{(\beta -3)},\lambda>{\mu  \beta}/{(\mu \beta-\mu-\beta)}$,
    \item[$\bullet$]~${3}/{2}<\beta<3,\mu>-{3 \beta}/{(\beta -3)},0<\lambda<{\mu  \beta}/{(\mu  \beta -\mu -\beta)}$,
    \item[$\bullet$]~${3}/{2}<\beta<3,\mu >-{3 \beta}/{(\beta -3)},\lambda>{\mu  \beta}/{(\mu  \beta -\mu -\beta)}$,
    \item[$\bullet$]~$\beta>2,{\beta}/{(\beta -1)}<\mu \le \gamma_1\,,\lambda>{\mu  \beta}/{(\mu  \beta -\mu -\beta)}$,
    \item[$\bullet$]~$2<\beta <{9}/{4},\gamma_2\,\le \mu <-{3 \beta}/{(\beta -3)},\lambda>{\mu  \beta}/{(\mu  \beta -\mu -\beta)}$,
    \item[$\bullet$]~${9}/{4}<\beta <3,\gamma_2\,\le \mu <-{3 \beta}/{\beta-3},0<\lambda <{\mu  \beta}/{(\mu \beta-\mu-\beta)}$,
    \item[$\bullet$]~$\beta\geq3,\mu\ge \gamma_2\,,0<\lambda<{\mu  \beta}/{(\mu  \beta -\mu -\beta)}$
    \end{description}
    (referred to cases $\mathfrak{D}_{3i}$ and $\mathfrak{D}_{4i}$ in Table \ref{Table3}, where $i= 1,2,...,9$).
\item[(iv)] If $E_2$ is a stable focus-node, then $|\lambda(\beta\mu - \beta - \mu)/\beta\mu|<1$, $\Delta<0$ and $|t_1\,t_2|<1$. So we get
    \begin{description}
    \item[$\bullet$]~$2<\beta\leq{9}/{4},\gamma_1\,<\mu <\gamma_2,0<\lambda <{\mu  \beta}/{(\mu \beta-\mu-\beta)}$,
    \item[$\bullet$]~$\beta>{9}/{4},\gamma_1\,<\mu <{\beta}/{(\beta -2)},0<\lambda <{\mu  \beta}/{(\mu  \beta -\mu -\beta)}$
    \end{description}
    (referred to cases $\mathfrak{D}_{5i}$ in Table \ref{Table3}, where $i= 1,2,3,4$).
\item[(v)] If $E_2$ is a saddle-focus, then one of the following conditions holds:
    \begin{description}
    \item[$\lozenge$]~$|\lambda(\beta\mu - \beta - \mu)/\beta\mu|>1$, $\Delta<0$, $|t_1\,t_2|<1$,
    \item[$\lozenge$]~$|\lambda(\beta\mu - \beta - \mu)/\beta\mu|<1$, $\Delta<0$, $|t_1\,t_2|>1$.
    \end{description}
    Solving the corresponding semi-algebraic systems, we get
   \begin{description}
    \item[$\bullet$]~$2<\beta\leq{9}/{4},\gamma_1<\mu <\gamma_2,\lambda>{\mu  \beta}/{(\mu  \beta -\mu -\beta)}$,
    \item[$\bullet$]~$\beta>{9}/{4},\gamma_1<\mu <{\beta}/{(\beta -2)},\lambda>{\mu  \beta}/{(\mu  \beta -\mu -\beta)}$,
    \item[$\bullet$]~$\beta>{9}/{4},{\beta}/{(\beta -2)}<\mu <\gamma_2,0<\lambda <{\mu  \beta}/{(\mu  \beta -\mu -\beta)}$
    \end{description}
    (referred to cases $\mathfrak{D}_{6i}$ in Table \ref{Table3}, where $i= 1,2,3,4,5,6$).
\item[(vi)] If $E_2$ is an unstable focus-node, then $|\lambda(\beta\mu - \beta - \mu)/\beta\mu|>1$, $\Delta<0$ and $|t_1\,t_2|>1$. Similar to the discussion in the above cases, we have
    \begin{description}
    \item[$\bullet$]~$\beta >{9}/{4},{\beta}/{(\beta -2)}<\mu <\gamma_2\,,\lambda>{\mu  \beta}/{(\mu  \beta -\mu -\beta)}$
    \end{description}
    (referred to cases $\mathfrak{D}_{7i}$ in Table \ref{Table3}, where $i= 1,2$).
\item[(vii)] If $E_2$ is non-hyperbolic, then one of the following conditions holds:
    \begin{description}
    \item[$\lozenge$]~$|\lambda(\beta\mu - \beta - \mu)/\beta\mu|=1$,
    \item[$\lozenge$]~$0<\Delta\geq0$, ${P_{E_2}}(1)=0$,
    \item[$\lozenge$]~$0<\Delta\geq0$, ${P_{E_2}}(-1)=0$,
    \item[$\lozenge$]~$\Delta<0$, $|t_1\,t_2|=1$.
    \end{description}
We further obtain
    \begin{description}
    \item[$\bullet$]~$\beta>1,\mu>{\beta}/{(\beta-1)},\lambda={\mu  \beta}/{(\beta  \mu -\beta -\mu)}$,
    \item[$\bullet$]~${3}/{2}<\beta<3,\mu=-{3 \beta}/{(\beta -3)},\lambda>0$,
    \item[$\bullet$]~$\beta>{9}/{4},\mu={\beta}/{(\beta -2)},\lambda>0$
    \end{description}
    (referred to cases $\mathfrak{L}_{1i}$, $\mathfrak{L}_{2j}$ and $\mathfrak{L}_{3k}$ in Table \ref{Table3}, where $i= 1,2,...,19$, $j= 1,2,3,4,5$, $k= 1,2$).
\end{description}
The whole proof is completed.
\end{proof}

\begin{rmk}
Since the characteristic polynomial of $E_3$ is a cubic polynomial, personal computers lack the computational resources required for bifurcation analysis, thus we can only consider the stability characteristics of $E_3$. The Jacobian matrix evaluated at $E_3$ is expressed as
$$
JF(E_3)\!=\!\left(\!\!
        \begin{array}{ccc}
          a_1-1 & a_1 & a_1 \\
          {\beta}/{\lambda} & 1 & -{\beta}/{\lambda} \\
          0 & a_2 & 1
        \end{array}
     \!\! \right),
$$
where
\begin{eqnarray*}
a_1=-\frac{\mu}{2}(1-\frac{1}{\lambda}-\frac{1}{\mu}+\frac{1}{\beta}),
a_2=\frac{\lambda}{2}(1-\frac{1}{\lambda}-\frac{1}{\mu}-\frac{1}{\beta}).
\end{eqnarray*}
Its characteristic equation taking the form
\begin{eqnarray*}
\mathcal{P}_{E_3}(t)&:=-t^3+D_1 t^2+D_2 t+D_3,
\end{eqnarray*}
where the coefficients are defined by
\begin{small}
\begin{align}
D_1 = 1 + a_1, \label{E3wdx} \
D_2 = 1-2 a_{1}+\frac{\beta  \left(a_{1}-a_{2}\right)}{\lambda}, \
D_3 = -1+a_{1}-\frac{\beta  a_{2}}{\lambda}+\frac{\left(2 a_{2}-1\right) \beta  a_{1}}{\lambda}.
\end{align}
\end{small}The stability characteristics of $E_3$ are determined by the following criteria:
The equilibrium point constitutes a sink if all three conditions are simultaneously satisfied:
$E_3$ is a sink if
\begin{eqnarray*}
&&\Lambda \in \left\{(\lambda,\mu,\beta)|\mu>1,~\lambda>\mu/(\mu-1),~\beta>\lambda\mu/(\lambda\mu - \lambda - \mu),\right.\\
&&~~~~~~\left.|D_1+D_3|<1+D_2,~|D_1-3 D_3|<3-D_2,~D_3^2+D_2-D_3 D_2<1\right\},
\end{eqnarray*}
where $D_1,D_2$ and $D_3$ are as defined in \eqref{E3wdx}.
\end{rmk}

In the next sections, we continue to discuss the bifurcation phenomena of the fixed points $O, E_1$ and $E_2$ under the non-hyperbolic cases.

\section{Codimension-1 bifurcations}

Section 3 is devoted to the bifurcations of transcritical, flip and Neimark-Sacker for all fixed points.
\subsection{Transcritical bifurcation of $O$}
\allowdisplaybreaks[4]
\setcounter{equation}{0}
In this subsection, we will discuss the transcritical bifurcation of $O$ when the parameter $\Lambda$ near $\{\lambda>0,\mu=1,\beta>0\}$, i.e., case $\mathbf{L}_{1}$ in Table \ref{Table1}.
\begin{thm}
For $\Lambda\in \mathbb{R}_+^{3}$,
system \eqref{eq2.1} undergoes a transcritical bifurcation as the parameter $\Lambda$ crosses
$\mu=1$.
In addition, if $\mu>1$, $O$ is unstable and $E_1$ is stable, and $O$, $E_1$ are coincident if $\mu=1$. For $\mu<1$ the origin $O$ still exists and remains stable but $E_1$ disappears.
\label{th3.1}
\end{thm}
\begin{proof}
Let $\theta:=\mu-1$. Expanding system \eqref{eq2.1} in Taylor series, we obtain the following mapping $\mathfrak{F}:\mathbb{R}^3\rightarrow\mathbb{R}^3$, defined by
\begin{eqnarray}
\mathfrak{F}:
\left[
\begin{array}{cc}
   x \\
   y \\
   z  \\
\end{array}
\right]
\mapsto
\left[\begin{array}{cc}
  \left(1+\theta \right) x-\left(1+\theta \right) x^{2}-\left(1+\theta \right) x y-\left(1+\theta \right) z x\\
 \beta  y x-\beta  y z\\
 \lambda  y z\\
 \end{array}\right].
\label{eq3.1}
\end{eqnarray}
The eigenvectors of the Jacobian matrix $J\mathfrak{F}(O)$ are
$$(0,0,1)^T,~~(0,1,0)^T~~\mbox{and}~~(1,0,0)^T,$$
and their corresponding eigenvalues are
$$0,~~0~~\mbox{and}~~1+\theta.$$
Applying the invertible transformation
\begin{eqnarray*}
\left[
\begin{array}{cc}
   x \\
   y \\
   z
\end{array}
\right]
=\left[
   \begin{array}{ccc}
     0 & 0 & 1 \\
     0 & 1 & 0 \\
     1 & 0 & 0 \\
   \end{array}
 \right]
\left[\begin{array}{cc}
  X \\
  Y \\
  Z
 \end{array}\right],
\label{eq3.2}
\end{eqnarray*}
we translate the linear part of system \eqref{eq3.1} into the normalized form, i.e.,
\begin{eqnarray}
\left[
\begin{array}{ccc}
   X \\
   Y \\
   Z
\end{array}
\right]
\mapsto
\left[
\begin{array}{ccc}
\lambda  Y X\\
\beta  Y Z-\beta  Y X\\
\left(1+\theta \right) Z-\left(1+\theta \right) Z^{2}-\left(1+\theta \right) Z Y-\left(1+\theta \right) X Z
\end{array}\right].
\label{eq3.2}
\end{eqnarray}
Taking $\theta$ as the bifurcation parameter, system \eqref{eq3.2} is transformed into the following form
\begin{eqnarray}
\left[
\begin{array}{ccc}
   X \\
   Y \\
   Z\\
   \theta
\end{array}
\right]
\mapsto
\left[
\begin{array}{ccc}
\lambda  Y X\\
\beta  Y Z-\beta  Y X\\
\left(1+\theta \right) Z-\left(1+\theta \right) Z^{2}-\left(1+\theta \right) Z Y-\left(1+\theta \right) X Z\\
\theta\\
\end{array}\right].
\label{eq3.3}
\end{eqnarray}

Since system \eqref{eq3.3} has two eigenvalues on the unit circle $S^1$, according to the center manifold theorem (\cite[pp.33-35]{Carr}) it possesses a two-dimensional $C^2$ center manifold, i.e.,
\begin{eqnarray}
\begin{array}{l}
W_1^c(O)=\{(X,Y,Z,\theta)\in \mathbb{R}^4:X=h_{11}(Z,\theta), Y=h_{12}(Z,\theta),|Z|<\zeta_1, |\theta|<\zeta_2\},
\end{array}
\label{mf1}
\end{eqnarray}
where $h_{11}$ and $h_{12}$ are $C^2$ functions near $(0,0)$ such that
$$
h_{11}(0,0)=0,~~ Dh_{11}(0,0)=0,~~ h_{12}(0,0)=0,~~Dh_{12}(0,0)=0,
$$
and $\zeta_1$, $\zeta_2$ are both sufficiently small positive constants.
Due to the $C^2$ smoothness, functions $h_{11}(Z,\theta)$ and $h_{12}(Z,\theta)$ have the following form
\begin{eqnarray}
\begin{array}{c}
y=h_{11}(Z,\theta)=d_{11}Z^2+d_{12}Z\,\theta+d_{13}\theta^2+O(|(Z,\theta)|^3),\\
z=h_{12}(Z,\theta)=d_{21}Z^2+d_{22}Z\,\theta+d_{23}\theta^2+O(|(Z,\theta)|^3),
\end{array}
\label{mf2}
\end{eqnarray}
where $d_{11}$, $d_{12}$, $d_{13}$, $d_{21}$, $d_{22}$ and $d_{23}$ are indeterminate.
Making use of the invariance of center manifold and (\ref{mf1})-(\ref{mf2}), we have
\begin{eqnarray}
\begin{array}{c}
h_{11}(\left(1+\theta \right) Z-\left(1+\theta \right) Z^{2}-\left(1+\theta \right) Z h_{12}(Z,\theta)-\left(1+\theta \right) h_{11}(Z,\theta) Z,\theta)\\
~~~~~~~~~~~~=\lambda  h_{11}(Z,\theta) h_{12}(Z,\theta)
\end{array}
\label{coe11}
\end{eqnarray}
and
\begin{eqnarray}
\begin{array}{c}
h_{12}(\left(1+\theta \right) Z-\left(1+\theta \right) Z^{2}-\left(1+\theta \right) Z h_{12}(Z,\theta)-\left(1+\theta \right) h_{11}(Z,\theta) Z,\theta)\\
~~~~~~~~~~~~=\beta  h_{12}(Z,\theta) Z-\beta  h_{11}(Z,\theta) h_{12}(Z,\theta).
\end{array}
\label{coe12}
\end{eqnarray}
It follows that
$$d_{11}=d_{21}=d_{12}=d_{13}=d_{22}=d_{23}=0$$
by comparing the coefficients of $Z^2$, $Z\theta$ and $\theta^2$ in \eqref{coe11} and \eqref{coe12}. Using (\ref{mf2}) again we obtain
\begin{eqnarray}
\begin{array}{c}
X=h_{11}(Z,\theta)=O(|(Z,\theta)|^3),\\
Y=h_{12}(Z,\theta)=O(|(Z,\theta)|^3).
\end{array}
\label{coe2-1}
\end{eqnarray}

Next, we substitute (\ref{coe2-1}) into the last two one-dimensional mappings of (\ref{eq3.3}) and get
\begin{eqnarray*}
\left[
\begin{array}{cc}
   Z \\
   \theta
\end{array}
\right]
\mapsto
\left[\begin{array}{cc}
Z+Z \theta-Z^{2} \theta -Z^{2}+O(|(Z,\theta)|^3)\\
\theta
 \end{array}\right],
\end{eqnarray*}
which defines a one-dimensional mapping
$$
Z\mapsto g_{1}(Z):= Z+Z \theta-Z^{2} \theta -Z^{2}+O(|(Z,\theta)|^3).
$$
One can verify that
$$
\left.\frac{\partial^2g_{1}}{\partial Z^2}\right|_{(Z,\theta)=(0,0)}=-2\neq0
$$
and
$$
\left.\frac{\partial^2 g_{1}}{\partial Z \partial \theta}\right|_{(Z,\theta)=(0,0)}=1\neq0,
$$
indicating that the non-degeneracy and transversality conditions (\cite[pp.504-508]{Wiggins}) for a transcritical bifurcation are satisfied.
Consequently, system \eqref{eq2.1} produces a transcritical bifurcation as the parameter $\Lambda$ crossing $\{\lambda>0,\mu=1,\beta>0\}$.
This completes the proof.
\end{proof}
Based on the proof of Theorem \ref{th3.1}, we find the occurrence of transcritical bifurcation in system \eqref{eq2.1}, where the stability of the fixed point O is altered, and the other fixed point $E_1$ appears and disappears in both side of $\delta=0$. This serves as a threshold of the original system \eqref{eq2.1}, i.e., $\{(\lambda,\mu,\beta)\in \mathbb{R}_+^{3}:\mu=1\}$ defines a boundary between stable and unstable states.
The threshold provides a way of regulating and controlling the population densities of prey, predator and top predator:

({\bf i}) In the case that $\lambda>0$, $\mu<1$ and $\beta>0$, the population density of prey will asymptotically converge to $0$, and the population densities of the predator and top predator will also gradually approach $0$. This indicates that the prey, predator and top predator will eventually become extinct.

({\bf ii}) For $\lambda>0$, $\mu>1$ and $\beta>0$, if the initial population density of prey is close to $0$, and the initial population densities of the predator and top predator are also near $0$, then the population density of prey will rise to ${(\mu-1)}/{\mu}$ and remain stable. Meanwhile, the population densities of the predator and top predator will gradually approach $0$.

\subsection{Transcritical and flip bifurcations of $E_1$}
\allowdisplaybreaks[4]
\setcounter{equation}{0}
In this subsection, we will discuss the bifurcations that occur near the fixed point $E_1:((\mu - 1)/\mu,0,0)$. We first study the transcritical bifurcation of $E_1$ as the parameter $\Lambda$ approaches $\{\lambda>0,\mu>1,\beta={\mu}/{(\mu-1)}\}$, see cases $\mathcal {L}_{1}$ and $\mathcal {L}_{3}$ in Table \ref{Table2}.

\begin{thm}
For $\Lambda\in \mathbb{R}_+^{3}$,
system \eqref{eq2.1} undergoes a transcritical bifurcation as the parameter $\Lambda$ crosses $\beta=\mu/(\mu - 1)$.
Further, if $\beta>\mu/(\mu - 1)$, $E_1$ is unstable and $E_2$ is stable, and $E_1$, $E_2$ are coincident if $\beta=\mu/(\mu - 1)$. For $\beta<\mu/(\mu - 1)$ the fixed point $E_2$ disappears, and there exists a region $\{\lambda>0,1<\mu<3,0<\beta<\mu/(\mu - 1)\}$ such that $E_1$ remains stable.
\label{th4.1}
\end{thm}

\begin{proof}
Let $\theta_1:=\beta-\mu/(\mu - 1)$.  As done in the proof of Theorem \ref{th3.1}, we translate $E_{1}$ to the origin $O$ and restrict the mapping $F$ to a two dimensional $C^2$ center manifold
\begin{eqnarray}
\begin{array}{c}
~~~~~~~~~~~~~~~~~~~u_1=h_{21}(v_1,\theta_1)={-\frac{\mu}{\mu^{2}-2 \mu +1}\,v_1^2}+O(|(v_1,\theta_1)|^3),\\
w_1=h_{22}(v_1,\theta_1)=O(|(v_1,\theta_1)|^3),
\end{array}
\end{eqnarray}
which defines a one-dimensional mapping
$$
v_1\mapsto g_{2}(v_1):= v_1-\frac{v_1^{2} \mu}{\mu -1}+\frac{\theta_1  v_1 \left(\mu -1\right)}{\mu}+O(|(v_1,\theta_1)|^3).
$$

One can check that
$$
\left.\frac{\partial^2g_{2}}{\partial v_1^2}\right|_{(v_1,\theta_1)=(0,0)}=-\frac{2 \mu}{\mu -1}\neq0
$$
and
$$
\left.\frac{\partial^2 g_{2}}{\partial v_1 \partial \theta_1}\right|_{(v_1,\theta_1)=(0,0)}=\frac{\mu -1}{\mu}\neq0,
$$
imply that the non-degeneracy and transversality conditions
for a transcritical bifurcation are fulfilled.
Consequently, system \eqref{eq2.1} undergoes a transcritical bifurcation as the parameter $\Lambda$ crosses $\{\lambda>0,\mu>1,\beta={\mu}/{(\mu-1)}\}$.
This completes the proof.
\end{proof}
\iffalse
Based on the proof of Theorem \ref{th4.1}, we see that the transcritical bifurcation occurred in the system \eqref{eq2.1}, where the fixed point $E_1$ exchanges its stability, while the other fixed point $E_2$ appears and disappears in both side of $\delta=0$,  which provides for the original system \eqref{eq2.1} a threshold, i.e., $\{(\mu, \beta, \lambda)\in \mathbb{R}_+^{3}:\beta={\mu}/{(\mu-1)}\}$ of a transit between stable and unstable.
This threshold gives a strategy for controlling the population densities of prey, predator, and top predator:

({\bf i}) when $\mu>0$, $\lambda>0$, and $\beta<(\mu - 1)/\mu$, the prey population density will converge to $(\mu - 1)/\mu$, while the predator and top predator populations will approach $0$.

({\bf ii}) In the case that $\mu>0$, $\lambda>0$, and $\beta>(\mu - 1)/\mu$, if the initial prey population density is close to $(\mu - 1)/\mu$, the initial predator and top predator population densities are near $0$, then the prey population density will increase to $1/\beta$ and remain stable; Meanwhile, the predator population density will rise to $(\beta  \mu -\beta -\mu)/(\beta  \mu)$; the top predator population density will gradually approach $0$.
\fi
Similarly, a flip bifurcation may occur in a neighborhood of $E_1$. In what follows, we focus on the cases $\mathcal {D}_{42}$ and $\mathcal {D}_{45}$ in Table \ref{Table2}.

\begin{thm}
As the parameter $\Lambda\in \mathbb{R}^{3}_{+}$ crosses $\mu=3$, system \eqref{eq2.1} undergoes a flip bifurcation near the fixed point $E_1$. Specifically, system \eqref{eq2.1} undergoes a supercritical flip bifurcation, and generates a stable period-two cycle from the region $\mathcal {D}_{1}$ to $\mathcal {D}_{32}$ and from the region $\mathcal {D}_{31}$ to $\mathcal {D}_{2}$ respectively.
\label{flipE1}
\end{thm}
\begin{proof}
Let $\varpi:=\mu-3$. Similar to the proof of Theorem \ref{th3.1}, we translate $E_{1}$ to the origin $O$ and restrict the mapping $F$ to a two dimensional $C^2$ center manifold
\begin{eqnarray*}
\begin{array}{c}
v_2=h_{31}(u_2,\varpi)=O(|(u_2,\varpi)|^3),\\
w_2=h_{32}(u_2,\varpi)=O(|(u_2,\varpi)|^3)
\end{array}
\end{eqnarray*}
that defines a one-dimensional mapping
\begin{eqnarray}
\begin{array}{c}
u_2\mapsto g_{3}(u_2):=-u_2-\varpi\,u_2-3 u^2_2-\varpi\,u^2_2+O(|(u_2,\varpi)|^4).
\end{array}
\label{qu}
\end{eqnarray}

One can calculate that
\begin{eqnarray*}
\left.\left(\frac{\partial g_{3}}{\partial \varpi}\frac{\partial^2g_{3}}{\partial u^2_2}+2\,\frac{\partial^2g_{3}}{\partial u_2 \partial\varpi}\right)\right|_{(u_2,\varpi)=(0,0)}=-2<0
\end{eqnarray*}
and
\begin{eqnarray}
\begin{array}{c}
\left.\left(\frac{1}{2}\,\left(\frac{\partial^2g_{3}}{\partial u_2^2}\right)^2+\frac{1}{3}\,\frac{\partial^3g_{3}}{\partial u_2 ^3}\right)\right|_{(u_2,\varpi)=(0,0)}=18>0.
 \end{array}
\label{coe4}
\end{eqnarray}
Hence, the conditions (F1) and (F2) of Theorem 3.5.1 in \cite[p.158]{Guckenheimer}
are fulfilled, implying that when the parameter $\Lambda$ crosses $\{\lambda>0,\beta>1,\mu=3\}$, system \eqref{eq2.1} undergoes a flip bifurcation near the fixed point $E_1$.

Note that system \eqref{qu} is topologically equivalent to the subsequent mapping
\begin{eqnarray}
\begin{array}{l}
u_2\mapsto g^*_{3}(u_2):=-u_2-\varpi\,u_2-3 u^2_2-\varpi\,u^2_2,
\label{quyu1}
\end{array}
\end{eqnarray}
we re-expand $g^{\ast}_{3}$ in Taylor series with $u_2$ and then obtain
%we abandon the form in \eqref{eq3.05251}, and perform a one-dimensional Taylor expansion of $g_2(u_1)$ with $u_1$ as a variable, then the following %mapping can be obtained
\begin{eqnarray}
g^{\ast}_{3}(u_2)=f_1(\varpi)\,u_2+f_2(\varpi)\,u_2^2+f_3(\varpi)\,u_2^3+O(u_2^4),
\label{quyu2}
\end{eqnarray}
where $f_1,f_2$ and $f_3$ are $C^1$ functions, satisfying
\begin{eqnarray*}
&&f_1(\varpi)=-(1+g(\varpi)),~~g(\varpi)=\varpi,~~f_2(\varpi)=-3-\varpi,~~f_3(\varpi)=0.
\end{eqnarray*}
Since $g(0)=0$ and
$$
\left.\frac{\mathrm{d}g}{\mathrm{d}\varpi}\right|_{\varpi=0}=-\left.\frac{\partial^2g_3}{\partial u_2\partial\varpi}\right|_{(u_2, \varpi)=(0,0)}=1\neq0,
$$
it follows that the function $g$ is locally invertible. Thus, mapping \eqref{quyu2} can be  transformed into
$$
\tilde{w}_2=\kappa(\varpi_1)\,u_2+c(\varpi_1)\,u_2^2+d(\varpi_1)\,u_2^3+O(u_2^4)
$$
by setting
\begin{eqnarray*}
\varpi_1:=g(\varpi)=\varpi,
\end{eqnarray*}
where the functions $\kappa(\varpi_1)=-(1+\varpi_1)$, $c(\varpi_1)$ and $d(\varpi_1)$ are all $C^1$.
Consequently, we get
\begin{eqnarray*}
c(0)=f_2(0)=\frac{1}{2}\,\frac{\partial^2g_{3}}{\partial u_2^2}, ~~~d(0)=\frac{1}{6}\frac{\partial^3g_{3}}{\partial u_2 ^3}.
\label{eq05162}
\end{eqnarray*}

Define a smooth transformation of coordinate by
\begin{eqnarray}
u_2=u_3+\varpi_2u_3^2,
\label{quyu3}
\end{eqnarray}
where $\varpi_2=\varpi_2(\varpi_1)$ is a $C^1$ function to be determined. Clearly, the transformation \eqref{quyu3} is invertible in a sufficiently small neighborhood of the origin $O$, and its inverse can be derived by using the method of undetermined coefficients as
\begin{eqnarray}
u_3=u_2-\varpi_2u_2^2+2\varpi_2^2u_2^3+O(u_2^4).
\label{quyu4}
\end{eqnarray}
According to \eqref{quyu3}-\eqref{quyu4}, we get
\begin{small}
\begin{eqnarray}\label{karpa}
\!\!\!\!\!\tilde{u}_3=\kappa u_3+(c+\varpi_2\kappa-\varpi_2\kappa^2)u_3^2+(d+2\varpi_2c-2\varpi_2\kappa(\varpi_2\kappa+c)+2\varpi_2^2\kappa^3)u_3^3
+O(u_3^4).
\end{eqnarray}
\end{small}Hence, the quadratic term of (\ref{karpa}) can be eliminated for sufficiently small $|\varpi_1|$ by setting
$$
\varpi_2(\varpi_1)=\frac{c(\varpi_1)}{\kappa^2(\varpi_1)-\kappa(\varpi_1)}
$$
because $\kappa^2(0)-\kappa(0)=2\neq0$. Thus, (\ref{karpa}) is transformed into
\begin{eqnarray*}
\tilde{u}_3=\kappa u_3+(d+\frac{2c^2}{\kappa^2-\kappa})u_3^3+O(u_3^4)=-(1+\varpi_1)u_3+e(\varpi_1)u_3^3+O(u_3^4)
\end{eqnarray*}
for certain $C^1$ function $e$ such that
$$
e(0)=c^2(0)+d(0)=\frac{1}{4}\,\left.\left(\frac{\partial^2g_{3}}{\partial u_2^2}\right)^2\right|_{(u_2,\varpi)=(0,0)}+\frac{1}{6}\,\left({\left.\frac{\partial^3g_{3}}{\partial u_2 ^3}\right)}\right|_{(u_2,\varpi)=(0,0)}.
$$

It is worth mentioning that $e(0)>0$ by \eqref{coe4}. Rescaling $u_3=u_4/\sqrt{|e(\varpi_1)|}$ we obtain
\begin{eqnarray}
u_4\mapsto g_{3\tau}(u_4):=-(1+\varpi_1)u_4+u_4^3+O(u_4^4).
\label{quyu5}
\end{eqnarray}
Based on the proof of \cite[Theorem~4.4, p.129]{Kuznetsov}, mapping \eqref{quyu5} is topologically equivalent to the following form
\begin{eqnarray}
u_4\mapsto g_{3*}(u_4)=-(1+\varpi_1)u_4+u_4^3
\label{quyu6}
\end{eqnarray}
near the origin $O$.
Finally, we compute the second iterate of mapping \eqref{quyu6} and get
\begin{eqnarray*}
\!\!\!\!\!\!\!\!\!\!g_{3*}^2(u_4)=\left( -\varpi_1-1 \right) ^{2}u_4+ \left(  \left( -\varpi_1-1 \right)\left( 1+\varpi_1 \right) ^{2}-\varpi_1-1 \right) u_4^{3}+O(u_4^5).
\end{eqnarray*}
%where $g_{2*}^2(u_4)$ is the second iterate for $g_{2*}(u_4)$.
It is straightforward to verify that $g_{3*}^2$ has a trivial fixed point $u_{40}=0$ and two nontrivial fixed points
$$
u_{41,42}=\pm\sqrt{\varpi_1}
$$
for $0<\varpi_1\ll1$.
%Then from the above equation we can get $u_{41,42}=\pm\sqrt{\varpi_1}$.
One can check that $|Dg_{3*}^2(u_{41})|<1$ (resp. $|Dg_{3*}^2(u_{42})|<1$).
So $u_{41}$ and $u_{42}$  are stable and constitute a stable period-two cycle for the original mapping $g_{3*}$.
This indicates that the mapping \eqref{quyu1} produces a stable period-two cycle near $E_1$ for $0<\varpi_1\ll1$. Besides, system \eqref{eq2.1} generates a stable period-two cycle near the fixed point $E_1$ for $0<\varpi_1\ll1$.
Therefore, according to \eqref{quyu6} system \eqref{eq2.1} undergoes a supercritical flip bifurcation, and generates a stable period-two cycle from the region $\mathcal {D}_{1}$ to $\mathcal {D}_{32}$ and from the region $\mathcal {D}_{31}$ to $\mathcal {D}_{2}$ respectively.
This completes the proof.
\end{proof}

%From Theorem \ref{flipE1}, we see that system \eqref{eq2.1} undergoes a flip bifurcation and produces a period-two orbit. Additionally, as the bifurcation parameter $\mu$ varies, there
%will form a four-period orbit, eight-period orbit, and so on, culminating in chaotic behavior. Such a case would result in an unregulated population density of prey, predator, and top predator, a condition that should be avoided.

\iffalse
Moreover, if $\Theta_1=0$, then mapping \eqref{eq2.1} may produce a codimension 2 bifurcation, i.e., generalized flip bifurcation. Due to the huge amount of calculation, we cannot give a detailed proof. We will look for a new proof methods in the next work.
\fi

\subsection{Transcritical and flip bifurcations of $E_2$}
\allowdisplaybreaks[4]
\setcounter{equation}{0}
In this subsection, we continue to discuss the bifurcations that occur near the fixed point $E_2:(1/\beta,(\beta  \mu -\beta -\mu)/\beta  \mu,0)$. We first consider the transcritical bifurcation near $E_2$ when the parameter $\Lambda$ is close to $\{\lambda>\mu/(\mu - 1), \mu>1, \beta = \lambda\mu/(\lambda\mu - \lambda - \mu)\}$, i.e., case $\mathfrak{L}_{11}$ in Table \ref{Table3}.
\begin{thm}
For parameter $\Lambda\in \mathbb{R}_+^{3}$,
system \eqref{eq2.1} undergos a transcritical bifurcation as $\Lambda$ crosses
$\beta = \lambda\mu/(\lambda\mu - \lambda - \mu)$.
Concretely, if $\beta>\lambda\mu/(\lambda\mu - \lambda - \mu)$, $E_2$ is unstable, and $E_2$, $E_3$ are coincident if $\beta=\lambda\mu/(\lambda\mu - \lambda - \mu)$. When $\beta<\lambda\mu/(\lambda\mu - \lambda - \mu)$,  $E_2$ exists and becomes stable, but $E_3$ disappears.
\label{th5.1}
\end{thm}
\begin{proof}
Let $\theta_2:=\beta-\lambda\mu/(\lambda\mu - \lambda - \mu)$. As done in the proof of Theorem \ref{th3.1}, by translating $E_{2}$ to the origin $O$ and restricting mapping $F$ into a two dimensional $C^2$ center manifold
\begin{eqnarray*}
&&u_5=h_{41}(w_5,\theta_2)=\frac{\left(\lambda -1\right)^{4} \mu^{5}-4 \left(\lambda +\frac{1}{8}\right) \left(\lambda -1\right)^{2} \lambda  \,\mu^{4}}{B}\,w_5^2\\
&&~~~~~~~~~~~~~~~~~~~~~+\frac{\left(6 \lambda^{4}-3 \lambda^{3}+\left(A-1\right) \lambda^{2}-3 A \lambda +2 A\right) \mu^{3}}{B}\,w_5^2\\
&&~~~~~~~~~~~~~~~~~~~~~+\frac{\left(-4 \lambda^{4}-3 A \lambda^{2}-\frac{1}{2} \lambda^{3}+\frac{7}{2} A \lambda -\frac{1}{2} A\right) \mu^{2}}{B}\,w_5^2\\
&&~~~~~~~~~~~~~~~~~~~~~+\frac{\left(\lambda^{4}+3 A \lambda^{2}-\frac{1}{2} A \lambda \right) \mu -A \lambda^{2}}{B}\,w_5^2+O(|(w_5,\theta_2)|^3),\\
&&v_5=h_{42}(w_5,\theta_2)=\frac{\left(\lambda -1\right)^{4} \mu^{5}-4 \left(\lambda +\frac{1}{8}\right) \left(\lambda -1\right)^{2} \lambda  \,\mu^{4}}{B}\,w_5^2\\
&&~~~~~~~~~~~~~~~~~~~~~+\frac{\left(6 \lambda^{4}-3 \lambda^{3}+\left(-A-1\right) \lambda^{2}+3 A \lambda -2 A\right) \mu^{3}}{B}\,w_5^2\\
&&~~~~~~~~~~~~~~~~~~~~~+\frac{\left(-4 \lambda^{4}+3 A \lambda^{2}-\frac{1}{2} \lambda^{3}-\frac{7}{2} A \lambda +\frac{1}{2} A\right) \mu^{2}}{B}\,w_5^2\\
&&~~~~~~~~~~~~~~~~~~~~~+\frac{\left(\lambda^{4}-3 A \lambda^{2}+\frac{1}{2} A \lambda \right) \mu +A \lambda^{2}}{B}\,w_5^2+O(|(w_5,\theta_2)|^3),
\label{coe2}
\end{eqnarray*}
where
\begin{eqnarray*}
&&A:=\lambda\,\sqrt{\mu^{4} \lambda^{2}-2 \mu^{3} \lambda^{2}-2 \lambda  \,\mu^{4}+\lambda^{2} \mu^{2}-2 \mu^{3} \lambda +\mu^{4}},\\
&&B:=\frac{2 \lambda}{\mu^{2} \left(\lambda  \mu -\lambda -\mu \right) \left(\lambda^{2} \mu^{2}-2 \lambda^{2} \mu -2 \lambda  \,\mu^{2}+\lambda^{2}-2 \lambda  \mu +\mu^{2}\right)},
\end{eqnarray*}
we obtain a one-dimensional mapping
$$
w_5\mapsto g_{4}(w_5):=w_5-2 w_5^{2} \lambda +\frac{\left(\lambda^{2} \mu^{2}-2 \lambda^{2} \mu -2 \lambda  \,\mu^{2}+\lambda^{2}+2 \lambda  \mu +\mu^{2}\right) \theta_2  w_5}{\mu^{2} \lambda}+O(|(w_5,\theta_2)|^3).
$$

One can check that
$$
\left.\frac{\partial^2g_{4}}{\partial w_5^2}\right|_{(w_5,\theta_2)=(0,0)}=-4\lambda\neq0
$$
and
$$
\left.\frac{\partial^2 g_{4}}{\partial w_5 \partial \theta_2}\right|_{(w_5,\theta_2)=(0,0)}=\frac{\left(\left(\mu -1\right) \lambda -\mu \right)^{2}}{\mu^{2} \lambda}\neq0,
$$
which indicate that the non-degeneracy and transversality conditions of a
transcritical bifurcation are satisfied. Consequently, system (2.1) undergoes a transcritical bifurcation as the parameter $\Lambda$ crosses $\{\lambda>\mu/(\mu - 1), \mu>1, \beta = \lambda\mu/(\lambda\mu - \lambda - \mu)\}$. This completes the proof.
\end{proof}
\iffalse
Based on the proof of Theorem \ref{th5.1}, we see that the transcritical bifurcation occurred in system \eqref{eq2.1}, where the fixed point $E_2$ exchanges its stability, while the other fixed point $E_3$ appears and disappears in both sides of $\delta=0$. This establishes a threshold for the original system \eqref{eq2.1}, i.e., $\{(\mu, \beta, \lambda)\in \mathbb{R}_+^{3}:\beta=\lambda\mu/(\lambda\mu - \lambda - \mu)\}$  of a transit between stable and unstable.
This threshold serves as a strategy for controlling the population densities of prey, predator, and top predator:

({\bf i}) if $\mu>0$, $\lambda>0$, and $\beta<\lambda\mu/(\lambda\mu - \lambda - \mu)$, the population density of prey will converge to $1/\beta$, the population density of  predator will tend to $(\beta\mu-\beta-\mu)/(\beta\mu)$, and the population density of  top predator will approach $0$.

({\bf ii}) when $\mu>0$, $\lambda>0$, and $\beta>\lambda\mu/(\lambda\mu - \lambda - \mu)$, if the initial population density of prey is near $1/\beta$, the initial population density of  predator is close to $(\beta\mu-\beta-\mu)/(\beta\mu)$, and the initial population density top predator is close to $0$, then the population density of prey will increase to $(\beta  \lambda  \mu -\beta  \lambda -\beta  \mu +\lambda  \mu)/(2 \beta  \lambda  \mu)$ and remain stable; the population density of predator will decrease to $1/\lambda$; the population density of top predator will rise to $(\beta  \lambda  \mu -\beta  \lambda -\beta  \mu -\lambda  \mu)/(2 \beta  \lambda  \mu)$, respcetively.
\fi
Similarly, a flip bifurcation may occur near $E_2$. We focus on the cases $\mathfrak{D}_{33}$ and $\mathfrak{D}_{34}$ in Table \ref{Table3}.

\begin{thm}
When the parameter $\Lambda\in \mathbb{R}_+^{3}$ crosses $\mu=-3 \beta/(\beta-3)$ and $\beta=2$, system \eqref{eq2.1} undergoes a flip bifurcation near the fixed point $E_2$. Specifically, system \eqref{eq2.1} undergoes a subcritical flip bifurcation, and generates an unstable period-two cycle from the region $\mathfrak{D}_{33}$ to $\mathfrak{D}_{34}$.
\label{flipE2}
\end{thm}
\begin{proof}
Let $\theta_3:=\mu+3 \beta/(\beta-3)$. Similar to the proof of Theorem \ref{th3.1}, we translate $E_{2}$ to the origin $O$ and restrict mapping $F$ into a two dimensional $C^2$ center manifold
\begin{eqnarray*}
&&u_6=h_{41}(v_6,\theta_3)=O(|(v_6,\theta_3)|^3),\\
&&w_6=h_{42}(v_6,\theta_3)=\frac{21 (\beta^{3}-\frac{117}{14} \beta^{2}+\frac{621}{28} \beta -\frac{135}{7}) \beta}{40 (\beta -\frac{3}{2})^{2} (\beta^{2}-\frac{21}{5} \beta +\frac{9}{2})}\,v_6^2\\
&&~~~~~~~~~~~~~~~~~~~~+\frac{3 (-128 \beta^{4}+1260 \beta^{3}-4698 \beta^{2}+7803 \beta -4860) \sqrt{\beta^{2} (-9+4 \beta)^{2}}}{2560 (\beta -\frac{3}{2})^{2} (\beta^{2}-\frac{21}{5} \beta +\frac{9}{2}) (\beta -\frac{9}{4})^{2}}\,v_6^2\\
&&~~~~~~~~~~~~~~~~~~~~+O(|(v_6,\theta_3)|^3),
\end{eqnarray*}
and then obtain a one-dimensional mapping
\begin{small}
\begin{eqnarray*}
\!\!\!\!\!\!\!\!\!&&v_6\mapsto g_{5}(v_6):=\frac{v_6 (16 \beta^{3}+2 \sqrt{\beta^{2} (-9+4 \beta)^{2}}\, \beta -72 \beta^{2}-3 \sqrt{\beta^{2} (-9+4 \beta)^{2}}+81 \beta)}{\sqrt{\beta^{2} (-9+4 \beta)^{2}}\, (2 \beta -6)}\\
\!\!\!\!\!\!\!\!\!&&~~~~~~~~~~~~~~~~~+\frac{6 v_6^{2} (4 \sqrt{\beta^{2} (-9+4 \beta)^{2}}\, \beta^{2}-18 \sqrt{\beta^{2} (-9+4 \beta)^{2}}\, \beta) \beta}{8 \sqrt{\beta^{2} (-9+4 \beta)^{2}}\, (\beta -3)^{2} (2\beta -3)}\\
\!\!\!\!\!\!\!\!\!&&~~~~~~~~~~~~~~~~~+\frac{6 v_6^{2} (8 \beta^{4}-24 \beta^{3} -18 \beta^{2}+27 \sqrt{\beta^{2} (-9+4 \beta)^{2}}+81 \beta) \beta}{8 \sqrt{\beta^{2} (-9+4 \beta)^{2}}\, (\beta -3)^{2} (2\beta -3)}\\
\!\!\!\!\!\!\!\!\!&&~~~~~~~~~~~~~~~~~-\frac{128 \theta_3  \,\beta^{3} (\beta -3)^{2} (2 \beta^{3}-8 \beta^{2}+\sqrt{\beta^{2} (-9+4 \beta)^{2}}+9 \beta) v_6 (2\beta -3)^{2}}{2\sqrt{\beta^{2} (-9+4 \beta)^{2}}\, (\sqrt{\beta^{2} (-9+4 \beta)^{2}}+3 \beta)^{2} (\sqrt{\beta^{2} (-9+4 \beta)^{2}}-3 \beta)^{2}}\\
\!\!\!\!\!\!\!\!\!&&~~~~~~~~~~~~~~~~~+\frac{27 \beta^{2} v_6^{3} (184 \beta^{4}-1848 \beta^{3}+6975 \beta^{2}-11664 \beta +7290)}{4 (-9+4 \beta)^{2} (\beta -3) (2 \beta -3)^{2} (10 \beta^{2}-42 \beta +45)}\\
\!\!\!\!\!\!\!\!\!&&~~~~~~~~~~~~~~~~~-\frac{27 \beta^{3} v_6^{3} (4 \beta -15) (11 \beta^{2}-48 \beta +54)}{4 \sqrt{\beta^{2} (-9+4 \beta)^{2}}\, (\beta -3) (2 \beta -3)^{2} (10 \beta^{2}-42 \beta +45)}\\
\!\!\!\!\!\!\!\!\!&&~~~~~~~~~~~~~~~~~-\frac{64 (16 \beta^{3}-96 \beta^{2}+261 \beta -243) \theta_3  \,\beta^{6} (\beta -3)^{2} v_6^{2} (2 \beta -3)}{(\sqrt{\beta^{2} (-9+4 \beta)^{2}}+3 \beta)^{3} (\sqrt{\beta^{2} (-9+4 \beta)^{2}}-3 \beta)^{3}}\\
\!\!\!\!\!\!\!\!\!&&~~~~~~~~~~~~~~~~~-\frac{1024 \beta^{11} (20 \beta^{2}-42 \beta -927) \delta  (\beta -3)^{2} v_6^{2} (2 \beta -3)}{(-9+4 \beta)^{2} \sqrt{\beta^{2} (-9+4 \beta)^{2}}\, (\sqrt{\beta^{2} (-9+4 \beta)^{2}}+3 \beta)^{3} (\sqrt{\beta^{2} (-9+4 \beta)^{2}}-3 \beta)^{3}}\\
\!\!\!\!\!\!\!\!\!&&~~~~~~~~~~~~~~~~~-\frac{1728 \beta^{7} (3584 \beta^{3}-9180 \beta^{2}+10989 \beta -5103) \delta  (\beta -3)^{2} v_6^{2} (2 \beta -3)}{(-9+4 \beta)^{2} \sqrt{\beta^{2} (-9+4 \beta)^{2}}\, (\sqrt{\beta^{2} (-9+4 \beta)^{2}}+3 \beta)^{3} (\sqrt{\beta^{2} (-9+4 \beta)^{2}}-3 \beta)^{3}}\\
\!\!\!\!\!\!\!\!\!&&~~~~~~~~~~~~~~~~~+\frac{4096 \theta_3^{2} \beta^{6} (\beta -3)^{6} v_6 (2\beta -3)^{3} (\beta -2)}{2\sqrt{\beta^{2} (-9+4 \beta)^{2}}\, (-\sqrt{\beta^{2} (-9+4 \beta)^{2}}+3 \beta)^{3} (\sqrt{\beta^{2} (-9+4 \beta)^{2}}+3 \beta)^{3} (-9+4 \beta)^{2}}\\
\!\!\!\!\!\!\!\!\!&&~~~~~~~~~~~~~~~~~+O(|(v_6,\theta_3)|^4).
\label{coe3}
\end{eqnarray*}
\end{small}
Let $\beta=2$, we get
$$
v_6\mapsto g_{5}(v_6):=-v_6+30 v_6^{2}-\frac{1}{2} v_6 \theta_3 -972 v_6^{3}+2 \theta_3  v_6^{2}+O(|(v_6,\theta_3)|^4).
$$
One can calculate that
\begin{eqnarray*}
\left.\left(\frac{\partial g_{5}}{\partial \theta_3}\frac{\partial^2g_{5}}{\partial u^2_2}+2\,\frac{\partial^2g_{5}}{\partial v_6 \partial\theta_3}\right)\right|_{(v_6,\theta_3)=(0,0)}=-1<0
\end{eqnarray*}
and
\begin{eqnarray*}
\left.\left(\frac{1}{2}\,\left(\frac{\partial^2g_{5}}{\partial v_6^2}\right)^2+\frac{1}{3}\,\frac{\partial^3g_{5}}{\partial v_6 ^3}\right)\right|_{(v_6,\theta_3)=(0,0)}=-144<0.
\end{eqnarray*}
Therefore, the non-degeneracy and transversality conditions of Theorem \ref{flipE2} are satisfied. Besides, similar to the proof of Theorem \ref{flipE1}, when the parameter $\Lambda$ crosses $\{\lambda>0,\mu=-3 \beta/(\beta-3),\beta=2\}$, system \eqref{eq2.1} undergoes a subcritical flip bifurcation and produces an unstable period-two orbit.
\end{proof}

\iffalse
According to Theorem \ref{flipE2}, system \eqref{eq2.1} undergoes a subcritical flip bifurcation, produces an unstable period-two orbit. Additionally, as the bifurcation parameter $\mu$ varies, it can lead to the formation of period-four orbits, period-eight orbits, and so on, culminating in chaos. This behavior can result in an uncontrolled population density of prey, predator, and top predator, which should be avoided.
\fi

\subsection{Neimark-Sacker bifurcation of $E_{2}$}
\allowdisplaybreaks[4]
\setcounter{equation}{0}
In this subsection, we consider the case that the eigenvalues of $JF(E_2)$ are a pair of conjugate complex numbers and lie on the unit disk, then discuss the Neimark-Sacker bifurcation of system \eqref{eq2.1} in a neighborhood of the fixed point $E_2$.
From \eqref{solu-1}, if $\Lambda$ belongs to case $\mathfrak{L}_{31}$ or $\mathfrak{L}_{32}$ (see Table~\ref{Table3}), i.e.,
\begin{eqnarray*}
\begin{array}{llll}
\mathfrak{L}:=\mathfrak{L}_{31}\cup \mathfrak{L}_{32}=\left\{\lambda>0, \mu=\frac{\beta}{\beta-2}, \beta>\frac{9}{4}\right\},
\end{array}
\end{eqnarray*}
the two eigenvalues of $JF(E_2)$ are
\begin{eqnarray*}
\begin{array}{llll}
\tilde{t}_{1}=\frac{2 \beta -5+{\bf{i}}\sqrt{4 \beta -9}}{2 \left(\beta -2\right)},
\tilde{t}_{2}=\frac{2 \beta -5-{\bf{i}}\sqrt{4 \beta -9}}{2 \left(\beta -2\right)},
\end{array}
\end{eqnarray*}
where ${\bf i}$ denotes the imaginary unit. As the parameter $\Lambda$ crosses $\mathfrak{L}$, system \eqref{eq2.1} may generate a Neimark-Sacker bifurcation.
In fact, we have the following result.
\begin{thm}
If the parameter $\Lambda\in \mathbb{R}_+^{3}$ crosses $\mathfrak{L}$ with
\begin{eqnarray*}
\beta\neq7/3,~\beta\neq5/2,~\beta\neq\frac{9}{2}+\frac{\sqrt{21}}{2},
\end{eqnarray*}
system \eqref{eq2.1} undergoes a Neimark-Sacker bifurcation near $E_2$. Specifically, if the first Lyapunov quantity $\mathfrak{A}>0$
(resp. $\mathfrak{A}<0$), where
\begin{eqnarray*}
\mathfrak{A}=\frac{\beta^{3} \left(\beta^{2}-9 \beta +15\right)}{16 \left(\beta -2\right)^{3} \left(4 \beta -9\right)},
\end{eqnarray*}
system \eqref{eq2.1} undergoes a subcritical (resp. supercritical) Neimark-Sacker bifurcation and generates a unique stable (resp. unstable) invariant circle surrounding $E_2$ from the region $\mathfrak{D}_{53}$ (resp. $\mathfrak{D}_{64}$) to $\mathfrak{D}_{64}$ (resp. $\mathfrak{D}_{53}$) and from the region $\mathfrak{D}_{63}$ (resp. $\mathfrak{D}_{71}$) to $\mathfrak{D}_{71}$ (resp. $\mathfrak{D}_{63}$).
\label{nsE2}
\end{thm}
\begin{proof}
In order to prove the conclusions, we need to verify the conditions (C.1), (C.2) and (C.3) stated in \cite[Theorem 4.6, p.141]{Kuznetsov}.
Denote
\begin{eqnarray*}
\mu_0:=\tilde{t}_{2}=\frac{2 \beta-5-{\bf{i}}\sqrt{4 \beta -9}}{2 \beta-4},
\end{eqnarray*}
we first check the non-degeneracy condition (C.2), i.e., $\mu_0^{i}\ne 1$ for $i=1,2,3,4$. Since $\mu_0$ is not a real number, the inequality $\mu_0^{i}\ne 1$ holds for $i=1,2$. Further, we declare that $\mu_0^{3}\neq1$. Otherwise, we get $\mu_0^2+\mu_0+1=0$, implying $$\mu_0=-\frac{1}{2}+\frac{\sqrt{3}}{2}\,{\bf{i}} ~~~\mbox{or}~~~ \mu_0=-\frac{1}{2}-\frac{\sqrt{3}}{2}\,{\bf{i}}, $$
and thus $\beta=7/3$, which contradicts to our assumption. Next, assume $\mu_0^{4}=1$, which is equivalent to $\mu_0^2+1=0$, implying the real part of $\mu_0$,
denoted as $\mathfrak{R}(\mu_0)$, is equal to $0$. Consequently, we have $\beta=5/2$, a contradiction to our assumption again. Hence, the non-degeneracy condition is satisfied.

On the other hand, if $\Lambda$ is near $\mathfrak{L}$, from \eqref{solu-1} we get
$$
|\tilde{t}_{2}|=\sqrt{\frac{\mu  \left(\beta -2\right)}{\beta}},
$$
and then
\begin{eqnarray*}
\left.\frac{d|\tilde{t}_{2}|}{d\,\mu}\right|_{\mu =\frac{\beta}{\beta-2}}=\frac{\beta -2}{2 \beta}>0,
\end{eqnarray*}
which implies that the transversality condition (C.1) holds.

Finally, we verify the non-degeneracy condition (C.3). Actually,
utilizing the transformation
$$(x,y,z)=\left(u_7+{\frac{1}{\beta}},v_7+{\frac{\beta  \mu -\beta -\mu}{\beta \mu}},w_7\right),$$
we translate $E_{2}$ to the origin $O$. Let
$$\mu =\frac{\beta}{\beta-2}$$
and employing the invertible linear transformation $(u_7,v_7,w_7)^T=\mathcal{H}(u_8,v_8,w_8)^T$, where
\begin{eqnarray*}
\mathcal{H}=\left[ \begin {array}{ccc}
1&0&-\frac{\beta  \left(2 \beta -\lambda \right)}{\beta^{3}-2 \beta^{2} \lambda +\beta  \,\lambda^{2}-2 \beta^{2}+5 \beta  \lambda -2 \lambda^{2}}\\
0&1&-\frac{\left(\beta^{2}-\beta  \lambda -4 \beta +2 \lambda\right) \beta}{\beta^{3}-2 \beta^{2} \lambda +\beta  \,\lambda^{2}-2 \beta^{2}+5 \beta  \lambda -2 \lambda^{2}}\\
0&0&\frac{\beta  \left(2 \beta -\lambda \right)}{\beta^{3}-2 \beta^{2} \lambda +\beta  \,\lambda^{2}-2 \beta^{2}+5 \beta  \lambda -2 \lambda^{2}}
\end {array}
 \right],
\end{eqnarray*}
system \eqref{eq2.1} is transformed into the following form
\begin{eqnarray}
\left[
\begin{array}{ccc}
   u_8  \\
   \noalign{\medskip}
   v_8\\
   \noalign{\medskip}
   w_8
\end{array}
\right]
\mapsto
\left[\begin{array}{cc}
\frac{u_8 \left(\beta -3\right)}{\beta -2}-\frac{v_8}{\beta -2}-\frac{u_8^{2} \beta}{\beta -2}-\frac{v_8 u_8 \beta}{\beta -2}-\frac{w_8 u_8 \left(\beta^{2}-\lambda  \beta -\beta +2 \lambda\right)}{\beta -2}\\
-\frac{w_8 v_8 \left(\lambda  \beta +\beta -2 \lambda\right)}{\beta -2}-\frac{w_8^{2} \left(\beta^{2} \lambda -\beta  \,\lambda^{2}+2 \beta^{2}-7 \lambda  \beta +3 \lambda^{2}\right)}{2 \beta -\lambda}\\
\noalign{\medskip}
  u_8+v_8+v_8 u_8 \beta +\frac{\beta  \left(\beta^{2}-\lambda  \beta -4 \beta +2 \lambda\right) u_8 w_8}{2 \beta -\lambda} -\frac{\beta  \left(\beta^{2}-\lambda  \beta -4 \beta +2 \lambda\right) v_8 w_8}{2 \beta -\lambda}\\
  -\frac{\beta  \left(\beta^{4}-2 \beta^{3} \lambda +\beta^{2} \lambda^{2}-8 \beta^{3}+12 \beta^{2} \lambda -4 \beta  \,\lambda^{2}+16 \beta^{2}-16 \lambda  \beta +4 \lambda^{2}\right) w_8^{2}}{\left(2 \beta -\lambda \right)^{2}} \\
\noalign{\medskip}
\frac{w_8 \lambda}{\beta}+v_8 w_8 \lambda +\frac{\lambda  \left(\beta^{2}-\lambda  \beta -4 \beta +2 \lambda\right) w_8^{2}}{2 \beta -\lambda}
 \end{array}\right].
\label{ns.1}
\end{eqnarray}
As done in the proof of Theorem \ref{th3.1}, mapping \eqref{ns.1} is restricted into a two dimensional $C^{2}$ center manifold
\begin{eqnarray*}
w_8=h_{5}(u_8,v_8)=O(|(u_8,v_8)|^3),
\end{eqnarray*}
and further converted into the following two dimensional form
\begin{eqnarray}
\left[
\begin{array}{ccc}
   u_8  \\
   \noalign{\medskip}
   v_8
\end{array}
\right]
\mapsto
\left[\begin{array}{cc}
\frac{u_8 \left(\beta -3\right)}{\beta -2}-\frac{v_8}{\beta -2}-\frac{u_8^{2} \beta}{\beta -2}-\frac{v_8 u_8 \beta}{\beta -2}\\
\noalign{\medskip}
  u_8+v_8+v_8 u_8 \beta
 \end{array}\right].
\label{ns.2}
\end{eqnarray}

Using the invertible linear transformation $(u_8,v_8)^T=H(\xi,\nu)^T$ again for mapping \eqref{ns.2}, here
\begin{eqnarray*}
H=\left[ \begin {array}{cc}
-\frac{1}{2 \left(\beta -2\right)}&\frac{\sqrt{4 \beta -9}}{2 \left(\beta -2\right)} \\
\noalign{\medskip}
1&0
\end {array} \right],
\end{eqnarray*}
we get
{\small\begin{eqnarray}
\left[
\begin{array}{cc}
   \xi_1\\
   \\
   \nu_1
\end{array}
\right]
\mapsto
\left[\begin{array}{cc}
a_{1,0}\xi_{{1}}+a_{0,1}\nu_{{1}}+a_{2,0}\xi_1^2+a_{1,1}\xi_{{1}}\nu_{{1}}+a_{0,2}\nu_1^2
\\
   \\
b_{1,0}\xi_{{1}}+b_{0,1}\nu_{{1}}+b_{2,0}\xi_1^2+b_{1,1}\xi_{{1}}\nu_{{1}}+b_{0,2}\nu_1^2
 \end{array}\right],
\label{ns.3}
\end{eqnarray}}
where
\begin{eqnarray*}
&&a_{1,0}=b_{0,1}=\frac{2 \beta -5}{2 \left(\beta -2\right)},~~a_{1,1}=\frac{\sqrt{4 \beta -9}\, \beta}{2 \left(\beta -2\right)},~~a_{2,0}=-\frac{\beta}{2 \left(\beta -2\right)},\\
&&b_{1,1}=-\frac{\beta  \left(\beta -4\right)}{2 \left(\beta -2\right)^{2}},~~b_{0,2}=\frac{\beta  \sqrt{4 \beta -9}\,}{2 \left(\beta -2\right)^{2}},~~a_{0,1}=-b_{1,0}=\frac{\sqrt{4 \beta -9}}{2 \left(\beta -2\right)},\\
&&b_{2,0}=\frac{\beta  \left(\beta -3\right)}{2 \sqrt{4 \beta -9}\, \left(\beta -2\right)^{2}},~~a_{0,2}=0.
\end{eqnarray*}

Let $z:=\xi+{\bf i}\nu$, mapping \eqref{ns.3} is converted into the following complex form
\begin{equation}
z\mapsto \xi_0z+\sum_{i=0}^2\gamma_{2-i,i}z^{2-i}\bar{z}^i,
\label{ns.4}
\end{equation}
where
\begin{eqnarray*}
\!\!\!\!&&\gamma_{2,0}=-\frac{\beta  \left(\beta -3\right)}{4 \left(\beta -2\right)^{2}}-\frac{{\bf{i}} \beta  \left(\beta -3\right) \left(\beta -\frac{5}{2}\right)}{2\sqrt{4 \beta -9}\, \left(\beta -2\right)^{2}},\\
\!\!\!\!&&\gamma_{1,1}=-\frac{\beta}{4 \beta-8}-\frac{3 \,\bf{i} \beta}{\sqrt{4 \beta -9}\, \left(4 \beta -8\right)},\\
\!\!\!\!&&\gamma_{0,2}=-\frac{\beta}{4 \left(\beta -2\right)^{2}}+\frac{{\bf{i}} \beta  \left(\beta^{2}-3 \beta +\frac{3}{2}\right)}{2\sqrt{4 \beta -9}\, \left(\beta -2\right)^{2}}.\\
\end{eqnarray*}
Applying the quadratic near-identity transformation
$$
z=\omega+q_{2,0}\omega^2+q_{1,1}\omega\bar{\omega}+q_{0,2}\bar{\omega}^2,
$$
where $q_{2,0}$, $q_{1,1}$ and $q_{0,2}$ are undetermined coefficients,
mapping \eqref{ns.4} becomes the form of
\begin{eqnarray}
\begin{array}{ll}
\omega\mapsto \xi_0\omega+(\gamma_{2,0}-(\xi_0^2-\xi_0)\,q_{2,0})\omega^2
+(\gamma_{1,1}-(\bar{\xi}_0\lambda_0-\xi_0)\,q_{1,1})\omega\bar{\omega}
\\~~~~~~
+(\gamma_{0,2}-(\bar{\xi}_0^2-\xi_0)\,q_{0,2})\bar{\omega}^2+O(|\omega|^3).
\label{ns.5}
\end{array}
\end{eqnarray}
Let
$$
q_{2,0}:=\frac{\gamma_{2,0}}{\xi_0^2-\xi_0},~~q_{1,1}:=\frac{\gamma_{1,1}}{\bar{\xi}_0\xi_0-\xi_0},
~~q_{0,2}:=\frac{\gamma_{0,2}}{\bar{\xi}_0^2-\xi_0},
$$
mapping \eqref{ns.5} can be simplified as
\begin{eqnarray}
&&\omega\mapsto \xi_0\omega+p_{3,0}\omega^3+p_{2,1}\omega^2\bar{\omega}+p_{1,2}\omega\bar{\omega}^2+p_{0,3}\bar{\omega}^3+O(|\omega|^4),
\label{ns.6}
\end{eqnarray}
where
\begin{eqnarray*}
&&\!\!\!\!\!\!\!p_{2,1}:=-\frac{\left(-9+{\bf{i}} \left(\beta -2\right) \sqrt{4 \beta -9}+4 \beta\right) \beta^{3}}{2 \left(3 \beta -7\right) \left(4 \beta -9\right) \left({\bf{i}} \sqrt{4 \beta -9}-2 \beta +5\right)}.\\
\end{eqnarray*}
Further, using the cubic near-identity transformation
$$
\omega=\vartheta+q_{3,0}\vartheta^3+q_{2,1}\vartheta^2\bar{\vartheta}+q_{1,2}\vartheta\bar{\vartheta}^2+q_{0,3}\bar{\vartheta}^3,
$$
where $q_{3,0}$, $q_{2,1}$, $q_{1,2}$ and $q_{0,3}$ are undetermined coefficients, mapping \eqref{ns.6} is reformulated as the form
\begin{eqnarray*}
&&\vartheta\mapsto\xi_0\vartheta+(p_{3,0}-(\xi_0^3-\xi)q_{3,0})\vartheta^3
+(p_{2,1}-(\xi_0|\xi_0|^2-\xi_0)q_{2,1})\vartheta^2\bar{\vartheta}
\\
&&~~~~~~+(p_{1,2}-(\bar{\xi}_0|\xi_0|^2-\xi_0)q_{1,2})\vartheta\bar{\vartheta}^2+
(p_{0,3}-(\bar{\xi}_0^3-\xi_0)q_{0,3})\bar{\vartheta}^3+O(|\vartheta|^4),
\end{eqnarray*}
which can be reduced as
\begin{equation}
\vartheta\mapsto\xi_0\vartheta+p_{2,1}\vartheta^2\bar{\vartheta}+O(|\vartheta|^4)
\label{ns.7}
\end{equation}
by selecting
$$
q_{3,0}:=\frac{p_{3,0}}{\xi_0^3-\xi_0},~~q_{2,1}:=0,~~q_{1,2}:=\frac{p_{1,2}}{\bar{\xi}_0|\xi_0|^2-\xi_0},
~~q_{0,3}:=\frac{p_{0,3}}{\bar{\xi}_0^3-\xi_0}.
$$

Finally, let $\xi_0:=e^{{\bf i}\varsigma}$, where
%%%%%%%%%%2 \beta-5-{\bf{i}}\sqrt{4 \beta -9}}{2 \beta-4}
\begin{eqnarray*}
\varsigma=
\left\{
\begin{array}{l}
{\rm arctan}{\frac{\sqrt{4\beta-9}}{5-2\beta}},~~5-2\beta>0,
\\
\pi+{\rm arctan}{\frac{\sqrt{4\beta-9}}{5-2\beta}},~~5-2\beta<0
%\label{varsigma1}
\end{array}
\right.
\end{eqnarray*}
and $\vartheta=\mathfrak{p}e^{{\bf i}\mathfrak{v}}$, mapping \eqref{ns.7} is changed to the form
$$
\vartheta=\mathfrak{p}e^{{\bf i}\mathfrak{v}}\mapsto\vartheta=\mathfrak{p}e^{{\bf i}(\mathfrak{v}+\varsigma)}
(1+\bar{\xi}_0p_{2,1}\mathfrak{p}^2)+O(\mathfrak{p}^4),
$$
which is equivalent to the mapping
\begin{eqnarray}
\left[
\begin{array}{cc}
  \mathfrak{p}  \\
   \mathfrak{v}
\end{array}
\right]
\mapsto
\left[\begin{array}{cc}
\mathfrak{p}|1+\bar{\xi}_0p_{2,1}\mathfrak{p}^2|+O(\mathfrak{p}^4)
\\
\mathfrak{v}+\varsigma+{\rm arg}(1+\bar{\xi}_0p_{2,1}\mathfrak{p}^2)+O(\mathfrak{p}^4)
 \end{array}\right].
\label{ns.8}
\end{eqnarray}
Note that
\begin{eqnarray*}
&&|1+\bar{\xi}_0p_{2,1}\mathfrak{p}^2|=((1+{\rm Re}(\bar{\xi}_0p_{2,1})\mathfrak{p}^2)^2+({\rm Im}(\bar{\xi}_0p_{2,1})\mathfrak{p}^2)^2)^{1/2}
\\
&&~~~~~~~~~~~~~~~~~\,=1+{\rm Re}(\bar{\xi}_0p_{2,1})\mathfrak{p}^2+O(\mathfrak{p}^4)
\end{eqnarray*}
and
$$
{\rm arg}(1+\bar{\xi}_0p_{2,1}\mathfrak{p}^2)={\rm arctan}\left(\frac{{\rm Im}(\bar{\xi}_0p_{2,1})\mathfrak{p}^2}{1+{\rm Re}(\bar{\xi}_0p_{2,1})\mathfrak{p}^2}\right)={\rm Im}(\bar{\xi}_0p_{2,1})\mathfrak{p}^2+O(\mathfrak{p}^4)
$$
for sufficiently small $\mathfrak{p}$, so mapping \eqref{ns.8} can be expressed as
\begin{eqnarray*}
\left\{
\begin{array}{ll}
\tilde{\mathfrak{p}}=\mathfrak{p}+\mathfrak{A}\,\mathfrak{p}^3+O(\mathfrak{p}^4),
\\
\tilde{\mathfrak{v}}= \mathfrak{v}+\varsigma+{\rm Im}(\bar{\xi}_0p_{2,1})\mathfrak{p}^2+O(\mathfrak{p}^4),
\end{array}
\right.
\end{eqnarray*}
where
\begin{eqnarray*}
&&\mathfrak{A}:={\rm Re}(\bar{\xi}_0p_{2,1})=\frac{\beta^{3} \left(\beta^{2}-9 \beta +15\right)}{16 \left(\beta -2\right)^{3} \left(4 \beta -9\right)}.
\end{eqnarray*}
It is clear that $\mathfrak{A}\neq0$ as $\beta\neq\frac{9}{2}+\frac{\sqrt{21}}{2}$, implying the condition (C.3) in \cite[p.141]{Kuznetsov} is satisfied.

Therefore, if $(\lambda, \mu, \beta)\in \mathfrak{L}$ with
\begin{eqnarray*}
\beta\neq7/3,~\beta\neq5/2,~\beta\neq\frac{9}{2}+\frac{\sqrt{21}}{2}
\end{eqnarray*}
hold, all conditions in \cite[Theorem 4.6, p.141]{Kuznetsov} are fulfilled, which follows that system \eqref{eq2.1} undergoes a Neimark-Sacker bifurcation.
To be more specific, if $\mathfrak{A}>0$ (or $\mathfrak{A}<0$), system \eqref{eq2.1} undergoes a subcritical (or supercritical) Neimark-Sacker bifurcation and generates a unique stable (or unstable) invariant circle surrounding $E_2$ from the region $\mathfrak{D}_{53}$ (resp. $\mathfrak{D}_{64}$) to $\mathfrak{D}_{64}$ (resp. $\mathfrak{D}_{53}$) and from the region $\mathfrak{D}_{63}$ (resp. $\mathfrak{D}_{71}$) to $\mathfrak{D}_{71}$ (resp. $\mathfrak{D}_{63}$).  This completes the proof.
\end{proof}

Theorem \ref{nsE2} demonstrates that when $\mathfrak{A}>0$ (or $\mathfrak{A}<0$), system \eqref{eq2.1} undergoes a subcritical (or supercritical) Neimark-Sacker bifurcation and produces a unique stable (or unstable) invariant circle near the non-hyperbolic condition $\left\{\lambda>0,\mu=\frac{\beta}{\beta-2},\beta>\frac{9}{4}, \beta\neq7/3,~\beta\neq5/2,~\beta\neq\frac{9}{2}+\frac{\sqrt{21}}{2}\right\}$. From a biological perspective, if the initial population densities of prey, predator and top predator are located inside (or outside) the invariant circle, their densities will oscillate near the invariant circle (or trend to infinity). Consequently, the population densities of prey, predator and top predator can coexist within the invariant circle (or all trend to infinity).
%This could lead to an uncontrolled in the numbers of prey, predator and top predator, thus necessitating the avoidance of Neimark-Sacker bifurcations.
%%%%%%%%%%%%%%%%%%%%%%%%%%%%%%%%%%%%%%%%%%%%%%%%%%%%%%%%%%%%%%%%%%%%%%%%%%%%%%%%%%%%%%%%%%%%%%%%%%%%%%%%%%%%%%%%%

\section{Codimension-2 bifurcations}
In this section, we consider the codimension-2 bifurcations at $E_2$.
\setcounter{equation}{0}
%In the third section, we have discussed the codimension $1$ bifurcations of system \eqref{eq2.1}. In what follows, we will focus on codimension 2 bifurcations.
\subsection{$1:2$ resonance at $E_2$}
From Section 2, we find that when the parameter $\Lambda$ crosses $\mathfrak{L}_{24}=\{(\lambda,\mu,\beta)|\lambda>0,\mu=9,\beta=9/4\}$ in Table \ref{Table3}, the eigenvalues of $J(E_2)$ are $t_{1}=t_{2}=-1$. Consequently, system \eqref{eq2.1} may undergo a 1:2 strong resonance.
\begin{thm}
System \eqref{eq2.1} undergoes a $1:2$ resonance bifurcation as the parameter $\Lambda\in \mathbb{R}_+^{3}$ crosses $\mathfrak{L}_{24}$. Specifically, in a sufficiently small neighborhood of the 1:2 resonance region $\mathfrak{L}_{24}$, the following bifurcation phenomena occur:
\begin{description}
\item[${\rm (i)}$] If the parameter $\Lambda$ crosses
            \begin{eqnarray*}
            &&\!\!\!\!\!\!\!\!\!\!\!\!\mathfrak{L}^{'}_{2}:=\left\{(\lambda,\mu,\beta)\bigg|\lambda>0,
            \mu =-\frac{3 \beta}{\beta-3}+O\left(\left|\left(\beta-\frac{9}{4}\right)\right|^2\right),
            \frac{3}{2}<\beta<3\right\}
            \end{eqnarray*}
            from the region $\mu <-3 \beta/(\beta-3)+O(|(\beta-9/4)|^2)$ to $\mu >-3 \beta/(\beta-3)+O(|(\beta-9/4)|^2)$, then system \eqref{eq2.1} undergoes a subcritical flip bifurcation and produces an unstable period-two orbit $\{Q_{11},Q_{12}\}$. Additionally, $E_2$ is from unstable node to saddle point.
\item[${\rm (ii)}$] As the parameter $\Lambda$ crosses
            \begin{eqnarray*}
            &&\mathfrak{L}^{'}_{3}:=\left\{(\lambda,\mu,\beta)\bigg|\lambda>0,\mu =\frac{\beta}{\beta -2}+O\left(\left|\left(\beta-\frac{9}{4}\right)\right|^2\right),
            \beta>\frac{9}{4}\right\}
            \end{eqnarray*}
            from the region $\mu <\beta(\beta -2)+O(|(\beta-9/4)|^2)$ to $\mu >\beta/(\beta -2)+O(|(\beta-9/4)|^2)$, system \eqref{eq2.1} undergoes a Neimark-Sacker bifurcation and generates a unique invariant circle $\Gamma$. Moreover, $\Gamma$ coexists with the unstable period-two orbit $\{Q_{11},Q_{12}\}$ and $E_2$ becomes a saddle-focus simultaneously.
\item[${\rm (iii)}$] When the parameter $\Lambda$ crosses
            \begin{eqnarray*}
            H_{2r}:=\left\{(\lambda,\mu,\beta)\bigg|\lambda>0,\mu =-\frac{41 \beta}{7 \beta -26}+O\left(\left|\left(\beta-\frac{9}{4}\right)\right|^2\right),\beta>\frac{9}{4}\right\}
            \end{eqnarray*}
            from the region $\mu <-41 \beta/(7 \beta -26)+O(|(\beta-9/4)|^2)$ to $\mu >-41 \beta/(7 \beta -26)+O(|(\beta-9/4)|^2)$, the invariant circle disappears.

\iffalse
\item[$(4)$] If the parameter $(\mu,\beta)$ are near the curve
            \begin{eqnarray*}
            &&\mathfrak{L}^{'}_1:=\left\{(\mu,\beta)\bigg|\mu =-\frac{3 \beta}{\beta-3}+O\left(\left|\left(\beta-\frac{9}{4}\right)\right|^2\right),\mu<9\right\},
            \end{eqnarray*}
            as $(\mu,\beta)$ from the region where $\mu <-3 \beta/(\beta-3)+O(|(\beta-9/4)|^2)$ to the region where $\mu >-3 \beta/(\beta-3)+O(|(\beta-9/4)|^2)$, then system \eqref{eq2.1} undergoes a subcritical flip bifurcation and generates in an unstable period two orbit $\{Q_{11},Q_{12}\}$. Meanwhile, $E_2$ is saddle point. Additionally, the saddle point and unstable period two orbit $\{Q_{11},Q_{12}\}$ coexist in the region between curves $\mathfrak{L}^{'}_1$ and $\mathfrak{L}^{'}_4$.
\fi
\end{description}
\label{th1-2}
\end{thm}

\begin{proof}
Employing the transformation
$$(x,y,z)=\left(m_1+\frac{1}{\beta},n_1+\frac{\beta  \mu -\beta -\mu}{\beta \mu},l_1\right),$$
$E_2$ is translated to the origin $O$.
Let
\begin{eqnarray*}
\alpha_1:=\mu-9~~\mbox{and}~~\alpha_2:=\beta-\frac{9}{4},
\label{1.2.1}
\end{eqnarray*}
system \eqref{eq2.1} can be transformed into the following form
\begin{eqnarray}
\left(
\begin{array}{l}
m_1
\\
n_1
\\
l_1
\end{array}
\right)
\mapsto
\left(
\begin{array}{l}
-\frac{\left(4 \alpha_1-4 \alpha_2+27\right) m_1}{4 \alpha_2+9}-\frac{4 \left(\alpha_1+9\right) n_1}{4 \alpha_2+9}-\frac{4 \left(\alpha_1+9\right) l_1}{4 \alpha_2+9}\\
-\left(\alpha_1+9\right) m_1^{2}-\left(\alpha_1+9\right) n_1 m_1-\left(\alpha_1+9\right) l_1 m_1
\\
\frac{m_1 \left(4 \alpha_1 \alpha_2+5 \alpha_1+32 \alpha_2+36\right)}{4 \left(\alpha_1+9\right)}+n_1-\frac{l_1 \left(4 \alpha_1 \alpha_2+5 \alpha_1+32 \alpha_2+36\right)}{4 \left(\alpha_1+9\right)}\\
+\frac{\left(4 \alpha_2+9\right) n_1 m_1}{4}-\frac{\left(4 \alpha_2+9\right) n_1 l_1}{4}
\\
\frac{\lambda  \left(4 \alpha_1 \alpha_2+5 \alpha_1+32 \alpha_2+36\right) l_1}{\left(4 \alpha_2+9\right) \left(\alpha_1+9\right)}+\lambda  n_1 l_1\\
\end{array}
\right).
\label{1.2.2}
\end{eqnarray}
According to the proof of Theorem \ref{th3.1}, mapping \eqref{1.2.2} is restricted to a two dimensional $C^{2}$ center manifold
\begin{eqnarray*}
l_1=h_{6}(m_1,n_1)=O(|(m_1,n_1)|^3),
\end{eqnarray*}
and then \eqref{1.2.2} is converted into the following two dimensional form
\begin{eqnarray}
\left[
\begin{array}{ccc}
   m_1  \\
   \noalign{\medskip}
   n_1
\end{array}
\right]
\mapsto
\left[\begin{array}{cc}
-\frac{\left(27-4 \alpha_2+4 \alpha_1\right) m_1}{4 \alpha_2+9}-\frac{4 \left(\alpha_1+9\right) n_1}{4 \alpha_2+9}-m_1^{2} \left(\alpha_1+9\right)-m_1\left(\alpha_1+9\right) n_1\\
\noalign{\medskip}
 \frac{m_1 \left(4 \alpha_1 \alpha_2+5 \alpha_1+32 \alpha_2+36\right)}{4 \left(\alpha_1+9\right)}+n_1+\frac{\left(4 \alpha_2+9\right) n_1 m_1}{4}
 \end{array}\right].
\label{1.2.3}
\end{eqnarray}

Applying the subsequent transformation
\begin{eqnarray*}
\left(
\begin{array}{cc}
m_1 \\
n_1
\end{array}
\right)
\rightarrow
\left(
\begin{array}{cc}
-2&1\\
~1&0
\end{array}
\right)
\left(
\begin{array}{cc}
\zeta_1
\\
\eta_1
\end{array}
\right)
\end{eqnarray*}
to mapping \eqref{1.2.3}, we obtain
\begin{eqnarray}
\begin{aligned}
\left(
\begin{array}{cc}
\zeta_1 \\
\eta_1
\end{array}
\right)
\rightarrow
\left(
\begin{array}{ccc}-1+a_{10}&1+a_{01}\\b_{10}&-1+b_{01}\end{array}
\right)
\left(
\begin{array}{cc}
\zeta_1 \\
\eta_1
\end{array}
\right)
+\left(
\begin{array}{cc}
g(\zeta_1,\eta_1,\alpha_1,\alpha_2)\\
h(\zeta_1,\eta_1,\alpha_1,\alpha_2)
\end{array}
\right),
\end{aligned}
\label{1.2.4}
\end{eqnarray}
where
\begin{eqnarray*}
&&g(\zeta_1,\eta_1,\alpha_1,\alpha_2)=-\frac{\left(4 \alpha_2+9\right) \zeta_1^{2}}{2}+\frac{\left(4 \alpha_2+9\right) \eta_1 \zeta_1}{4},\\
&&h(\zeta_1,\eta_1,\alpha_1,\alpha_2)=-\left(2 \alpha_1+27+4 \alpha_2\right) \zeta_1^{2}+\frac{\left(6 \alpha_1+63+4 \alpha_2\right) \eta_1 \zeta_1}{2}-\eta_1^{2} \left(\alpha_1+9\right).
\end{eqnarray*}
In order to simplify the linear part of \eqref{1.2.4}, we utilize the transformation
\begin{eqnarray*}
\left(
\begin{array}{cc}
\zeta_1 \\
\eta_1
\end{array}
\right)
\rightarrow
\left(
\begin{array}{cc}
1+a_{01}&0\\
-a_{10}&1
\end{array}
\right)
\left(
\begin{array}{cc}
\zeta_2
\\
\eta_2
\end{array}
\right),
\end{eqnarray*}
and then mapping \eqref{1.2.4} is changed into the following mapping
\begin{eqnarray}
\left(
\begin{array}{cc}
\zeta_2 \\
\eta_2
\end{array}
\right)
\mapsto
\left(
\begin{array}{cc}
-\zeta_2+\eta_2+g_\ast(\zeta_2,\eta_2,\alpha_1,\alpha_2)\\
c_{10}\zeta_2+(-1+c_{20})\eta_2+h_\ast(\zeta_2,\eta_2,\alpha_1,\alpha_2)
\end{array}
\right),
\label{1.2.5}
\end{eqnarray}
where
{\small
\begin{eqnarray*}
\!\!\!\!\!\!\!\!\!\!&&c_{10}=\frac{\left(-4 \alpha_1-48\right) \alpha_2+3 \alpha_1}{4 \alpha_2+9},\\
\!\!\!\!\!\!\!\!\!\!&&c_{20}=\frac{-4 \alpha_1+16 \alpha_2}{4 \alpha_2+9},\\
\!\!\!\!\!\!\!\!\!\!&&g_\ast(\zeta_2,\eta_2,\alpha_1,\alpha_2)=\frac{\left(4 \alpha_2+9\right) \eta_2 \zeta_2}{4}-\frac{\left(4 \alpha_2+9\right) \zeta_2^{2}}{2},\\
\!\!\!\!\!\!\!\!\!\!&&h_\ast(\zeta_2,\eta_2,\alpha_1,\alpha_2)=\frac{\zeta_2^{2} \left(8 \alpha_1^{2} \alpha_2-16 \alpha_2^{2} \alpha_1-6 \alpha_1^{2}+80 \alpha_1 \alpha_2-128 \alpha_2^{2}-171 \alpha_1+144 \alpha_2-972\right)}{4 \left(\alpha_1+9\right)}
\\
\!\!\!\!\!\!\!\!\!\!&&~~~~~~~~~~~~~~~~~~~~~-\frac{\eta_2 \zeta_2 \left(8 \alpha_1^{2} \alpha_2-16 \alpha_2^{2} \alpha_1-22 \alpha_1^{2}+80 \alpha_1 \alpha_2-128 \alpha_2^{2}-459 \alpha_1+144 \alpha_2-2268\right)}{8 \left(\alpha_1+9\right)}\\
\!\!\!\!\!\!\!\!\!\!&&~~~~~~~~~~~~~~~~~~~~~-\left(\alpha_1+9\right) \eta_2^{2}.
\end{eqnarray*}}Let
\begin{eqnarray*}
\epsilon_1(\alpha_1,\alpha_2):=c_{10},~~~~~\epsilon_2(\alpha_1,\alpha_2):=c_{20}.
\label{eq6.003}
\end{eqnarray*}
One can check that the mapping
\begin{eqnarray}
(\alpha_1,\alpha_2)\mapsto(\epsilon_1(\alpha_1,\alpha_2),\epsilon_2(\alpha_1,\alpha_2))
\label{1.2.6}
\end{eqnarray}
is regular at $(\alpha_1,\alpha_2)=(0,0)$ because
\begin{eqnarray*}
\det
\left.
\left(
\begin{array}{ll}
\frac{\partial\epsilon_1}{\partial\alpha_1}&\frac{\partial\epsilon_1}{\partial\alpha_2}\\
\frac{\partial\epsilon_2}{\partial\alpha_1}&\frac{\partial\epsilon_2}{\partial\alpha_2}
\end{array}
\right)
\right|_{(\alpha_1,\alpha_2)=(0,0)}
=\det
\left(
\begin{array}{ll}
~~\frac{1}{3}&-\frac{16}{3}\\
-\frac{4}{9}&~~\frac{16}{9}
\end{array}
\right)
=-\frac{16}{9}\neq0.
\end{eqnarray*}
Thus the transversality condition of $1:2$ resonance shown in \cite[p.443]{Kuznetsov} is fulfilled.

Next, we continue to verify the nondegeneracy conditions. Let $\epsilon:=(\epsilon_1,\epsilon_2)$ be a new parameter, mapping \eqref{1.2.5} can be converted into the following mapping
\begin{eqnarray}
\left(
\begin{array}{l}
\zeta_2 \\
\eta_2
\end{array}
\right)
\mapsto
\left(
\begin{array}{l}
-\zeta_2+\eta_2+g_{20}\zeta^2_2+g_{11}\zeta_2\eta_2+g_{02}\eta^2_2+g_{30}\zeta^3_2+g_{21}\zeta^2_2\eta_2\\
+g_{12}\zeta_2\eta^2_2+g_{03}\eta^3_2\\
\epsilon_1\zeta_2+(-1+\epsilon_2)\eta_2+h_{20}\zeta^2_2+h_{11}\zeta_2\eta_2+h_{02}\eta^2_2+h_{30}\zeta^3_2\\
+h_{21}\zeta^2_2\eta_2+h_{12}\zeta_2\eta^2_2+h_{03}\eta^3_2\\
\end{array}
\right),
\label{1.2.7}
\end{eqnarray}
where
\begin{eqnarray*}
&&g_{20}=-\frac{9}{2}+\frac{\epsilon_{1}}{2}+\frac{3 \epsilon_{2}}{8}+O(|(\epsilon_{1},\epsilon_{2})|^2),~~
g_{11}=\frac{9}{4}-\frac{\epsilon_{1}}{4}-\frac{3 \epsilon_{2}}{16}+O(|(\epsilon_{1},\epsilon_{2})|^2),\\
&&g_{02}=O(|(\epsilon_{1},\epsilon_{2})|^2),g_{30}=O(|(\epsilon_{1},\epsilon_{2})|^2),~~
g_{12}=O(|(\epsilon_{1},\epsilon_{2})|^2),\\
&&g_{21}=O(|(\epsilon_{1},\epsilon_{2})|^2),g_{03}=O(|(\epsilon_{1},\epsilon_{2})|^2),\\
&&h_{20}=-27+\frac{9 \epsilon_{2}}{2}+\frac{3 \epsilon_{1}}{4}+O(|(\epsilon_{1},\epsilon_{2})|^2),~
h_{11}=\frac{63}{2}-\frac{33 \epsilon_{2}}{4}-\frac{19 \epsilon_{1}}{8}+O(|(\epsilon_{1},\epsilon_{2})|^2),\\
&&h_{02}=3 \epsilon_{2}-9+\epsilon_{1}+O(|(\epsilon_{1},\epsilon_{2})|^2),~~h_{30}=O(|(\epsilon_{1},\epsilon_{2})|^2),~
h_{12}=O(|(\epsilon_{1},\epsilon_{2})|^2),\\
&&h_{21}=O(|(\epsilon_{1},\epsilon_{2})|^2),~~h_{03}=O(|(\epsilon_{1},\epsilon_{2})|^2).
\end{eqnarray*}
According to \cite[Lemma 9.9, p. 437]{Kuznetsov}, there exists a near-identity change of coordinates which transforms system \eqref{1.2.7} into a $C^1$ mapping $\Gamma_{\epsilon}:\mathbb{R}^2\rightarrow\mathbb{R}^2$, given by
\begin{eqnarray*}
\begin{aligned}
\left(
\begin{array}{l}
\zeta_2 \\
\eta_2
\end{array}
\right)
\rightarrow
\left(
\begin{array}{ccc}-1&1\\ \epsilon_1&-1+\epsilon_2\end{array}
\right)
\left(
\begin{array}{l}
\zeta_2 \\
\eta_2
\end{array}
\right)
+\left(
\begin{array}{l}
~~~~~~~~~~~0\\
C(\epsilon)\zeta_2^3+D(\epsilon)\zeta_2^2\eta_3
\end{array}
\right)
+O(|(\zeta_2,\eta_2)|^4),
\end{aligned}
\end{eqnarray*}
where
\begin{eqnarray*}
&&C(\epsilon)=\frac{243}{4}+\frac{243}{8} \epsilon_{2}+\frac{27}{4} \epsilon_{1}+O(|\epsilon|^2),\\
&&D(\epsilon)=\frac{2187}{4}-\frac{1539}{8} \epsilon_{2}-\frac{243}{4} \epsilon_{1}+O(|\epsilon|^2).
\end{eqnarray*}
Consequently, $C(0)=243/4$ and $D(0)+3C(0)=729$, indicating that the non-degeneracy conditions $(R2.1)$ and $(R2.2)$ of $1:2$ resonance shown in \cite[p.443]{Kuznetsov} are satisfied. So system \eqref{eq2.1} generates a $1:2$ resonance.

Finally, from \cite[Theorem 9.3, p.442]{Kuznetsov}, we see that the second iterate of the smooth mapping $\Gamma_{\epsilon}$ can be represented as
$$(\zeta_2,\eta_2)^T\mapsto\varphi(1,\zeta_2,\eta_2,\epsilon_1,\epsilon_2)+O(|(\zeta_2,\eta_2)|^4)$$
for sufficiently small $|\epsilon|=|(\epsilon_1, \epsilon_2)|$, where $\varphi(1,\zeta_2,\eta_2,\epsilon_1,\epsilon_2)$ is the time-1 mapping of a planar flow $\varphi(t,\zeta_2,\eta_2,\epsilon_1,\epsilon_2)$, which is smoothly equivalent to the following system
\begin{eqnarray}
\begin{aligned}
%\frac{d}{dt}
\left(
\begin{array}{l}
\dot{\zeta}_3 \\
\dot{\eta}_3
\end{array}
\right)
\rightarrow
\left(
\begin{array}{cc}0&1\\ \gamma_1(\epsilon)&\gamma_2(\epsilon)\end{array}
\right)
\left(
\begin{array}{l}
\zeta_3 \\
\eta_3
\end{array}
\right)\!\!
+\!\!\left(
\begin{array}{l}
~~~~~~~~~~~0\\
C_1(\epsilon)\zeta_3^3+D_1(\epsilon)\zeta_3^2\eta_3
\end{array}
\right),
\end{aligned}
\label{1.2.8}
\end{eqnarray}
where
\begin{eqnarray*}
&&\gamma_1(\epsilon):=4\epsilon_1+O(|\epsilon|^2),~~~~~~\gamma_2(\epsilon):=-2\epsilon_1-2\epsilon_2+O(|\epsilon|^2),\\
&&C_1(\epsilon):=243+\frac{243 \epsilon_{2}}{2}+27 \epsilon_{1}+O(|\epsilon|^2),\\
&&D_1(\epsilon):=-1458+\frac{405 \epsilon_{2}}{2}+81 \epsilon_{1}+O(|\epsilon|^2).
\end{eqnarray*}
Employing the scaling transformation
$$(\zeta_3,\eta_3)=\left(-\frac{\sqrt{C_1(\epsilon)}}{D_1(\epsilon)}\zeta_4,
\frac{(C_1(\epsilon))^{\frac{3}{2}}}{(D_1(\epsilon))^2}\eta_4\right),
~~~~~~dt=-\frac{D_1(\epsilon)}{C_1(\epsilon)}d\tau,$$
mapping \eqref{1.2.8} is rewritten as
\begin{eqnarray}
\begin{aligned}
%\frac{d}{dt}
\left(
\begin{array}{l}
\dot{\zeta}_4 \\
\dot{\eta}_4
\end{array}
\right)
\rightarrow
\left(
\begin{array}{cc}0&1\\ \delta_1(\epsilon)&\delta_2(\epsilon)\end{array}
\right)
\left(
\begin{array}{l}
\zeta_4 \\
\eta_4
\end{array}
\right)
+\left(
\begin{array}{l}
~~~~0\\
\zeta_4^3-\zeta_4^2\eta_4
\end{array}
\right),
\end{aligned}
\label{1.2.9}
\end{eqnarray}
where
\begin{eqnarray}
\begin{array}{l}
\delta_1(\epsilon):=\frac{(D_1(\epsilon))^2}{(C_1(\epsilon))^2}\gamma_1(\epsilon)=144\,\epsilon_1+O(|\epsilon|^2),\\
\delta_2(\epsilon):=-\frac{D_1(\epsilon)}{C_1(\epsilon)}\gamma_2(\epsilon)=-12\,\epsilon_1-12\,\epsilon_2+O(|\epsilon|^2).
\end{array}
\label{1.2.10}
\end{eqnarray}
Let $\delta:=(\delta_1,\delta_2)$ as a new parameter, system \eqref{1.2.9} has the following bifurcations (the proof of which can be found in \cite[p.443-447]{Kuznetsov} around $(\zeta_4,\eta_4)=(0,0)$ for sufficiently small $|\delta|$), as shown in Figure \ref{resonance1bi2-2}.
\begin{figure}[htbp]
\centering
{
\includegraphics[width=10cm]{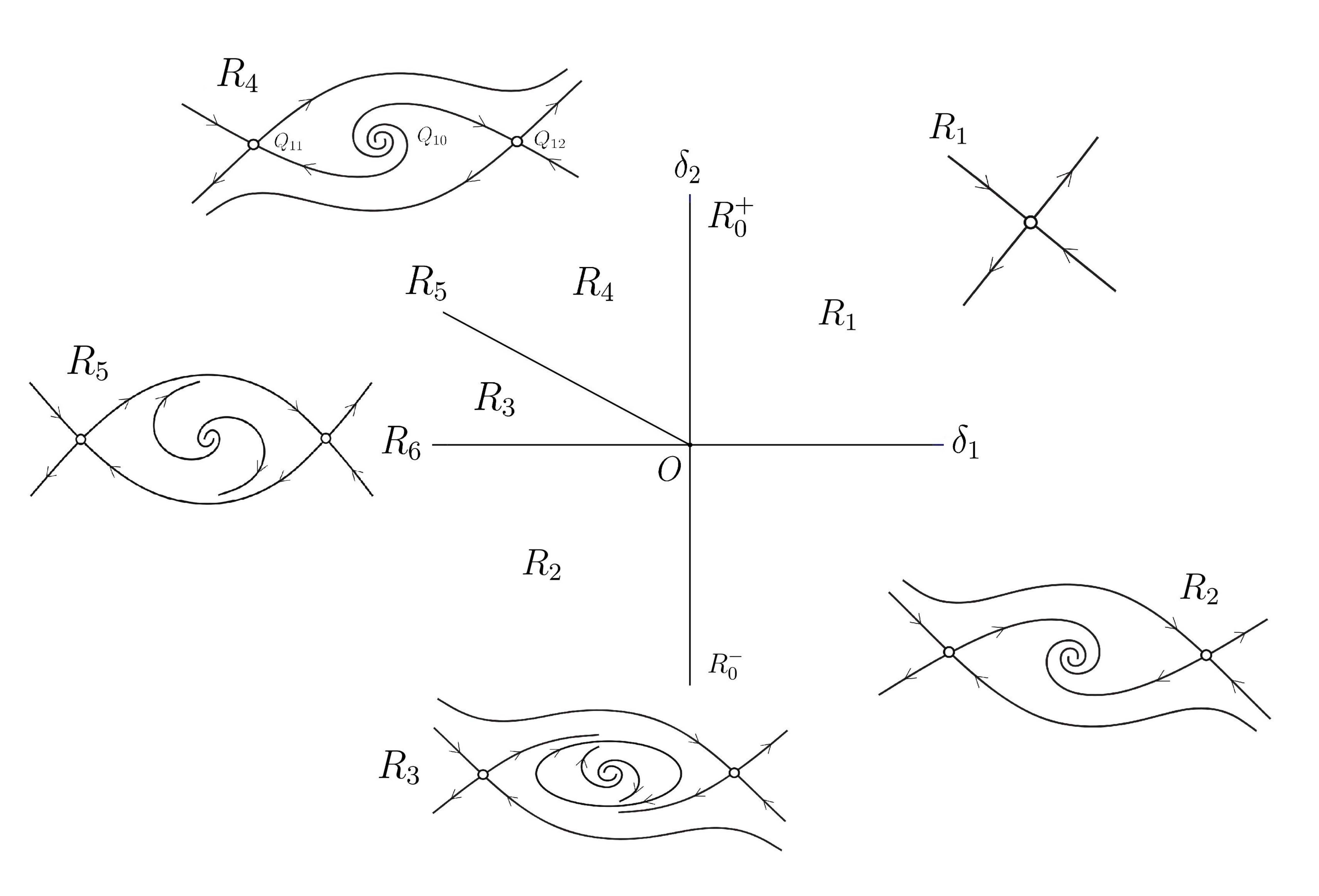}
}
\caption{Bifurcation diagram of system \eqref{1.2.9}}
\label{resonance1bi2-2}
\end{figure}
\begin{description}
\item[(i)]~~When the parameter $\delta$ lies in the region $R_1:=\{\delta|\delta_1>0, \delta_2\in\mathbb{R}\}$, system \eqref{1.2.9} possesses a unique trivial equilibrium point $Q_{10}$, which is a saddle point. If $\delta$ lies in the line $R^-_0:=\{\delta|\delta_1=0, \delta_2<0\}$, $Q_{10}$ is still a saddle point. As the parameter $\delta$ crosses $R^-_0$ from $R_1$ to the region $R_2:=\{\delta|\delta_1<0, \delta_2<0\}$, system \eqref{1.2.9} undergoes a pitchfork bifurcation and generates a pair of symmetry-coupled saddles $Q_{11}$, $Q_{12}$. Meanwhile, the trivial equilibrium point $Q_{10}$ becomes a stable node and subsequently turns into a stable focus in the region $R_2$.

\item[(ii)]~When $\delta$ lies in the line $R_6:=\{\delta|\delta_1<0, \delta_2=0\}$, the equilibrium point $Q_{10}$ remains as a stable focus. When the parameter $\delta$ crosses $R_6$ from $R_2$ to the region $R_3:=\{\delta|0<\delta_2<-\delta_1/5+o(\delta_1),\delta_1<0\}$, system \eqref{1.2.9} undergoes a Hopf bifurcation and produces a unique stable limit cycle. The trivial equilibrium point $Q_{10}$ becomes an unstable focus simultaneously.

\item[(iii)]When $\delta$ lies in the curve $R_5:=\{\delta|\delta_2=-\delta_1/5+o(\delta_1),\delta_1<0\}$, the trivial equilibrium point $Q_{10}$ is still an unstable focus, while the two nontrivial equilibrium points $Q_{11}$ and $Q_{12}$ remain to be saddle points. Further, the heteroclinic orbits connecting saddles $Q_{11}$ and $Q_{12}$ appear simultaneously and form a heteroclinic cycle, then the limit cycle disappears. As $\delta$ crosses $R_5$ from $R_3$ to the region $R_4:=\{\delta|\delta_2>-\delta_1/5+o(\delta_1),\delta_1<0\}$, system \eqref{1.2.9} undergoes a heteroclinic bifurcation and the two heteroclinic orbits disappear. Additionally, the unstable equilibrium point $Q_{10}$ coexists with the saddles $Q_{11}$ and $Q_{12}$.

\item[(iv)]When $\delta$ lies in the line $R_0^{+}:=\{\delta|\delta_2>0, \delta_1=0\}$, the three equilibrium points coincide and form a saddle point. As $\delta$ crosses $R^+_0$ from $R_4$ to $R_1$, system \eqref{1.2.9} undergoes a pitchfork bifurcation and the two nontrivial equilibrium points $Q_{11}$ and $Q_{12}$ disappear.
\end{description}

Let $\phi_{\delta}^1$ be the time-1 map of system \eqref{1.2.9}. Then the equilibrium points of system \eqref{1.2.9} are the fixed points of $\phi_{\delta}^1$, both system \eqref{1.2.9} and mapping $\phi_{\delta}^1$ undergo pitchfork bifurcations. Furthermore, the Hopf bifurcation of system \eqref{1.2.9} turns into a Neimark-Sacker bifurcation of the mapping $\phi_{\delta}^1$ as the limit cycle of system \eqref{1.2.9} becomes a closed invariant curve of $\phi_{\delta}^1$ with the corresponding stability properties.

As described in \cite[p.446]{Kuznetsov}, we see that the mapping $\phi_{\delta}^1$ approximates the second iterate of the original mapping $\Gamma_{{\epsilon}}$ near the 1:2 resonance region. Consequently, the trivial fixed point $Q_{10}$ of $\phi_{\delta}^1$ is the origin $O$ of $\Gamma_{{\epsilon}}$, while the other two nontrivial fixed points of $\phi_{\delta}^1$ form a period-two orbit. Thus, the pitchfork bifurcation of $\phi_{\delta}^1$ turns into a flip bifurcation of $\Gamma_{{\epsilon}}$. The invariant circle of $\phi_{\delta}^1$ produced from the Neimark-Sacker bifurcation near $Q_{10}$ corresponds to an invariant circle of $\Gamma_{{\epsilon}}$ generated by the same bifurcation near the trivial fixed point $Q_{10}$. Moreover, for $\Gamma_{{\epsilon}}$ the curve analogous to the heteroclinic bifurcation curve $R_5$ does not exist. However, in a neighborhood of $R_5$, the mapping $\Gamma_{{\epsilon}}$ produces infinite series of bifurcations in which homoclinic structures are involved. Consequently, as the parameter crosses the corresponding bifurcation sets, the long-period cycles appear and disappear from fold bifurcations.

Near the $1:2$ resonance point, when the parameter $\Lambda$ crosses $\mathfrak{L}^{'}_{2}$, the period-two orbit is generated. Due to the subcritical bifurcation direction, this periodic orbit is unstable and works together with the stability change of the fixed point $E_2$ to form a dynamic and drastic transition. When the parameter further changes to $\mathfrak{L}^{'}_{3}$, system generates an invariant circle $\Gamma$, coexisting with period-two orbit. When the parameters approach the curve $H_{2r}$, heteroclinic orbits appear on the invariant circle, connecting the saddle points of period-two orbit. This heteroclinic structure leads to the destabilization of the invariant circle, ultimately resulting in its disappearance. These phenomena manifest as the sensitivity of dynamic systems to small changes in parameters.
\end{proof}

%\begin{rmk}
%When the parameter $\Lambda$ crosses $\tilde{\mathfrak{L}}_{22}:=\{(\lambda,\mu,\beta)|\lambda=\mu(6-\mu)/2(\mu -3),\mu=-3 \beta/(\beta -3),3/2<\beta<2\}\subset \mathfrak{L}_{22}$, the eigenvalues of $J(E_2)$ are also $t_{1}=t_{2}=-1$. It follows that system \eqref{eq2.1} undergoes a $1:2$ resonance bifurcation as the parameter $\Lambda$ crosses $\tilde{\mathfrak{L}}_{22}$ by using the same idea as done in the proof of Theorem \ref{th1-2}.
%\end{rmk}
%According to Theorem \ref{th1-2},  we see that the system is capable of producing a $1:2$ resonance and exhibits diverse dynamic traits near this resonance point. Hence, when the parameter $(\lambda,\mu,\beta)$ approaches $\{\lambda>0,\mu=9,\beta=9/4\}$, the system's dynamics become highly responsive to parameter variations. Additionally, the non-degenerate Neimark-Sacker bifurcation bears significant biological importance, encompassing periodic or quasi-periodic oscillations in predator-prey populations. These oscillations give rise to extended periodic fluctuations, substantial population outbreaks, and even chaos within the predator-prey system, stemming from periodic cycles with periods of 2, 4, 8, or due to the presence of a homoclinic structure.
%Next, we will investigate another codimension two bifurcation, i.e., $1:3$ resonance.

%%%%%%%%%%%%%%%%%%%%%%%%%%%%%%%%%%%%%%%%%%%%%%%%%%%%%%%%%%%%%%%%%%%%%%%%%%%%%%%%%%%%%%%%%%%%%%%%%
\subsection{$1:3$ resonance at $E_2$}
It is known that when the eigenvalues of the Jacobian matrix at certain fixed point are
$$-\frac{1}{2}+\frac{\sqrt{3}}{2}\,{\bf{i}}~~\mbox{and}~~-\frac{1}{2}-\frac{\sqrt{3}}{2}\,{\bf{i}},$$
the two-dimensional system may have a $1:3$ resonance near the fixed point.

Regarding system \eqref{eq2.1},  when the following conditions
\begin{small}
\begin{eqnarray*}
&&{P_{E_2}}(-\frac{1}{2}+\frac{\sqrt{3}}{2}\,{\bf{i}})=\frac{\left(\left(2 \lambda +1\right) \mu^{2}+\left(-\lambda -4\right) \mu -\lambda \right) \beta^{2}+\left(\left(-7 \lambda -1\right) \mu^{2}+4 \lambda  \mu \right) \beta +5 \lambda  \,\mu^{2}}{2 \mu  \,\beta^{2}}\\
&&~~~~~~~~~~~~~~~~~~~~~~+{\bf{i}}\,\frac{\left(\left(-\mu^{2}+\left(-2-3 \lambda \right) \mu +3 \lambda \right) \beta^{2}+\left(\left(\lambda +3\right) \mu +2 \lambda \right) \mu  \beta -\lambda  \,\mu^{2}\right) \sqrt{3}}{2 \mu  \,\beta^{2}}=0,\\
&&{P_{E_2}}(-\frac{1}{2}-\frac{\sqrt{3}}{2}\,{\bf{i}})=\frac{\left(\left(2 \lambda +1\right) \mu^{2}+\left(-\lambda -4\right) \mu -\lambda \right) \beta^{2}+\left(\left(-7 \lambda -1\right) \mu^{2}+4 \lambda  \mu \right) \beta +5 \lambda  \,\mu^{2}}{2 \mu  \,\beta^{2}}\\
&&~~~~~~~~~~~~~~~~~~~~~~+{\bf{i}}\,\frac{\left(\left(\mu^{2}+\left(2+3 \lambda \right) \mu -3 \lambda \right) \beta^{2}-\left(\left(\lambda +3\right) \mu +2 \lambda \right) \mu  \beta +\lambda  \,\mu^{2}\right) \sqrt{3}}{2 \mu  \,\beta^{2}}=0,\\
&&\bigtriangleup({P_{E_2}}(t))=\frac{\left(\mu^{2}+\left(-4 \beta^{2}+4 \beta \right) \mu +4 \beta^{2}\right) {\left(\lambda\left(\left(\lambda -2\right) \beta^{2}+\left(-2 \lambda +1\right) \beta +\lambda \right) \mu^{2}+\beta^{2} \lambda^{2}\right)}^{2}}{\beta^{6} \mu^{4}}\\
&&~~~~~~~~~~~~~~+\frac{\left(\mu^{2}+\left(-4 \beta^{2}+4 \beta \right) \mu +4 \beta^{2}\right) {\left(\left(\beta^{2}+\left(\lambda -2\right) \beta -\lambda \right) \mu^{3}\right)}^{2}}{\beta^{6} \mu^{4}}\\
&&~~~~~~~~~~~~~~+\frac{\left(\mu^{2}+\left(-4 \beta^{2}+4 \beta \right) \mu +4 \beta^{2}\right) \left(-2\lambda\left(\left(\lambda -1\right) \beta -\lambda \right) \beta  \mu \right)^{2}}{\beta^{6} \mu^{4}}<0,\\
&&\lambda >0,\mu >0,\beta >0,\frac{\beta \mu -\beta -\mu}{\beta \mu}\ge 0,
\end{eqnarray*}
\end{small}
are satisfied, a $1:3$ resonance may occur near $E_2$. Employing the polynomial complete discrimination system theory to solve the corresponding semi-algebraic system
\begin{small}
\begin{eqnarray*}
\!\!\!\!\!\!\!\!\!\!&&PS_3:=\left\{\lambda>0,\mu>0,\beta>0,
\frac{\beta \mu -\beta -\mu}{\beta \mu}\ge 0,
{P_{E_2}}(-\frac{1}{2}+\frac{\sqrt{3}}{2}\,{\bf{i}})=0,\right.
\\
\!\!\!\!\!\!\!\!\!\!&&~~~~~~~~~~~~\left.{P_{E_2}}(-\frac{1}{2}-\frac{\sqrt{3}}{2}\,{\bf{i}})=0,
\bigtriangleup({P_{E_2}}(t))<0\right\},
\end{eqnarray*}
\end{small}we get
$$\mu=7~~\mbox{and}~~\beta=\frac{7}{3}.$$
Consequently, the following conclusions are obtained.
\begin{thm}
System \eqref{eq2.1} undergoes a $1:3$ resonance as the parameter $\Lambda\in \mathbb{R}_+^{3}$ crosses $\{(\lambda,\mu,\beta)|\lambda>0,\mu=7,\beta=7/3\}$. Specifically, in a sufficiently small neighborhood of the $1:3$ resonance region, the following bifurcation phenomena occur:
\begin{description}
    \item[${\rm (i)}$]System $\eqref{eq2.1}$ always has a fixed point $E_2$. As the parameter $\Lambda$ crosses
    $$
    \mathfrak{L}^{''}_3=\left\{(\lambda,\mu,\beta)\bigg|\lambda>0,
    \mu=\frac{\beta}{\beta-2}+O\left(\left|\left(\beta-\frac{7}{3}\right)\right|^2\right)
    ,\beta>\frac{9}{4}\right\}
    $$
    from the side $\mu <\beta/(\beta-2)+O(|(\beta-7/3|^2)$ to the side $\mu >\beta/(\beta-2)+O(|(\beta-7/3)|^2)$, system \eqref{eq2.1} undergoes a nondegenerate supercritical Neimark-Sacker bifurcation and produces a stable invariant $\Gamma$ circle surrounding $E_2$.
    \item[${\rm (ii)}$]System $\eqref{eq2.1}$ always has a saddle cycle of period-three $\{T_{11}, T_{12},T_{13}\}$.
    \item[${\rm (iii)}$]As the parameter $\Lambda$ crosses
    \begin{footnotesize}
    \begin{eqnarray*}
    H_{3r}:=\left\{(\lambda,\mu,\beta)\bigg|\lambda>0,\mu=-\frac{9 \beta}{4}+\frac{1001}{80}-\frac{21 \sqrt{120 \beta -279}}{80}+O\left(\left|\left(\beta-\frac{7}{3}\right)\right|^2,\beta>\frac{9}{4}\right)
    \right\}
    \end{eqnarray*}
    \end{footnotesize}from the region $\mu <-9 \beta/4+1001/80-21 \sqrt{120 \beta -279}/80+O\left(\left|\left(\beta-7/3\right)\right|^2\right)$ to $\mu >-9 \beta/4+1001/80-21 \sqrt{120 \beta -279}/80+O\left(\left|\left(\beta-7/3\right)\right|^2\right)$, the invariant circle disappears, but the saddle cycle of period-three $\{T_{11}, T_{12},T_{13}\}$ still exists.
\end{description}
\label{th1-3}
\end{thm}

\begin{proof}
Translating $E_2$ to the origin $O$ with the transformation
$$(x,y,z)=\left(m_2+\frac{1}{\beta},n_2+\frac{\beta  \mu -\beta -\mu}{\beta \mu},l_2\right),$$
and taking $\hat{\varepsilon}:=(\hat{\varepsilon}_1, \hat{\varepsilon}_2)=(\mu-7, \beta-7/3)$ as a new parameter, system \eqref{eq2.1} becomes the mapping
\begin{eqnarray}
\left(
\begin{array}{l}
m_2
\\
n_2
\\
l_2
\end{array}
\right)
\mapsto
\left(
\begin{array}{l}
-\frac{\left(3 \hat{\varepsilon}_{1}-3 \hat{\varepsilon}_{2}+14\right) m_2}{3 \hat{\varepsilon}_{2}+7}-\frac{3 \left(\hat{\varepsilon}_{1}+7\right) n_2}{3 \hat{\varepsilon}_{2}+7}-\frac{3 \left(\hat{\varepsilon}_{1}+7\right) l_2}{3 \hat{\varepsilon}_{2}+7}\\
-\left(\hat{\varepsilon}_{1}+7\right) m_2^{2}-\left(\hat{\varepsilon}_{1}+7\right) n_2 m_2-\left(\hat{\varepsilon}_{1}+7\right) l_2 m_2
\\
\frac{m_2 \left(3 \hat{\varepsilon}_{2} \hat{\varepsilon}_{1}+4 \hat{\varepsilon}_{1}+18 \hat{\varepsilon}_{2}+21\right)}{3 \left(\hat{\varepsilon}_{1}+7\right)}+n_2-\frac{l_2 \left(3 \hat{\varepsilon}_{2} \hat{\varepsilon}_{1}+4 \hat{\varepsilon}_{1}+18 \hat{\varepsilon}_{2}+21\right)}{3 \left(\hat{\varepsilon}_{1}+7\right)}\\
+\frac{\left(3 \hat{\varepsilon}_{2}+7\right) m_2 n_2}{3}-\frac{\left(3 \hat{\varepsilon}_{2}+7\right) n_2 l_2}{3}
\\
\frac{\lambda  \left(3 \hat{\varepsilon}_{2} \hat{\varepsilon}_{1}+4 \hat{\varepsilon}_{1}+18 \hat{\varepsilon}_{2}+21\right) l_2}{\left(3 \hat{\varepsilon}_{2}+7\right) \left(\hat{\varepsilon}_{1}+7\right)}+\lambda  n_2 l_2\\
\end{array}
\right).
\label{1.3.1}
\end{eqnarray}
Using the idea done in the proof of Theorem \ref{th3.1}, mapping \eqref{1.3.1} is restricted to a two dimensional $C^{2}$ center manifold
\begin{eqnarray*}
l_2=h_{7}(m_2,n_2)=O(|(m_2,n_2)|^3),
\end{eqnarray*}
and then it is translated into the subsequent two dimensional form
\begin{eqnarray}
\left[
\begin{array}{ccc}
   m_2  \\
   \noalign{\medskip}
   n_2
\end{array}
\right]
\mapsto
\left[\begin{array}{cc}
-\frac{\left(14-3 \hat{\varepsilon}_{2}+3 \hat{\varepsilon}_{1}\right) m_2}{3 \hat{\varepsilon}_{2}+7}-\frac{3 \left(\hat{\varepsilon}_{1}+7\right) n_2}{3 \hat{\varepsilon}_{2}+7}-m^{2}_2 \left(\hat{\varepsilon}_{1}+7\right)-n_2 m_2 \left(\hat{\varepsilon}_{1}+7\right)\\
\noalign{\medskip}
 \frac{m_2 \left(3 \hat{\varepsilon}_{1} \hat{\varepsilon}_{2}+4 \hat{\varepsilon}_{1}+18 \hat{\varepsilon}_{2}+21\right)}{3 \left(\hat{\varepsilon}_{1}+7\right)}+n_2+\frac{\left(3 \hat{\varepsilon}_{2}+7\right) m_2 n_2}{3}
 \end{array}\right].
\label{1.3.2}
\end{eqnarray}
Employing the transformation
\begin{eqnarray*}
\left(
\begin{array}{cc}
m_2 \\
n_2
\end{array}
\right)
\rightarrow
\left(
\begin{array}{cc}
\chi_{31}&\chi_{32}\\
1&0
\end{array}
\right)
\left(
\begin{array}{l}
m_3
\\
n_3
\end{array}
\right),
\end{eqnarray*}
where
\begin{eqnarray*}
&&\chi_{31}=-\frac{9 \hat{\varepsilon}_{1}^{2}+126 \hat{\varepsilon}_{1}+441}{18 \hat{\varepsilon}_{1} \hat{\varepsilon}_{2}^{2}+66 \hat{\varepsilon}_{1} \hat{\varepsilon}_{2}+108 \hat{\varepsilon}_{2}^{2}+56 \hat{\varepsilon}_{1}+378 \hat{\varepsilon}_{2}+294},\\
&&\chi_{32}=-\frac{3 \sqrt{36 \hat{\varepsilon}_{1} \hat{\varepsilon}_{2}^{2}-9 \hat{\varepsilon}_{1}^{2}+132 \hat{\varepsilon}_{1} \hat{\varepsilon}_{2}+216 \hat{\varepsilon}_{2}^{2}-14 \hat{\varepsilon}_{1}+756 \hat{\varepsilon}_{2}+147}\, \left(\hat{\varepsilon}_{1}+7\right)}{2 \left(3 \hat{\varepsilon}_{1} \hat{\varepsilon}_{2}+4 \hat{\varepsilon}_{1}+18 \hat{\varepsilon}_{2}+21\right) \left(3 \hat{\varepsilon}_{2}+7\right)},\\
\end{eqnarray*}
system \eqref{1.3.2} turns into the mapping
\begin{eqnarray}
\left(
\begin{array}{l}
m_3 \\
n_3
\end{array}
\right)
\rightarrow
\left(
\begin{array}{l}
k_{11}m_3+k_{12}n_3+p_{20}m^2_3+p_{11}m_3n_3+p_{02}n^2_3+p_{30}m^3_3\\
+p_{21}m^2_3n_3+p_{12}m_3n^2_3+p_{12}m_3n^2_3+p_{03}n^3_3+O(|(m_3,n_3)|^4)\\
k_{21}m_3+k_{22}n_3+q_{20}m^2_3+q_{11}m_3n_3+q_{02}n^2_3+q_{30}m^3_3\\
+q_{21}m^2_3n_3+q_{12}m_3n^2_3+q_{12}m_3n^2_3+q_{03}n^3_3+O(|(m_3,n_3)|^4)\\
\end{array}
\right),
\label{1.3.3}
\end{eqnarray}
where
\begin{eqnarray*}
&&k_{11}=k_{22}=\frac{6\hat{\varepsilon}_{2}-3\hat{\varepsilon}_{1}-7}{6\hat{\varepsilon}_{2}+14},\\
&&k_{12}=-k_{21}=-\frac{\sqrt{-9 \hat{\varepsilon}_{1}^{2}+\left(36 \hat{\varepsilon}_{2}^{2}+132 \hat{\varepsilon}_{2}-14\right) \hat{\varepsilon}_{1}+216 \hat{\varepsilon}_{2}^{2}+756 \hat{\varepsilon}_{2}+147}}{6 \hat{\varepsilon}_{2}+14},\\
&&p_{30}=p_{02}=p_{03}=p_{21}=p_{12}=O(|(\hat{\varepsilon}_{1},\hat{\varepsilon}_{2})|^2),\\
&&p_{20}=-\frac{7}{2}-\frac{\hat{\varepsilon}_{1}}{3}+3 \hat{\varepsilon}_{2}+O(|(\hat{\varepsilon}_{1},\hat{\varepsilon}_{2})|^2),\\
&&p_{11}=-\frac{7 \sqrt{3}}{6}+\frac{\sqrt{3}\, \hat{\varepsilon}_{1}}{9}-2 \sqrt{3}\, \hat{\varepsilon}_{2}+O(|(\hat{\varepsilon}_{1},\hat{\varepsilon}_{2})|^2),\\
&&q_{30}=q_{21}=q_{12}=q_{03}=O(|(\hat{\varepsilon}_{1},\hat{\varepsilon}_{2})|^2),\\
&&q_{20}=7 \sqrt{3}+\frac{19 \sqrt{3}\, \hat{\varepsilon}_{1}}{6}-\frac{69 \sqrt{3}\, \hat{\varepsilon}_{2}}{2}+O(|(\hat{\varepsilon}_{1},\hat{\varepsilon}_{2})|^2),\\
&&q_{11}=\frac{35}{2}+\frac{13 \hat{\varepsilon}_{1}}{3}-30 \hat{\varepsilon}_{2}+O(|(\hat{\varepsilon}_{1},\hat{\varepsilon}_{2})|^2),\\
&&q_{02}=\frac{7 \sqrt{3}}{2}+\frac{\sqrt{3}\, \hat{\varepsilon}_{1}}{6}+\frac{9 \sqrt{3}\, \hat{\varepsilon}_{2}}{2}+O(|(\hat{\varepsilon}_{1},\hat{\varepsilon}_{2})|^2).\\
\end{eqnarray*}
Let $z_1:=m_1+{\bf{i}}\,n_1$, mapping \eqref{1.3.3} can be represented in the following complex form
\begin{small}
\begin{eqnarray}
\!\!\!\!\!\!\!\!\!z_1\mapsto\zeta(\hat{\varepsilon})z_1+t_{20}z_1^2+t_{11}z_1\bar{z_1}+t_{02}\bar{z_1}^2+t_{30}z_1^3+t_{21}z_1^2\bar{z_1}+t_{12}z_1\bar{z_1}^2+t_{03}\bar{z_1}^3+O(|z_1|^4),
\label{1.3.4}
\end{eqnarray}
\end{small}
where
\begin{eqnarray*}
&&\!\!\!\!\!\!\!\!\zeta(\hat{\varepsilon})=\frac{\sqrt{3}\, \left(3 \,{\bf{i}}-\sqrt{3}\right)}{6}-\frac{\sqrt{3}\, \left(3 \sqrt{3}+{\bf{i}}\right) \hat{\varepsilon}_{1}}{42}+\frac{3 \sqrt{3}\, \left(\sqrt{3}+5 \,{\bf{i}}\right) \hat{\varepsilon}_{2}}{14}+O(|(\hat{\varepsilon}_{1},\hat{\varepsilon}_{2})|^2),\\
&&\!\!\!\!\!\!\!\!t_{20}=\frac{7 \sqrt{3}\, \left(\sqrt{3}+{\bf{i}}\right)}{6}+\frac{\sqrt{3}\, \left(6 \sqrt{3}+13 \,{\bf{i}}\right) \hat{\varepsilon}_{1}}{18}-\frac{\sqrt{3}\, \left(9 \sqrt{3}+37 \,{\bf{i}}\right) \hat{\varepsilon}_{2}}{4}+O(|(\hat{\varepsilon}_{1},\hat{\varepsilon}_{2})|^2),\\
&&\!\!\!\!\!\!\!\!t_{11}=\frac{7 \left(-\sqrt{3}+9 \,{\bf{i}}\right) \sqrt{3}}{12}+\frac{\sqrt{3}\, \left(-\sqrt{3}+30 \,{\bf{i}}\right) \hat{\varepsilon}_{1}}{18}-\frac{\sqrt{3}\, \left(-\sqrt{3}+30 \,{\bf{i}}\right) \hat{\varepsilon}_{2}}{2}+O(|(\hat{\varepsilon}_{1},\hat{\varepsilon}_{2})|^2),\\
&&\!\!\!\!\!\!\!\!t_{02}=\frac{7 \sqrt{3}\, \left({\bf{i}}-3 \sqrt{3}\right)}{12}+\frac{7 \sqrt{3}\, \left(-\sqrt{3}+2 \,{\bf{i}}\right) \hat{\varepsilon}_{1}}{18}-\frac{\sqrt{3}\, \left(-11 \sqrt{3}+41 \,{\bf{i}}\right) \hat{\varepsilon}_{2}}{4}+O(|(\hat{\varepsilon}_{1},\hat{\varepsilon}_{2})|^2),\\
&&\!\!\!\!\!\!\!\!t_{30}=t_{21}=t_{12}=t_{03}=O(|(\hat{\varepsilon}_{1},\hat{\varepsilon}_{2})|^2).\\
\end{eqnarray*}

According to \cite[Lemma 9.12, p.448]{Kuznetsov}, for sufficiently small $|\hat{\varepsilon}|$ we can utilize a near-identity transformation to convert mapping \eqref{1.3.4} into the form
\begin{eqnarray}
w\mapsto\Gamma_{\hat{\varepsilon}}(w):=\zeta(\hat{\varepsilon})w+B(\hat{\varepsilon})\bar{w}^2+A(\hat{\varepsilon})w^2\bar{w}+O(|w|^4),
\label{1.3.5}
\end{eqnarray}
where
\begin{eqnarray*}
&&B(\hat{\varepsilon}):=\frac{g_{02}}{4},\\
&&A(\hat{\varepsilon}):=\frac{g_{20}g_{11}(2\zeta(\hat{\varepsilon})+\bar{\zeta}(\hat{\varepsilon})-3)}{4\,(\bar{\zeta}(\hat{\varepsilon})-1)(\zeta^2(\hat{\varepsilon})-\zeta(\hat{\varepsilon}))}
+\frac{|g_{11}|^2}{1-\bar{\zeta}(\hat{\varepsilon})}+\frac{g_{21}}{4}.
\end{eqnarray*}
Then it follows from \cite[Lemma 9.13, p. 450]{Kuznetsov} that the third iterate of the mapping \eqref{1.3.5} can be expressed as
\begin{eqnarray*}
\Gamma_{\hat{\varepsilon}}^3(w)=\phi(1,w,\hat{\varepsilon})+O(|w|^4),
\end{eqnarray*}
here $\phi(1,w,\hat{\varepsilon})$ denotes the time-1 flow of the planar differential system
\begin{eqnarray}
\begin{aligned}
\frac{dw}{dt}=\varrho(\hat{\varepsilon})w+B_1(\hat{\varepsilon})\bar{w}^2+A_1(\hat{\varepsilon})w^2\bar{w}
\end{aligned},
\label{1.3.6}
\end{eqnarray}
where
\begin{eqnarray*}
&&\varrho(\hat{\varepsilon})=\left(\frac{3}{7}+\frac{2 \sqrt{3}\,{\bf{i}}}{7}\right)\hat{\varepsilon}_1+\left(\frac{54}{7}-\frac{27 \sqrt{3}\,{\bf{i}}}{7}\right)\hat{\varepsilon}_2+O(|\hat{\varepsilon}|^2),\\
&&B_1(\hat{\varepsilon})=3\bar{\zeta}(\hat{\varepsilon})B(\hat{\varepsilon}),\\
&&A_1(\hat{\varepsilon})=-3|B(\hat{\varepsilon})|^2+3\zeta^2(\hat{\varepsilon})A(\hat{\varepsilon}).
\end{eqnarray*}

Let $\varrho(\hat{\varepsilon}):=\beta_1+\beta_2{\bf{i}}$, where
\begin{eqnarray}
\begin{cases}
\begin{array}{l}
\beta_1:=\frac{3\,\hat{\varepsilon}_1}{7}+\frac{54\,\hat{\varepsilon}_2}{7},\\
\beta_2:=\frac{2 \sqrt{3}\,\hat{\varepsilon}_1}{7}-\frac{27 \sqrt{3}\,\hat{\varepsilon}_2}{7}.
\end{array}
\end{cases}
\label{1.3.7}
\end{eqnarray}
One can check that
\begin{eqnarray*}
\det
\left.
\left(
\begin{array}{ll}
\frac{\partial\beta_1}{\partial\hat{\varepsilon}_1}&\frac{\partial\beta_1}{\partial\hat{\varepsilon}_2}\\
\frac{\partial\beta_2}{\partial\hat{\varepsilon}_1}&\frac{\partial\beta_2}{\partial\hat{\varepsilon}_2}
\end{array}
\right)
\right|_{(\hat{\varepsilon}_1,\hat{\varepsilon}_2)=(0,0)}
=-\frac{27 \sqrt{3}}{7}\neq0,
\end{eqnarray*}
implying that the resonance condition (R3.0) shown in \cite[p.451]{Kuznetsov} is satisfied. Let $\tilde{\beta}:=(\beta_1,\beta_2)\in\mathbb{R}$ as a new parameter, system \eqref{1.3.6} can be rewritten as
\begin{eqnarray}
\frac{dw}{dt}=(\beta_1+\beta_2\,{\bf{i}})w+b_1(\tilde{\beta})\bar{w}^2+c_1(\tilde{\beta})w^2\bar{w}+O(|\tilde{\beta}|^2|w|,|w|^4).
\label{1.3.8}
\end{eqnarray}
Taking $\tilde{\beta}=0$, we obtain
$$b_1(0)=\frac{21}{4}+\frac{7\sqrt{3}\,\bf{i}}{2}\neq0~~{\mbox{and}}~~\mathrm{Re}(c_1(0))=-\frac{1715}{16}\neq0,$$
implying that the conditions (R3.1) and (R3.2) in \cite[p.451]{Kuznetsov} are fulfilled. Consequently, system \eqref{eq2.1} undergoes a $1:3$ resonance.

Finally, we utilize the transformation $\omega=\gamma(\tilde{\beta})\eta$ with
\begin{eqnarray*}
\gamma(\tilde{\beta}):=\frac{\exp(\frac{\arg(b_1(\tilde{\beta}))}{3}\,{\bf{i}})}{|b_1(\tilde{\beta})|}
\end{eqnarray*}
to transform system \eqref{1.3.8} into
\begin{eqnarray}
\frac{d\eta}{dt}=((\beta_1+\beta_2\,{\bf{i}})\eta+\bar{\eta}^2+c(\tilde{\beta})\eta^2\bar{\eta}+O(|\tilde{\beta}|^2|\eta|,|\eta|^4),
\label{1.3.9}
\end{eqnarray}
where
$$
c(\tilde{\beta}):=\frac{c_1(\tilde{\beta})}{|b_1(\tilde{\beta})|^2}.
$$
Using the polar coordinates $\eta=\rho e^{{\bf{i}}\iota}$, we have the following approximating system of system \eqref{1.3.9}, given by
\begin{eqnarray}
\begin{cases}
\begin{array}{l}
\frac{d\rho}{dt}=\beta_1\rho+\rho^2\cos(3\iota)+R_c(\beta)\rho^3,\\
\frac{d\iota}{dt}=\beta_2-\rho \sin(3\iota)+I_c(\beta)\rho^2,
\end{array}
\end{cases}
\label{1.3.10}
\end{eqnarray}
where $R_c(\tilde{\beta})=\mathrm{Re}(c(\tilde{\beta}))$ and $I_c(\tilde{\beta})=\mathrm{Im}(c(\tilde{\beta}))$ are the real and imaginary parts of $c(\tilde{\beta})$, respectively.

From the fact $R_c(0)=-5/3<0$ and the theory of \cite[pp.451-454]{Kuznetsov}, we obtain the bifurcation diagrams of system \eqref{1.3.10}, as described in the following  (see Figure \ref{resonance1}).
\begin{figure}[htbp]
\centering
{
\includegraphics[width=10cm]{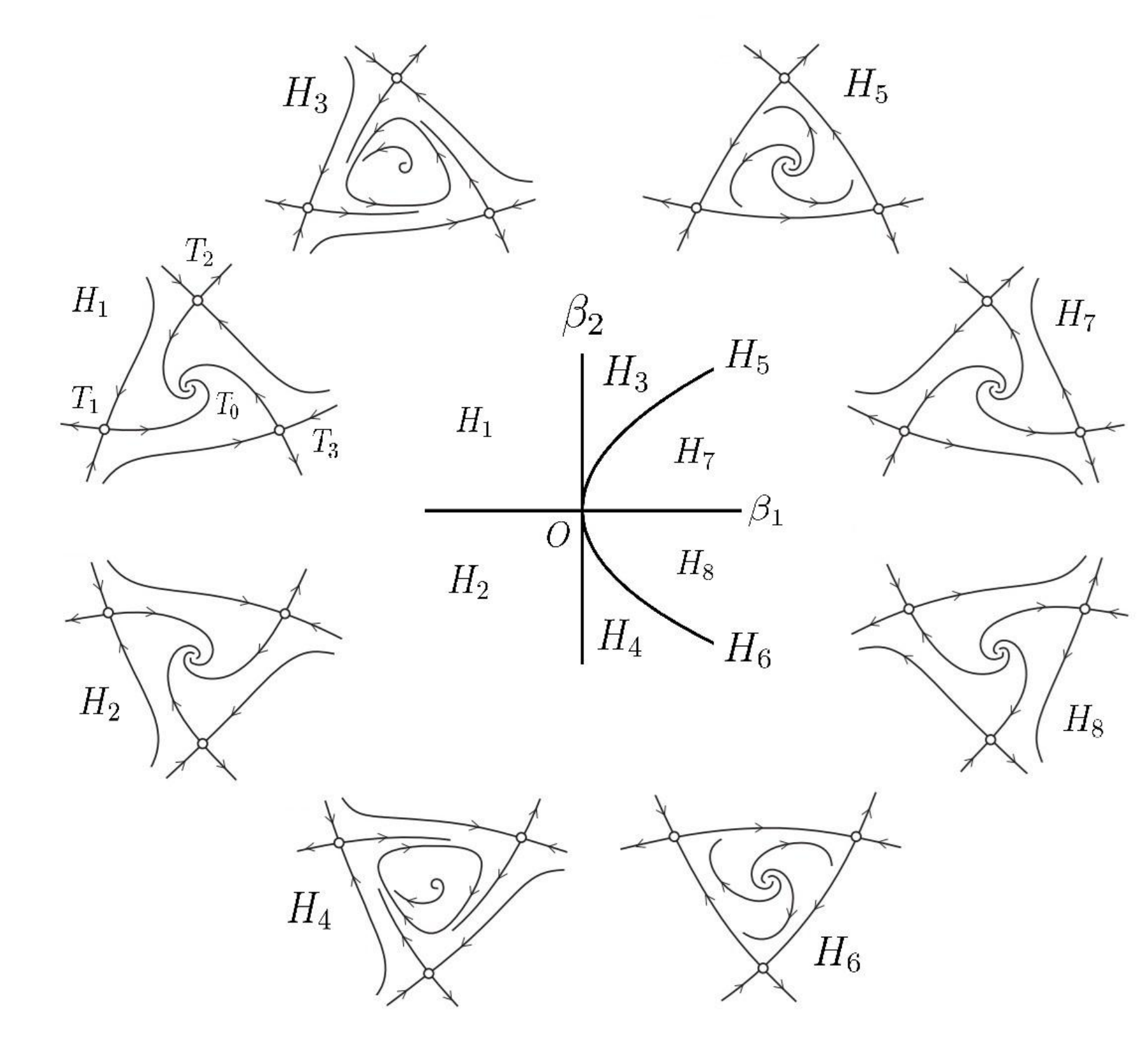}
}
\caption{Bifurcation diagram of system \eqref{1.3.10}}
\label{resonance1}
\end{figure}
\begin{description}
\item[(i)]As the parameter $(\beta_1, \beta_2)$ near the origin, system \eqref{1.3.10} possesses a trivial equilibrium point $T_0$ with $\rho=0$, which is a focus in the case $(\beta_1, \beta_2)\neq0$ and a degenerate saddle point in the case $(\beta_1, \beta_2)=0$. Additionally, there exist three nontrivial equilibrium points $T_k:(\rho_k,\iota_{s,k}),k = 1, 2, 3$, which are all saddle points, located on a circle with the radius of $r_s=\sqrt{\beta_1^2+\beta_2^2}+O(|\beta|^{\frac{3}{2}})$, and separated by the angle $2\pi/3$ in $\iota$-coordinate.

\item[(ii)]When $\beta_1<0$, the trivial equilibrium point $T_0$ is a stable focus. As $\tilde{\beta}$ crosses the $\beta_2$-axis from the left to right, $T_0$ becomes an unstable focus, system \eqref{1.3.10} undergoes a Hopf bifurcation, and produces a unique stable limit cycle.

\item[(iii)]There exists a bifurcation curve
    \begin{eqnarray*}
    H:=\left\{(\beta_1,\beta_2)\bigg|\beta_1=-\frac{R_c(0)\,\beta_2^2}{2}+o(\beta_2^2)\right\},
    \end{eqnarray*}
    which consists of $H_5$ and $H_6$ (see Figure \ref{resonance1}). Then system \eqref{1.3.10} has a heteroclinic cycle formed by the coincidence of the stable and unstable manifolds of the nontrivial saddles when $\tilde{\beta}\in H$. This heteroclinic cycle resembles a triangle and is stable from the inside.

\item[(iv)]As $\tilde{\beta}$ crosses the bifurcation curve $H$ from the left to right, the limit cycle disappears through a heteroclinic bifurcation.
\end{description}

According to \cite[pp.434]{Kuznetsov}, mapping $\Gamma_{\hat{\varepsilon}}$ is related to the differential equation \eqref{1.3.10} with the following properties: $\Gamma_{\hat{\varepsilon}}$ always possesses a trivial fixed point $O:(0,0)$, which corresponds to the trivial equilibrium point $T_0$ of \eqref{1.3.10}. $\Gamma_{\hat{\varepsilon}}$ undergoes a non-degenerate Neimark-Sacker bifurcation and generates an invariant circle surrounding the fixed point $T_0$ on a bifurcation curve corresponding to the Hopf bifurcation curve $\beta_2$-axis of \eqref{1.3.10}. The mapping $\Gamma_{\hat{\varepsilon}}$ also has a saddle cycle of period-three, which corresponds to the three nontrivial equilibrium points $T_k, k=1,2,3$ of \eqref{1.3.10}. If $({\hat{\varepsilon}}_1,{\hat{\varepsilon}}_2)$ lies on the curve corresponding to the bifurcation curve $H$ of \eqref{1.3.10}, then the saddle cycle of period-three lies in the invariant circle.

Near the $1:3$ resonance point, the system always has a period-three saddle cycle. When the parameter $\Lambda$ crosses $\mathfrak{L}^{''}_3$, the system undergoes a supercritical Neimark-Sacker bifurcation, producing a stable invariant circle. When the parameter approaches the curve $H_{3r}$, the period-three saddle cycle interacts with the invariant circle, and the invariant circle disappears, leaving only the period-three saddle cycle. These phenomena are manifested in ecological models as severe oscillations or chaos in population density, revealing the vulnerability of the food chain system at critical parameters.
\end{proof}
%we observe that system \eqref{eq2.1} experiences a $1:3$ resonance, leading to complex dynamic properties near $1:3$ resonance point $(7, 7/3)$ in Theorem \ref{th1-3}. Consequently, the intricate dynamic behavior of system \eqref{eq2.1} is not only dependent on bifurcation parameters, but also exhibits high sensitivity to parameter perturbations. From a biological viewpoint, as the parameter $\lambda>0,~(\mu,\beta)$ approaches the point $(7,7/3)$., the prey-predator populations may undergo periodic or quasi periodic fluctuations (caused by non degenerate Neimark-Sacker bifurcation), period-3 fluctuations (caused by the saddle cycle of period-3), long-period fluctuations, large-scale population outbreaks, and even chaotic behavior (caused by homoclinic structures).
%Next, we will focus on another codimension two bifurcation, i.e., $1:4$ resonance.

%%%%%%%%%%%%%%%%%%%%%%%%%%%%%%%%%%%%%%%%%%%%%%%%%%%%%%%%%%%%%%%%%%%%%%%%%%%%%%%%%%%%%%%%%%%%%%%%%
\subsection{$1:4$ resonance at $E_2$}
In this subsection, we discuss another type of codimension-2 bifurcation, i.e., $1:4$ resonance. Regarding system $(1.2)$, when the following conditions
\begin{small}
\begin{eqnarray*}
&&{P_{E_2}}({\bf{i}})=-2+\frac{\mu}{\beta}-2 \lambda +\lambda  \mu -\frac{3 \mu  \lambda}{\beta}+\frac{\lambda}{\mu}+\frac{3 \lambda}{\beta}+\frac{2 \mu  \lambda}{\beta^{2}}\\
&&~~~~~~~~~+{\bf{i}}\, \left(\frac{2 \lambda}{\mu}+1+\frac{\lambda}{\beta}+\frac{2 \mu}{\beta}-\mu +\frac{\mu  \lambda}{\beta}-2 \lambda -\frac{\mu  \lambda}{\beta^{2}}\right)=0,\\
&&{P_{E_2}}(-{\bf{i}})=-2+\frac{\mu}{\beta}-2 \lambda +\lambda  \mu -\frac{3 \mu  \lambda}{\beta}+\frac{\lambda}{\mu}+\frac{3 \lambda}{\beta}+\frac{2 \mu  \lambda}{\beta^{2}}\\
&&~~~~~~~~~+{\bf{i}}\,\left(-\frac{2 \lambda}{\mu}-1-\frac{\lambda}{\beta}-\frac{2 \mu}{\beta}+\mu -\frac{\mu  \lambda}{\beta}+2 \lambda +\frac{\mu  \lambda}{\beta^{2}}\right)=0,\\
&&\bigtriangleup({P_{E_2}}(t))=\frac{\left(\mu^{2}+\left(-4 \beta^{2}+4 \beta \right) \mu +4 \beta^{2}\right) {\left(\lambda\left(\left(\lambda -2\right) \beta^{2}+\left(-2 \lambda +1\right) \beta +\lambda \right) \mu^{2}+\beta^{2} \lambda^{2}\right)}^{2}}{\beta^{6} \mu^{4}}\\
&&~~~~~~~~~~~~~~+\frac{\left(\mu^{2}+\left(-4 \beta^{2}+4 \beta \right) \mu +4 \beta^{2}\right) {\left(\left(\beta^{2}+\left(\lambda -2\right) \beta -\lambda \right) \mu^{3}\right)}^{2}}{\beta^{6} \mu^{4}}\\
&&~~~~~~~~~~~~~~+\frac{\left(\mu^{2}+\left(-4 \beta^{2}+4 \beta \right) \mu +4 \beta^{2}\right) \left(-2\lambda\left(\left(\lambda -1\right) \beta -\lambda \right) \beta  \mu \right)^{2}}{\beta^{6} \mu^{4}}<0,\\
&&\lambda >0,\mu >0,\beta >0,\frac{\beta \mu -\beta -\mu}{\beta \mu}\ge 0,
\end{eqnarray*}
\end{small}
are satisfied, the matrix $JF(E_2)$ possesses a pair of complex eigenvalues equal to $\pm{\bf{i}}$. Hence, a $1:4$ resonance may occur near $E_2$. Further, solving the corresponding semi-algebraic system of above conditions, we obtain
%\begin{small}
%\begin{eqnarray*}
%&&\!\!\!\!\!\!\!\!\!\!PS_4:=\left\{\mu>0,\beta>0,\lambda>0,
%\frac{\beta \mu -\beta -\mu}{\beta \mu}\ge 0,
%{P_{E_2}}({\bf{i}})=0,
%{P_{E_2}}(-{\bf{i}})=0,
%\bigtriangleup({P_{E_2}}(t))<0\right\}.
%\end{eqnarray*}
%\end{small}By solving $PS_4$, we obtain
$$\mu=5~~\mbox{and}~~\beta=\frac{5}{2},$$
%Next, we will investigate $1:4$ resonance near $E_2$ under the above conditions.
and from which, we have

\begin{thm}
System \eqref{eq2.1} undergoes a $1:4$ resonance as the parameter $\Lambda\in \mathbb{R}_+^{3}$ crosses $\{(\lambda,\mu,\beta)|\lambda>0,\mu=5,\beta=5/2\}$. More specifically, in a sufficiently small neighborhood of the $1:4$ resonance region, the following bifurcation phenomena occur:
\begin{description}
    \item[${\rm (i)}$]System $\eqref{eq2.1}$ always has a fixed point $E_2$. As the parameter $\Lambda$ crosses
    $$
    \bar{\mathfrak{L}}_3:=\left\{(\lambda,\mu,\beta)\bigg|\lambda>0,\mu =\frac{\beta}{\beta-2}+O\left(\left|\left(\beta-\frac{5}{2}\right)\right|^2\right)
    ,\beta>\frac{9}{4}\right\}
    $$
    from the region $\mu<{\beta}/{\beta-2}+O(|(\beta-5/2)^2)$ to $\mu>{\beta}/{\beta-2}+O(|(\beta-5/2)^2)$, system $\eqref{eq2.1}$ undergoes a supercritical Neimark-Sacker bifurcation and generates a stable invariant circle $\Gamma$ surrounding $E_2$.
    \item[${\rm (ii)}$]System $\eqref{eq2.1}$ always has a saddle cycle of period-four $\{K_{11},K_{12},K_{13},K_{14}\}$.
    \item[${\rm (iii)}$]When the parameter $\Lambda$ lies in the homoclinic bifurcation region, the period-four saddle cycle lies in the invariant circle. If the parameter $\Lambda$ crosses the homoclinic bifurcation region from one side to the other side, the invariant circle disappears, but the saddle cycle of period-four $\{K_{11},K_{12},K_{13},K_{14}\}$ still exists.
\end{description}
\label{th1-4}
\end{thm}

\begin{proof}
Firstly, translating $E_2$ to the origin $O$ with the transformation
$$(x,y,z)=\left(m_4+\frac{1}{\beta},n_4+\frac{\beta  \mu -\beta -\mu}{\beta \mu},l_4\right),$$
and introducing $\tilde{\varepsilon}:=(\tilde{\varepsilon}_1, \tilde{\varepsilon}_2)=(\mu-5, \beta-5/2)$ as a new parameter, system \eqref{eq2.1} turns into the mapping
\begin{eqnarray}
\left(
\begin{array}{l}
m_4
\\
n_4
\\
l_4
\end{array}
\right)
\mapsto
\left(
\begin{array}{l}
-\frac{\left(2 \tilde{\varepsilon}_{1}-2 \tilde{\varepsilon}_{2}+5\right) m_4}{2 \tilde{\varepsilon}_{2}+5}-\frac{2 \left(\tilde{\varepsilon}_{1}+5\right) n_4}{2 \tilde{\varepsilon}_{2}+5}-\frac{2 \left(\tilde{\varepsilon}_{1}+5\right) l_4}{2 \tilde{\varepsilon}_{2}+5}\\
-\left(\tilde{\varepsilon}_{1}+5\right) m_4^{2}-\left(\tilde{\varepsilon}_{1}+5\right) n_4 m_4-\left(\tilde{\varepsilon}_{1}+5\right) l_4 m_4
\\
\frac{\left(2 \tilde{\varepsilon}_{2} \tilde{\varepsilon}_{1}+3 \tilde{\varepsilon}_{1}+8 \tilde{\varepsilon}_{2}+10\right) m_4}{2 \left(\tilde{\varepsilon}_{1}+5\right)}+n_4-\frac{\left(2 \tilde{\varepsilon}_{2} \tilde{\varepsilon}_{1}+3 \tilde{\varepsilon}_{1}+8 \tilde{\varepsilon}_{2}+10\right) l_4}{2 \left(\tilde{\varepsilon}_{1}+5\right)}\\
+\frac{m_4 n_4 \left(2 \tilde{\varepsilon}_{2}+5\right)}{2}-\frac{l_4 n_4 \left(2 \tilde{\varepsilon}_{2}+5\right)}{2}
\\
\frac{\lambda  \left(2 \tilde{\varepsilon}_{2} \tilde{\varepsilon}_{1}+3 \tilde{\varepsilon}_{1}+8 \tilde{\varepsilon}_{2}+10\right) l_4}{\left(2 \tilde{\varepsilon}_{2}+5\right) \left(\tilde{\varepsilon}_{1}+5\right)}+\lambda  n_4 l_4\\
\end{array}
\right).
\label{1.4.1}
\end{eqnarray}
As done in the proof of Theorem \ref{th3.1}, mapping \eqref{1.4.1} is restricted to a two dimensional $C^{2}$ center manifold
\begin{eqnarray*}
l_4=h_{8}(m_4,n_4)=O(|(m_4,n_4)|^3),
\end{eqnarray*}
and subsequently transformed to the following two dimensional form
\begin{eqnarray}
\left[
\begin{array}{ccc}
   m_4  \\
   \noalign{\medskip}
   n_4
\end{array}
\right]
\mapsto
\left[\begin{array}{cc}
-\frac{\left(2 \tilde{\varepsilon}_{1}-2 \tilde{\varepsilon}_{2}+5\right) m_4}{2 \tilde{\varepsilon}_{2}+5}-\frac{2 \left(\tilde{\varepsilon}_{1}+5\right) n_4}{2 \tilde{\varepsilon}_{2}+5}-m_4^{2} \left(\tilde{\varepsilon}_{1}+5\right)-n_4 m_4 \left(\tilde{\varepsilon}_{1}+5\right)\\
\noalign{\medskip}
 \frac{\left(2 \tilde{\varepsilon}_{1} \tilde{\varepsilon}_{2}+3 \tilde{\varepsilon}_{1}+8 \tilde{\varepsilon}_{2}+10\right) m_4}{2 \left(\tilde{\varepsilon}_{1}+5\right)}+n_4+\frac{m_4 n_4 \left(2 \tilde{\varepsilon}_{2}+5\right)}{2}
 \end{array}\right].
\label{1.4.2}
\end{eqnarray}

Next, employing the transformation
\begin{eqnarray*}
\left(
\begin{array}{cc}
m_4 \\
n_4
\end{array}
\right)
\rightarrow
\left(
\begin{array}{cc}
\chi_{41}&\chi_{42}\\
1&0
\end{array}
\right)
\left(
\begin{array}{l}
m_5
\\
n_5
\end{array}
\right),
\end{eqnarray*}
where
\begin{eqnarray*}
&&\chi_{41}=-\frac{2 \tilde{\varepsilon}_{1}^{2}+20 \tilde{\varepsilon}_{1}+50}{4 \tilde{\varepsilon}_{1} \tilde{\varepsilon}_{2}^{2}+16 \tilde{\varepsilon}_{1} \tilde{\varepsilon}_{2}+16 \tilde{\varepsilon}_{2}^{2}+15 \tilde{\varepsilon}_{1}+60 \tilde{\varepsilon}_{2}+50},\\
&&\chi_{42}=-\frac{2 \sqrt{4 \tilde{\varepsilon}_{1} \tilde{\varepsilon}_{2}^{2}-\tilde{\varepsilon}_{1}^{2}+16 \tilde{\varepsilon}_{1} \tilde{\varepsilon}_{2}+16 \tilde{\varepsilon}_{2}^{2}+5 \tilde{\varepsilon}_{1}+60 \tilde{\varepsilon}_{2}+25}\, \left(\tilde{\varepsilon}_{1}+5\right)}{\left(2 \tilde{\varepsilon}_{1} \tilde{\varepsilon}_{2}+3 \tilde{\varepsilon}_{1}+8 \tilde{\varepsilon}_{2}+10\right) \left(2 \tilde{\varepsilon}_{2}+5\right)},
\end{eqnarray*}
system \eqref{1.4.2} becomes
\begin{eqnarray}
\left(
\begin{array}{l}
m_4 \\
n_4
\end{array}
\right)
\rightarrow
\left(
\begin{array}{l}
x_{11}m_4+x_{12}n_4+x_{20}m^2_3+x_{11}m_4n_4+x_{02}n^2_3+x_{30}m^3_3\\
+x_{21}m^2_3n_4+x_{12}m_4n^2_3+x_{12}m_4n^2_3+x_{03}n^3_3+O(|(m_4,n_4)|^4)\\
y_{11}m_4+y_{12}n_4+y_{20}m^2_3+y_{11}m_4n_4+y_{02}n^2_3+y_{30}m^3_3\\
+y_{21}m^2_3n_4+y_{12}m_4n^2_3+y_{12}m_4n^2_3+y_{03}n^3_3+O(|(m_4,n_4)|^4)
\end{array}
\right),
\label{1.4.3}
\end{eqnarray}
where
\begin{eqnarray*}
&&x_{11}=y_{12}=\frac{2 \tilde{\varepsilon}_{2}-\tilde{\varepsilon}_{1}}{2 \tilde{\varepsilon}_{2}+5},\\
&&x_{12}=-y_{11}=-\frac{\sqrt{-\tilde{\varepsilon}_{1}^{2}+\left(4 \tilde{\varepsilon}_{2}^{2}+16 \tilde{\varepsilon}_{2}+5\right) \tilde{\varepsilon}_{1}+16 \tilde{\varepsilon}_{2}^{2}+60 \tilde{\varepsilon}_{2}+25}}{2 \tilde{\varepsilon}_{2}+5},\\
&&x_{30}=x_{02}=x_{03}=x_{21}=x_{12}=O(|(\tilde{\varepsilon}_{1},\tilde{\varepsilon}_{2})|^2),\\
&&x_{20}=-\frac{5}{2}-\frac{\tilde{\varepsilon}_{1}}{4}+2 \tilde{\varepsilon}_{2}+O(|(\tilde{\varepsilon}_{1},\tilde{\varepsilon}_{2})|^2),\\
&&x_{11}=-\frac{5}{2}-\tilde{\varepsilon}_{2}+O(|(\tilde{\varepsilon}_{1},\tilde{\varepsilon}_{2})|^2),\\
&&y_{30}=y_{21}=y_{12}=y_{03}=O(|(\tilde{\varepsilon}_{1},\tilde{\varepsilon}_{2})|^2),\\
&&y_{20}=\frac{5}{2}+\tilde{\varepsilon}_{1}-11 \tilde{\varepsilon}_{2}+O(|(\tilde{\varepsilon}_{1},\tilde{\varepsilon}_{2})|^2),\\
&&y_{11}=\frac{15}{2}+\frac{9 \tilde{\varepsilon}_{1}}{4}-14 \tilde{\varepsilon}_{2}+O(|(\tilde{\varepsilon}_{1},\tilde{\varepsilon}_{2})|^2),\\
&&y_{02}=\tilde{\varepsilon}_{1}+5+O(|(\tilde{\varepsilon}_{1},\tilde{\varepsilon}_{2})|^2).
\end{eqnarray*}
Let $z_2:=m_3+n_3\,{\bf{i}}$, mapping \eqref{1.4.3} can be rewritten in the complex form
\begin{small}
\begin{eqnarray}
z_2\mapsto\eta(\tilde{\varepsilon})z_2+l_{20}z_2^2+l_{11}z_2\bar{z_2}+l_{02}\bar{z_2}^2+l_{30}z_2^3+l_{21}z_2^2\bar{z_2}+l_{12}z_2\bar{z_2}^2+l_{03}\bar{z_2}^3+O(|z_2|^4),
\label{1.4.4}
\end{eqnarray}
\end{small}where
\begin{small}
\begin{eqnarray*}
&&\!\!\!\!\!\!\!\!\eta(\tilde{\varepsilon})={\bf{i}}+\left(-\frac{1}{5}+\frac{{\bf{i}}}{10}\right) \tilde{\varepsilon}_{1}+\left(\frac{2}{5}+\frac{4 \,{\bf{i}}}{5}\right) \tilde{\varepsilon}_{2}+O(|(\tilde{\varepsilon}_{1},\tilde{\varepsilon}_{2})|^2),\\
&&\!\!\!\!\!\!\!\!l_{20}=\frac{5}{2}+\tilde{\varepsilon}_{1}-\left(6+5 \,\bf{i}\right) \tilde{\varepsilon}_{2}+O(|(\tilde{\varepsilon}_{1},\tilde{\varepsilon}_{2})|^2),\\
&&\!\!\!\!\!\!\!\!l_{11}=-\frac{5}{4}+\frac{15 \,\bf{i}}{4}+\left(-\frac{1}{8}+\bf{i}\right) \tilde{\varepsilon}_{1}+\left(1-\frac{11 \,\bf{i}}{2}\right) \tilde{\varepsilon}_{2}+O(|(\tilde{\varepsilon}_{1},\tilde{\varepsilon}_{2})|^2),\\
&&\!\!\!\!\!\!\!\!l_{02}=-5-\frac{5 \,\bf{i}}{2}-\frac{5 \tilde{\varepsilon}_{1}}{4}+\left(8-6 \,\bf{i}\right) \tilde{\varepsilon}_{2}+O(|(\tilde{\varepsilon}_{1},\tilde{\varepsilon}_{2})|^2),\\
&&\!\!\!\!\!\!\!\!l_{30}=l_{21}=l_{12}=l_{03}=O(|(\tilde{\varepsilon}_{1},\tilde{\varepsilon}_{2})|^2).
\end{eqnarray*}
\end{small}

According to \cite[Lemma 9.14, p.455]{Kuznetsov}, for sufficiently small $|\tilde{\varepsilon}|$ we apply a near-identity transformation to transform mapping \eqref{1.4.4} into the form
\begin{eqnarray}
w_1\mapsto\Gamma_{\gamma}(w_1):=\zeta(\gamma)w_1+C(\gamma)\bar{w_1}^2+D(\gamma)w^2_1\bar{w_1}+O(|w_1|^4),
\label{1.4.5}
\end{eqnarray}
where
\begin{eqnarray*}
&&C(\gamma):=\frac{(1+3\,{\bf{i}})l_{20}l_{11}}{4}+\frac{(1-{\bf{i}})l_{11}\bar{l}_{11}}{2}
-\frac{(1+{\bf{i}})l_{02}\bar{l}_{02}}{4}+\frac{l_{21}}{2},\\
&&D(\gamma):=\frac{({\bf{i}}-1)l_{11}l_{02}}{4}-\frac{(1+{\bf{i}})l_{02}\bar{l}_{20}}{4}+\frac{l_{03}}{6}.
\end{eqnarray*}
Since $\zeta(0)=\exp(\frac{\pi\,{\bf{i}}}{2})= {\bf{i}}$, we can introduce a new parameter $\omega := (\omega_1,\omega_2)$, satisfying
\begin{eqnarray}
\zeta(0)^3\,\zeta(\gamma)=\exp(\omega_1+{\bf{i}}\,\omega_2).
\label{1.4.6}
\end{eqnarray}
Then, from \eqref{1.4.6} we obtain a mapping
\begin{eqnarray}
\gamma\mapsto \omega(\gamma)=(\omega_1(\gamma),\omega_2(\gamma)).
\label{1.4.7}
\end{eqnarray}
Note that
\begin{eqnarray*}
\det
\left.
\left(
\begin{array}{ll}
\frac{\partial\omega_1}{\partial\tilde{\varepsilon}_1}&\frac{\partial\omega_1}{\partial\tilde{\varepsilon}_2}\\
\frac{\partial\omega_2}{\partial\tilde{\varepsilon}_1}&\frac{\partial\omega_2}{\partial\tilde{\varepsilon}_2}
\end{array}
\right)
\right|_{(\tilde{\varepsilon}_1,\tilde{\varepsilon}_2)=(0,0)}
=-\frac{1}{5}\neq0,
\end{eqnarray*}
so the mapping \eqref{1.4.7} is regular at $0$. Hence, the new parameter $\omega$ can be used as the %new
unfolding parameter and we present its normal form as
\begin{eqnarray}
N_{\omega}:w_1\mapsto \exp\left(\frac{\pi\,{\bf{i}}}{2}+\omega_1+{\bf{i}}\,\omega_2\right)w_1+c(\omega)w_1\,\bar{w_1}^2+d(\omega)\,\bar{w_1}^3,
\label{1.4.8}
\end{eqnarray}
where $c$ and $d$ are smooth functions of $\omega$ such that $c(0) = C(0)$ and $d(0) = D(0)$.
Define a linear transformation $\mathcal{T} : \mathbb{C}\rightarrow \mathbb{C}$ by
\begin{eqnarray}
w_1\mapsto \mathcal{T}\,w_1:=\bar{\lambda}(0)\,w_1=\exp\left(\frac{\pi\,{\bf{i}}}{2}\right)\,w_1=-{\bf{i}}\,w_1,
\label{1.4.9}
\end{eqnarray}
which is a clockwise rotation through $\pi/2$ and \eqref{1.4.8} remains unchanged with respect to $\mathcal{T}$. Note that $\mathcal{T}^4 = {\rm id}$, implying the corresponding phase portraits will possess $\mathbb{Z}_4$-symmetry.

Based on the proof of \cite[Theorem 3.28, p.80]{Wiggins}, for sufficiently small $|\omega|$ mapping \eqref{1.4.8} fulfills
\begin{eqnarray*}
\mathcal{T}N_{\omega}(w_1)=\psi^1_{\omega}(w_1)+O(|w_1|^4),
\end{eqnarray*}
here $\psi^1_{\omega}$ is a flow of the following approximating planar system
\begin{eqnarray}
\dot{w_1}=(\omega_1+{\bf{i}}\,\omega_2)\,w_1+c_1(\omega)w_1\bar{w}_1^2+d_1(\omega)\bar{w}_1^3,~~w_1\in\mathbb{C},
\label{1.4.10}
\end{eqnarray}
where $c_1$ and $d_1$ are smooth complex-valued functions of $\omega$ such that
\begin{eqnarray*}
c_1(0)=\bar{\lambda}(0)\,c(0)=-{\bf{i}}\,c(0),~~~~~d_1(0)=\bar{\lambda}(0)\,d(0)=-{\bf{i}}\,d(0).
\end{eqnarray*}

Employing the scaling
\begin{eqnarray*}
w_1:=\frac{\exp(\frac{{\bf{i}}\,\arg(d_1(\omega))}{4})\,\xi_1}{\sqrt{|d_1(\omega)|}}
\end{eqnarray*}
and introducing
\begin{eqnarray*}
A(\omega):=\frac{c_1(\omega)}{|d_1(\omega)|},
\end{eqnarray*}
system \eqref{1.4.10} is transformed into the following differential equation
\begin{eqnarray}
\dot{w}=(\omega_1+{\bf{i}}\,\omega_2)\,w+A(\omega)w\bar{w}^2+\bar{w}^3,~~w\in\mathbb{C}.
\label{1.4.11}
\end{eqnarray}
Let $\upsilon:=\bar{w}$, system \eqref{1.4.11} becomes
\begin{eqnarray}
\dot{\upsilon}=(\omega_1+{\bf{i}}\,\omega_2)\,\upsilon+\bar{A}(\omega)\upsilon\bar{\upsilon}^2+\bar{\upsilon}^3,~~\upsilon\in\mathbb{C}
\label{1.4.12}.
\end{eqnarray}
Rewriting \eqref{1.4.12} in polar coordinate $\upsilon=\rho \exp({\bf{i}}\,\nu)$, we obtain
\begin{eqnarray}
\begin{array}{l}
\frac{d\,\rho}{dt}=\omega_1\,\rho+\mathrm{Re}(\bar{A}(\omega))\rho^3+\rho^3\cos(4\,\nu),\\
\frac{d\,\nu}{dt}=\omega_2+\mathrm{Im}(\bar{A}(\omega))\rho^2+\rho^2\sin(4\,\nu).
\end{array}
\label{1.4.13}
\end{eqnarray}
One can check that
\begin{eqnarray*}
a_0:=\mathrm{Re}(\bar{A}(0))=-\frac{8}{325},~b_0:=\mathrm{Im}(\bar{A}(0))=-\frac{4}{325}.
\end{eqnarray*}
Further, from \cite[pp.80-81]{Kuznetsov1} we see that the bifurcation diagram of system \eqref{1.4.13} in $(\omega_1,\omega_2)$-plane is dependent on $A_0=a_0+{\bf{i}}\,b_0$. The division of $A_0$-plane corresponding to different bifurcation diagrams is shown in Figure \ref{partion4-1}. Additionally, since $a_0=-8/325$ and $b_0=-4/325$, $A_0$ is located in region I in Figure \ref{partion4-1}. Therefore, it follows from \cite[p.83]{Kuznetsov1} that system \eqref{1.4.13} has the following bifurcations in a neighbourhood of $(\omega_1,\omega_2)=(0,0)$ for sufficiently small $|\omega|$ (see Figure \ref{rea4-2}).
\begin{figure}[htbp]
\centering
{
\includegraphics[width=10cm]{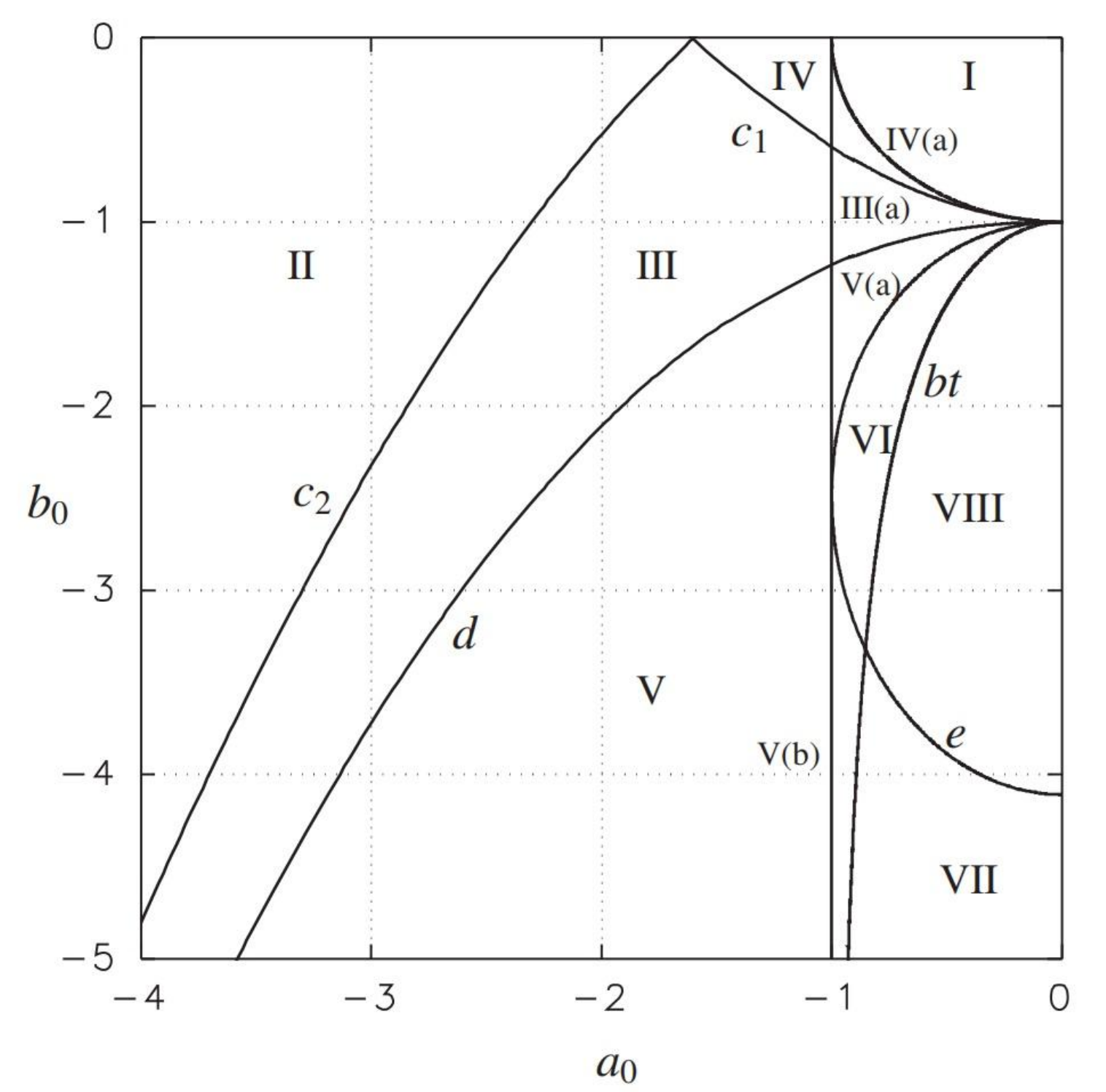}
}
\caption{Partitioning of $(a_0, b_0)$-plane of system \eqref{1.4.12}}
\label{partion4-1}
\end{figure}
\begin{figure}[htbp]
\centering
{
\includegraphics[width=10cm]{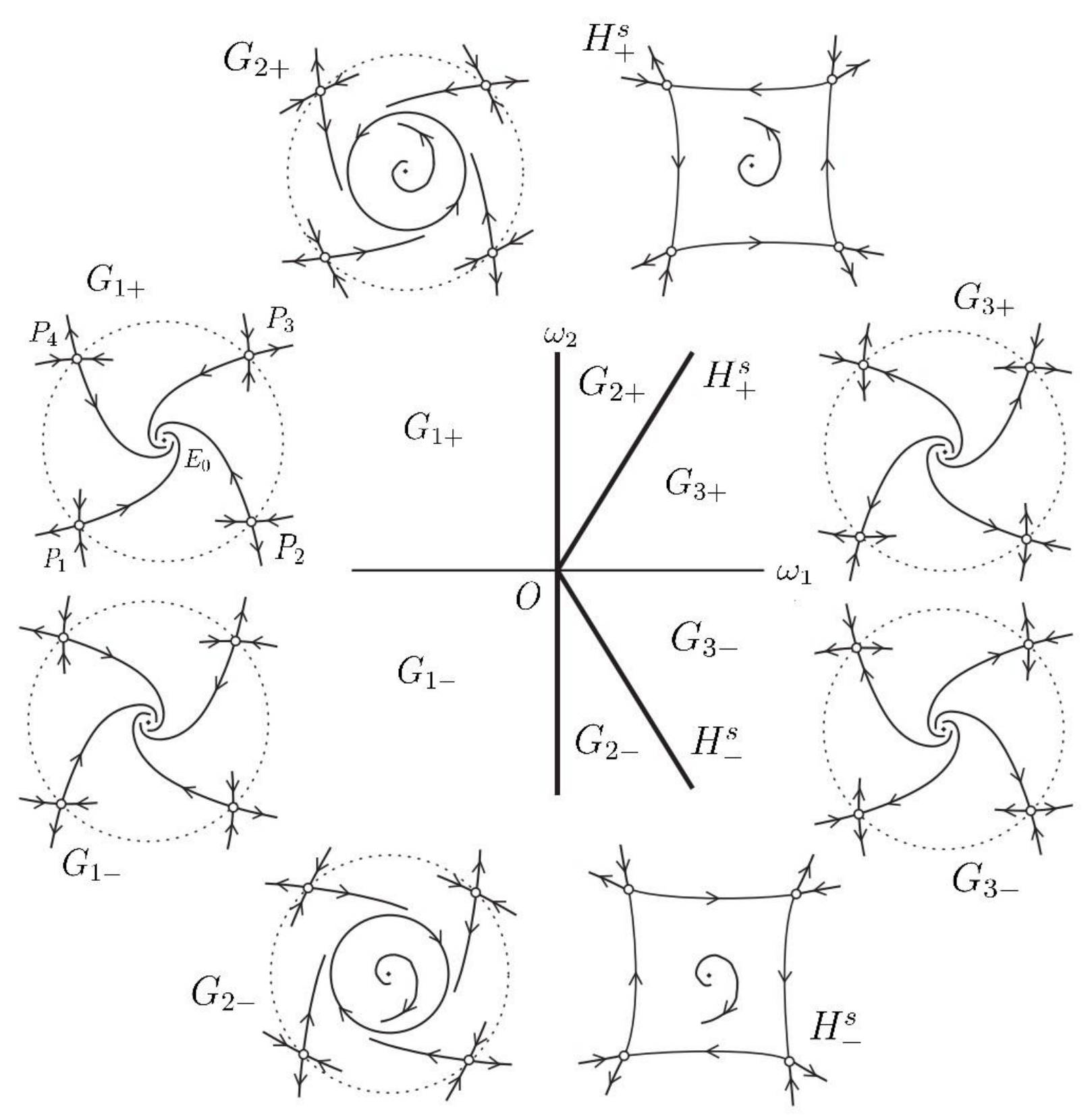}
}
\caption{Bifurcation diagram of system \eqref{1.4.13}}
\label{rea4-2}
\end{figure}
\begin{description}
\item[(i)]When $\omega=(\omega_1,\omega_2)$ lies in the regions $G_{1+}$ and $G_{1-}$ in Figure \ref{rea4-2}, system \eqref{1.4.13} possesses a trivial equilibrium point $E_0$ and four nontrivial saddles $P_{i}$ ($i=1,2,3,4$). As $\omega$ crosses the curve
    $$N=\{(\omega_1,\omega_2)\in\mathbb{R}^2:\omega_1=0\}$$
    from $G_{1+}$ and $G_{1-}$ to the regions $G_{2+}$ and $G_{2-}$, respectively, system \eqref{1.4.13} undergoes a Andronnov-Hopf bifurcation and produces a unique limit cycle in $G_{2+}$ and $G_{2-}$. Hence, the invariant circle and four saddle points coexist in $G_{2+}$ and $G_{2-}$.
\item[(ii)]When $\omega$ lies in the curves $H^{s}_{-}$ and $H^{s}_{+}$, the four saddle points $P_{i}$ ($i=1,2,3,4$) are in the invariant circle, which form a heteroclinic orbit.
\item[(iii)]When $\omega$ crosses $H^{s}_{+}$ and $H^{s}_{-}$ from $G_{2+}$ and $G_{2-}$ to $G_{3+}$ and $G_{3-}$, respectively, the limit cycle of system \eqref{1.4.13} disappears, i.e., the invariant circle disappears, but the four saddle points $P_{i}$ ($i=1,2,3,4$) still exist.
\end{description}

As demonstrated in \cite[p.458]{Kuznetsov}, the trivial equilibrium of system \eqref{1.4.13} correspond to the trivial fixed points of the mapping $N_\omega$, while the four symmetric nontrivial equilibrium points $P_{i}$ ($i=1,2,3,4$) correspond to a period-four cycle of $N_\omega$. The Hopf bifurcation in system \eqref{1.4.13} corresponds to the Neimark-Sacker bifurcation of $N_\omega$ and the heteroclinic connections in system \eqref{1.4.13} become heteroclinic structures of $N_\omega$. We complete the whole proof.

\end{proof}
%\begin{rmk}
%Here we do not know the analytical expression of curve $H^{s}_{-}$ and $H^{s}_{+}$, so we cannot give the analytical expression of curve $F$. Therefore, the curve $F$ only represents the approximate position in Figure \ref{resonance4}. However, in the next section, we will numerically simulate that the system \eqref{eq2.1} has a homoclinic structures under certain parameter conditions, i.e., the period-$4$ saddle cycle lie in the invariant circle $\Gamma$, It also imply that $F$ exist.
%\end{rmk}

%According to Theorem \ref{th1-4}, it is understood that the system has the capability to generate a $1:4$ resonance and exhibits intricate dynamic properties in close proximity to the $1:4$ resonance point $(1,1)$. Therefore, when the parameter $\lambda>0,~(\mu,\beta)$ varies near the point $(5,5/2)$, the dynamic properties of the system are highly sensitive to changes in the parameter.
%In the context of biology, non-degenerate Neimark-Sacker bifurcations can trigger periodic or quasi-periodic oscillations in prey-predator populations. Additionally, the presence of an invariant periodic orbit with a period-4 signifies that after four time intervals, the prey and predator populations shift from a steady state to a repeatable (approaching) state.

%%%%%%%%%%%%%%%%%%%%%%%%%%%%%%%%%%%%%%%%%%%%%%%%%%%%%%%%%%%%%%%%%%%%%%%%%%%%%%%%%%%%%%%%%%%%%%%%%%%%%%%%%%%%%%%%%%%%%%%%%%%%%%%%%%%%%%%%%%%%%%%
\section{Weak resonance Arnold tongues}
From Theorem \ref{nsE2}, it is proved that when $\lambda>0,~\mu=\beta/(\beta-2)$ and $\beta>9/4$, system \eqref{eq2.1} undergoes a Neimark-Sacker bifurcation and generates a unique attractive invariant circle. In this section, we will investigate the parameter conditions of Arnold tongue on the invariant circle, so that system \eqref{eq2.1} possesses period $m$ orbits for $m\geq 5$. As demonstrated in \cite{Arrowsmith,Whitley}, this phenomenon requires the eigenvalues $t_{1}$, $t_{2}$, shown in $\eqref{solu-1}$, cross the unit circle and take the form of $\exp(2\pi\,{\bf i}\,n/m)$, where $n$ is a positive integer with irreducible $n/m\in(0,1)$.

\begin{thm}
Suppose $m\geq5$ is an integer, $n/m\in(0,1)$ with $n\in \mathbb{Z}^+$ is a irreducible fraction. Assume that the parameter $(\mu,\beta)$ is near $(\mu_*,\beta_*)$ and lies within an Arnold tongue $\mathcal{A}_{n/m}$ given by
$$\mathcal{A}_{n/m}=\left\{(\mu,\beta)\bigg|T_{-}+\frac{2\pi n}{m}<\arctan\left(\frac{\sqrt{4\beta^{2} \mu-4\beta^{2}-4\beta\mu-\mu^{2}}}{2 \beta-\mu}\right)<T_{+}+\frac{2\pi n}{m}\right\},$$
where
$\mu_* =\beta_*/(\beta_*-2)$, $\beta_*=(4\cos(2\pi n/m) - 5)/(2(\cos(2\pi n/m) - 1))$, and
\begin{eqnarray*}
&&T_{\pm}\approx\frac{\tilde{\varrho}_{2}(0)}{\check{\varrho}_{3}(0)}\,\left(\sqrt{{\frac{\mu  \left(\beta -2\right)}{\beta}}}-1\right)\pm\frac{|\varsigma(0)|}{|\check{\varrho}_{3}(0)|^{(m-2)/2}}\,\left(\sqrt{{\frac{\mu  \left(\beta -2\right)}{\beta}}}-1\right)^{\frac{m-2}{2}}.
\end{eqnarray*}
Then system \eqref{eq2.1} possesses two orbits of period-$m$, one is attractive and the other is repellent on an invariant circle, produced from the Neimark-Sacker bifurcation.
\label{th5-1}
\end{thm}
\begin{proof}
For convenience, we write $t_{1}$ shown in $\eqref{solu-1}$ as the following exponential form
\begin{eqnarray}
t_{1}=(1+\chi_1(\mu,\beta))\,\exp({\bf i}\,(2\,\pi\,n/m+\chi_2(\mu,\beta))).
\label{a1}
\end{eqnarray}
It is clear that $\chi_1(\mu_*,\beta_*)=\chi_2(\mu_*,\beta_*)=0$. Based on the Poincar\'e normal form theory (see Lemma 2 in \cite[p.44]{Iooss} and \cite[p.259]{Arrowsmith}), for $(\mu,\beta)$ near $(\mu_*,\beta_*)$ system \eqref{eq2.1} has a normal form of complex type
\begin{eqnarray}
&&z\mapsto t_{1}z+\sum_{k=1}^{[(m-2)/2]}\varrho_{k+1,k}(\chi)\,z^{k+1}\,\bar{z}^{k}+\varsigma(\chi)\,\bar{z}^{m-1}+O(|z|^m),
\label{a2}
\end{eqnarray}
where $\chi:=(\chi_1(\mu,\beta),\chi_2(\mu,\beta))$, $[\cdot]$ denotes the integer part of $\cdot$, and $\varrho_{k+1,k}(\chi)$, $\varsigma(\chi)$ are both analytic functions of $\chi$. Let $z:=\check{\rho}\,\exp(2\pi{\bf i}\theta)$. Then system \eqref{a2} becomes the following polar coordinate form
\begin{eqnarray}
\begin{aligned}
\!\!\!\!\!\!\!\!\left(
\begin{array}{cc}
\check{\rho} \\
\theta
\end{array}
\right)\!\!
\rightarrow\!\!
\left(
\begin{array}{l}
(1+\chi_1)\,\check{\rho}+\sum_{k=1}^{[(m-2)/2]}\check{\varrho}_{2\,k+1}(\chi)\,\check{\rho}^{2\,k+1}
+\check{\varsigma}(\chi,\theta)\,\check{\rho}^{m-1}+O(\check{\rho}^{m})\\
\theta+\frac{2\,\pi\,n}{m}+\chi_2+\sum_{k=1}^{[(m-2)/2]}\tilde{\varrho}_{2\,k}(\chi)\,\check{\rho}^{2\,k+1}
+\tilde{\varsigma}(\chi,\theta)\,\check{\rho}^{m-2}+O(\check{\rho}^{m-1})
\end{array}
\right),
\end{aligned}
\label{a3}
\end{eqnarray}
where $\check{\varrho}_{2\,k+1}$, $\tilde{\varrho}_{2\,k}$ are both analytic functions independent of $\theta$, and
\begin{eqnarray*}
&&\check{\varsigma}(\chi,\theta)=\mathrm{Re}(\varsigma(\chi)\,\exp(-{\bf i}\,(m\,\theta+2\,\pi\,n/m+\chi_2))),\\
&&\tilde{\varsigma}(\chi,\theta)=\mathrm{Im}\left(\frac{\varsigma(\chi)\,\exp(-{\bf i}\,(m\,\theta+2\,\pi\,n/m+\chi_2))}{1+\chi_1}\right).
\end{eqnarray*}
One can compute that
\begin{eqnarray*}
&&\check{\varrho}_{3}(\chi)=\mathrm{Re}(\varrho_{2,1}\,\exp(-{\bf i}\,(2\,\pi\,n/m+\chi_2))),\\
&&\tilde{\varrho}_{2}(\chi)=\frac{\mathrm{Im}(\varrho_{2,1}\,\exp(-{\bf i}\,(2\,\pi\,n/m+\chi_2)))}{1+\chi_1}.
\end{eqnarray*}

By \cite [Theorem 2.4] {Whitley}, system \eqref{a3} has two orbits of period-$m$, one is attractive and the other is repellent on an invariant circle, produced from the Neimark-Sacker bifurcation if $\chi$ lies within a tongue $\mathcal{A}_{n/m}$ bounded by
\begin{eqnarray}
&&\chi_2\approx\frac{\tilde{\varrho}_{2}(0)\,\chi_1}{\check{\varrho}_{3}(0)}\pm\frac{|\varsigma(0)|\,\chi_1^{(m-2)/2}}{|\check{\varrho}_{3}(0)|^{(m-2)/2}}
\label{a4}
\end{eqnarray}
in the parameter $(\chi_1,\chi_2)$-plane. We further calculate
\begin{eqnarray*}
&&\check{\varrho}_{3}(0)=-\frac{\beta_{*}^{3}}{8 \left(\beta_{*}-2\right)^{3}}
=-\frac{\mu_{*}^{3}}{8},\\
&&\tilde{\varrho}_{2}(0)=\frac{\beta_{*}^{3} \left(2 \beta_{*}^{3}-8 \beta_{*}^{2}+3 \beta_{*}+11\right)}{8 \left(3 \beta_{*}-7\right) \sqrt{4 \beta_{*}-9}\, \left(\beta_{*}-2\right)^{3}}
=-\frac{\mu_{*}^{3} \left(\mu_{*}^{3}-13 \mu_{*}^{2}+39 \mu_{*}-11\right)}{8 \sqrt{\frac{-\mu_{*}+9}{\mu_{*}-1}}\, \left(\mu_{*}-1\right)^{2} \left(\mu_{*}-7\right)}.
\end{eqnarray*}
Besides, it follows from \eqref{a1} that $|t_{1}|=1+\chi_1$ and $\arg(t_{1})=2\pi n/m+\chi_2$. Making use of $\eqref{solu-1}$ again we have
\begin{eqnarray*}
&&|t_{1}|=\sqrt{\frac{\mu  \left(\beta -2\right)}{\beta}},\\
&&\arg(t_{1})=\arctan\left(\frac{\sqrt{4\beta^{2} \mu-4\beta^{2}-4\beta\mu-\mu^{2}}}{2 \beta-\mu}\right).
\end{eqnarray*}
Therefore, we get
\begin{eqnarray}
\chi_1=\sqrt{\frac{\mu  \left(\beta -2\right)}{\beta}}-1,
\label{a5}
\end{eqnarray}
\begin{eqnarray}
\chi_2=\arctan\left(\frac{\sqrt{4\beta^{2} \mu-4\beta^{2}-4\beta\mu-\mu^{2}}}{2 \beta-\mu}\right)-\frac{2\pi n}{m}.
\label{a6}
\end{eqnarray}
Substituting \eqref{a5} and \eqref{a6} into \eqref{a4}, we obtain $T_{\pm}$ and $\mathcal{A}_{n/m}$. The proof is completed.
\end{proof}
\begin{rmk}
The value $\varsigma(0)$ in Theorem \ref{th5-1} is the coefficient of the term $\bar{z}^{m-1}$ at $\chi=(0,0)$ in \eqref{a2}. Clearly, $\varsigma(\chi)$ varies with $m$ and in general the expression of $\varsigma(\chi)$ is simple when $m$ is small.
\end{rmk}

%%%%%%%%%%%%%%%%%%

\section{Existence of Marotto's Chaos}
\setcounter{equation}{0}
\allowdisplaybreaks[4]
In this section, we will investigate the conditions given in \cite{Marotto,Marotto1978,Salman} such that system \eqref{eq2.1} possesses chaotic behaviors in the sense of Marotto.
%we review some definitions and conclusions given in \cite{Marotto,Marotto1978,Salman}.

A fixed point $E$  is referred as an {\it expanding fixed point} (see \cite{Chen,Marotto,Salman}) of $f:\mathbb{R}^n\rightarrow \mathbb{R}^n$, where $f$ is differentiable in a neighborhood of $E$ (denoted as $U(E)$) if the modules of all eigenvalues of $Df(
x)$ exceed 1 for all $x\in U(E)$.
\begin{defi}{\rm(\cite{Marotto,Salman})}
Let $E$ be an expanding fixed point of $f$ in $U(E)$, and suppose that there exists a point $x_0\in U(E)\backslash \{E\}$ such that $x_M=E$ and $|Df(x_k)|\ne0$ for $1\leq k\leq M$, where $x_k=f^k(x_0)$. Then $E$ is called a {\it snap-back} repeller of $f$.
\label{defi1}
\end{defi}

\begin{lm}{\rm(\cite{Marotto,Salman})}
If $f$ possesses a snap-back repeller, then $f$  is chaotic.
\label{th7}
\end{lm}

Based on the above description, we first investigate the conditions so that $E_2$ is an expanding fixed point of system \eqref{eq2.1}. The characteristic equation at the fixed point $E_2: (x^*, y^*,z^*)$ can be formulated as
\begin{equation}
\lambda^3 + \mathcal{M}(x^*, y^*, z^*)\lambda^2 + \mathcal{N}(x^*, y^*, z^*)\lambda + \mathcal{W}(x^*, y^*, z^*) = 0,
\label{marotto1}
\end{equation}
with coefficients defined by
\begin{small}
\begin{eqnarray*}
&&\!\!\!\!\mathcal {M}=\left(2 x^*+y^*+z^*-1\right) \mu -\beta  \left(x^*-z^*\right)-\lambda  y^*,
\\
&&\!\!\!\!\mathcal {N}=\left(\left(y^* z^*-2 \left(x^*+\frac{z^*}{2}-\frac{1}{2}\right) \left(x^*-z^*\right)\right) \beta -2 \lambda  \mu y^* \left(x^*+\frac{y^*}{2}+\frac{z^*}{2}-\frac{1}{2}\right)\right) +\beta  \lambda  y^* x^*,
\\
&&\!\!\!\!\mathcal {W}=2 \lambda \mu \beta  y^* x^* \left(x^*+z^*-\frac{1}{2}\right).
\end{eqnarray*}
\end{small}Through the substitution $\mathcal {Z}=\lambda+\mathcal {M}/3$, we transform \eqref{marotto1} into its depressed cubic form:
\begin{eqnarray}
\mathcal {Z}^3+\left(\mathcal {N}-\frac{\mathcal {M}^2}{3} \right)\mathcal {Z}+\mathcal {W}+\frac{2\mathcal {M}^3}{27}-\frac{\mathcal {N}\mathcal {W}}{3}
:=
\mathcal {Z}^3+\mathscr{M}\mathcal {Z}+\mathscr{N}=0.
\label{marotto2}
\end{eqnarray}
Define the discriminant of the depressed cubic as
\begin{equation*}
\bar{\Delta} := \left(\frac{\mathscr{N}}{2}\right)^2 + \left(\frac{\mathscr{M}}{3}\right)^3.
\end{equation*}
For the case $\bar{\Delta}>0$, the depressed cubic equation \eqref{marotto2} possesses:
\begin{itemize}
\item A unique real root:
\begin{eqnarray*}
\mathcal {Z}_1=\sqrt[3]{-\frac{\mathscr{N}}{2}+\sqrt{\bar{\Delta}}}
+\sqrt[3]{-\frac{\mathscr{N}}{2}-\sqrt{\bar{\Delta}}}.
\end{eqnarray*}
\item A pair of complex conjugate roots:
\begin{eqnarray*}
\mathcal {Z}_2=-\frac{1}{2}\mathcal {Z}_1+\frac{\sqrt{3}}{2}{\bf{i}}\mathcal {Z}_1,~~~
\mathcal {Z}_3=-\frac{1}{2}\mathcal {Z}_1-\frac{\sqrt{3}}{2}{\bf{i}}\mathcal {Z}_1.
\end{eqnarray*}
\end{itemize}
Consequently, the original characteristic equation \eqref{marotto1} admits eigenvalues
\begin{eqnarray*}
\lambda_1=\mathcal {Z}_1-\frac{\mathcal {M}^2}{3},~~~\lambda_2=\mathcal {Z}_2-\frac{\mathcal {M}^2}{3},~~~\lambda_3=\mathcal {Z}_3-\frac{\mathcal {M}^2}{3}.
\end{eqnarray*}
The critical expanding condition requires all eigenvalues to satisfy $\lambda_i>1$, which consequently implies
\iffalse
\mathcal {M}(x^*, y^*,z^*)
\mathcal {N}(x^*, y^*,z^*)
\mathcal {W}(x^*, y^*,z^*)
\fi
\begin{small}
\begin{eqnarray*}
\left\{\begin{array}{ll}
\left(\frac{\mathcal {W}(x^*, y^*,z^*)}{2}+\frac{\mathcal {M}^3(x^*, y^*,z^*)}{27}-\frac{\mathcal {M}(x^*, y^*,z^*)\mathcal {N}(x^*, y^*,z^*)}{6}\right)^2+\left(\frac{\mathcal {N}(x^*, y^*,z^*)}{3}-\frac{\mathcal {M}^2(x^*, y^*,z^*)}{9}\right)^3>0,
\\
\left\{\frac{\mathcal {W}(x^*, y^*,z^*)}{2}+\frac{\mathcal {M}^3(x^*, y^*,z^*)}{27}-\frac{\mathcal {M}(x^*, y^*,z^*)\mathcal {N}(x^*, y^*,z^*)}{6}\right.\\
+\left.\left({\left(\frac{\mathcal {W}(x^*, y^*,z^*)}{2}+\frac{\mathcal {M}^3(x^*, y^*,z^*)}{27}-\frac{\mathcal {M}(x^*, y^*,z^*)\mathcal {N}(x^*, y^*,z^*)}{6}\right)^2+\left(\frac{\mathcal {N}(x^*, y^*,z^*)}{3}-\frac{\mathcal {M}^2(x^*, y^*,z^*)}{9}\right)^3}\right)^{\frac{1}{2}}\right\}^{\frac{1}{3}}\\
+\left\{\frac{\mathcal {W}(x^*, y^*,z^*)}{2}+\frac{\mathcal {M}^3(x^*, y^*,z^*)}{27}-\frac{\mathcal {M}(x^*, y^*,z^*)\mathcal {N}(x^*, y^*,z^*)}{6}\right.\\
-\left.\left({\left(\frac{\mathcal {W}(x^*, y^*,z^*)}{2}+\frac{\mathcal {M}^3(x^*, y^*,z^*)}{27}-\frac{\mathcal {M}(x^*, y^*,z^*)\mathcal {N}(x^*, y^*,z^*)}{6}\right)^2+\left(\frac{\mathcal {N}(x^*, y^*,z^*)}{3}-\frac{\mathcal {M}^2(x^*, y^*,z^*)}{9}\right)^3}\right)^{\frac{1}{2}}\right\}^{\frac{1}{3}}\\
>\max\left\{1+\frac{\mathcal {M}(x^*, y^*,z^*)}{3},-\frac{\mathcal {M}(x^*, y^*,z^*)}{6}+\sqrt{1-\frac{\mathcal {M}^2(x^*, y^*,z^*)}{12}}\right\}
\\
or
\\
\left\{\frac{\mathcal {W}(x^*, y^*,z^*)}{2}+\frac{\mathcal {M}^3(x^*, y^*,z^*)}{27}-\frac{\mathcal {M}(x^*, y^*,z^*)\mathcal {N}(x^*, y^*,z^*)}{6}\right.\\
+\left.\left({\left(\frac{\mathcal {W}(x^*, y^*,z^*)}{2}+\frac{\mathcal {M}^3(x^*, y^*,z^*)}{27}-\frac{\mathcal {M}(x^*, y^*,z^*)\mathcal {N}(x^*, y^*,z^*)}{6}\right)^2+\left(\frac{\mathcal {N}(x^*, y^*,z^*)}{3}-\frac{\mathcal {M}^2(x^*, y^*,z^*)}{9}\right)^3}\right)^{\frac{1}{2}}\right\}^{\frac{1}{3}}\\
+\left\{\frac{\mathcal {W}(x^*, y^*,z^*)}{2}+\frac{\mathcal {M}^3(x^*, y^*,z^*)}{27}-\frac{\mathcal {M}(x^*, y^*,z^*)\mathcal {N}(x^*, y^*,z^*)}{6}\right.\\
-\left.\left({\left(\frac{\mathcal {W}(x^*, y^*,z^*)}{2}+\frac{\mathcal {M}^3(x^*, y^*,z^*)}{27}-\frac{\mathcal {M}(x^*, y^*,z^*)\mathcal {N}(x^*, y^*,z^*)}{6}\right)^2+\left(\frac{\mathcal {N}(x^*, y^*,z^*)}{3}-\frac{\mathcal {M}^2(x^*, y^*,z^*)}{9}\right)^3}\right)^{\frac{1}{2}}\right\}^{\frac{1}{3}}\\
<\min\left\{\frac{\mathcal {M}(x^*, y^*,z^*)}{3}-1,-\frac{\mathcal {M}(x^*, y^*,z^*)}{6}-\sqrt{1-\frac{\mathcal {M}^2(x^*, y^*,z^*)}{12}}\right\}.
\end{array}\right.
\end{eqnarray*}
\end{small}
Define the parameter domains:
\begin{small}
\begin{eqnarray*}
&&D_1:=\left\{(x,y,z)\bigg|\left(\frac{\mathcal {W}(x, y,z)}{2}+\frac{\mathcal {M}^3(x, y,z)}{27}-\frac{\mathcal {M}(x, y,z)\mathcal {N}(x, y,z)}{6}\right)^2\right.\\
&&~~~~~+\left.\left(\frac{\mathcal {N}(x, y,z)}{3}-\frac{\mathcal {M}^2(x, y,z)}{9}\right)^3>0\right\},\\
&&D_2:=\left\{(x,y,z)\bigg|\left\{\frac{\mathcal {W}(x, y,z)}{2}+\frac{\mathcal {M}^3(x, y,z)}{27}-\frac{\mathcal {M}(x, y,z)\mathcal {N}(x, y,z)}{6}\right.\right.\\
&&~~~~~+\left(\left(\frac{\mathcal {W}(x, y,z)}{2}+\frac{\mathcal {M}^3(x, y,z)}{27}-\frac{\mathcal {M}(x, y,z)\mathcal {N}(x, y,z)}{6}\right)^2\right.\\
&&~~~~~+\left.\left.\left(\frac{\mathcal {N}(x, y,z)}{3}-\frac{\mathcal {M}^2(x, y,z)}{9}\right)^3\right)^{\frac{1}{2}}\right\}^{\frac{1}{3}}\\
&&~~~~~+\left\{\frac{\mathcal {W}(x, y,z)}{2}+\frac{\mathcal {M}^3(x, y,z)}{27}-\frac{\mathcal {M}(x, y,z)\mathcal {N}(x, y,z)}{6}\right.\\
&&~~~~~-\left(\left(\frac{\mathcal {W}(x, y,z)}{2}+\frac{\mathcal {M}^3(x, y,z)}{27}-\frac{\mathcal {M}(x, y,z)\mathcal {N}(x, y,z)}{6}\right)^2\right.\\
&&~~~~~+\left.\left.\left(\frac{\mathcal {N}(x, y,z)}{3}-\frac{\mathcal {M}^2(x, y,z)}{9}\right)^3\right)^{\frac{1}{2}}\right\}^{\frac{1}{3}}\\
&&~~~~~>\max\left.\left\{1+\frac{\mathcal {M}(x, y,z)}{3},-\frac{\mathcal {M}(x, y,z)}{6}+\sqrt{1-\frac{\mathcal {M}^2(x, y,z)}{12}}\right\}\right\},\\
&&D_3:=\left\{(x,y,z)\bigg|\left\{\frac{\mathcal {W}(x, y,z)}{2}+\frac{\mathcal {M}^3(x, y,z)}{27}-\frac{\mathcal {M}(x, y,z)\mathcal {N}(x, y,z)}{6}\right.\right.\\
&&~~~~~+\left(\left(\frac{\mathcal {W}(x, y,z)}{2}+\frac{\mathcal {M}^3(x, y,z)}{27}-\frac{\mathcal {M}(x, y,z)\mathcal {N}(x, y,z)}{6}\right)^2\right.\\
&&~~~~~+\left.\left.\left(\frac{\mathcal {N}(x, y,z)}{3}-\frac{\mathcal {M}^2(x, y,z)}{9}\right)^3\right)^{\frac{1}{2}}\right\}^{\frac{1}{3}}\\
&&~~~~~+\left\{\frac{\mathcal {W}(x, y,z)}{2}+\frac{\mathcal {M}^3(x, y,z)}{27}-\frac{\mathcal {M}(x, y,z)\mathcal {N}(x, y,z)}{6}\right.\\
&&~~~~~-\left(\left(\frac{\mathcal {W}(x, y,z)}{2}+\frac{\mathcal {M}^3(x, y,z)}{27}-\frac{\mathcal {M}(x, y,z)\mathcal {N}(x, y,z)}{6}\right)^2\right.\\
&&~~~~~+\left.\left.\left(\frac{\mathcal {N}(x, y,z)}{3}-\frac{\mathcal {M}^2(x, y,z)}{9}\right)^3\right)^{\frac{1}{2}}\right\}^{\frac{1}{3}}\\
&&~~~~~<\min\left.\left\{\frac{\mathcal {M}(x, y,z)}{3}-1,-\frac{\mathcal {M}(x, y,z)}{6}-\sqrt{1-\frac{\mathcal {M}^2(x, y,z)}{12}}\right\}\right\},
\end{eqnarray*}
\end{small}
we come to the following conclusion.
\begin{lm}
If $(x,y,z)\in{D_1\cap D_2}$ or $(x,y,z)\in{D_1\cap D_3}$, then $\bar{\Delta}>0$ and $|\lambda_{1,2,3}|>1$.
In addition, if the fixed point $E_2: (x^*,y^*,z^*)=(1/\beta,(\beta  \mu -\beta -\mu)/\beta  \mu,0)$ of system \eqref{eq2.1} satisfies
$$
E_2:(x^*,y^*,z^*)\in U(E_2):=\{(x,y,z)|(x,y,z)\in{D_1\cap D_2}~or~(x,y,z)\in{D_1\cap D_3}\},
$$
then $E_2:(x^*,y^*,z^*)$ is an expanding fixed point of $U(E_2)$.
\label{lmchaos}
\end{lm}

Next, our aim is finding a point $E': (x',y',z')\in U(E_2)\backslash \{E_2\}$ satisfying $|DF_{\lambda,\mu,\beta}^N(E')|\ne 0$ for certain positive integer $N$.

Suppose that
\begin{eqnarray}
\small
\left\{\begin{array}{ll}
x''=\mu x' \left(1-x'-y'-z'\right),  \\
y''=\beta y' \left(x'-z'\right),  \\
z''=\lambda y' z'
\end{array}\right.
\label{fixPchaos1}
\end{eqnarray}
and
\begin{eqnarray}
\small
\left\{\begin{array}{ll}
x^*=\mu x'' \left(1-x''-y''-z''\right),  \\
y^*=\beta y'' \left(x''-z''\right),  \\
z^*=\lambda y'' z''.
\end{array}\right.
\label{fixPchaos2}
\end{eqnarray}
This implies that $F^2_{\lambda,\mu,\beta}$ maps $E': (x',y',z')$ to the fixed point $E_2: (x^*,y^*,z^*)$ if there exist another solutions different from $E_2$ to equations \eqref{fixPchaos1}-\eqref{fixPchaos2}. Actually, from \eqref{fixPchaos2} we find that the solutions satisfy the following equations
\begin{eqnarray}
\small
\left\{\begin{array}{ll}
x''=\frac{1}{\mu} x^*+x''^2+\frac{1}{\mu}x''y''+x''z'',  \\
y''=\frac{y^*}{\beta(x''-z'')},  \\
z''=\frac{z^*}{\lambda y''}.
\end{array}\right.
\label{fixPchaos21}
\end{eqnarray}
In order to solve $x'$, $y'$, $z'$, substituting \eqref{fixPchaos1}-(\ref{fixPchaos2}) into (\ref{fixPchaos21}) and we get
\begin{eqnarray}
\left\{\begin{array}{ll}
x'=\frac{1-\mathit{x'} \left(-1+\mathit{x'}+\mathit{y'}+\mathit{z'}\right) \beta^{2} \mathit{y'} \left(\mathit{x'}-\mathit{z'}\right) \mu +\mu  \left(1+\mathit{x'} \left(-1+\mathit{x'}+\mathit{y'}+\mathit{z'}\right)^{2} \mu^{2}-\left(\lambda  \mathit{y'} \mathit{z'}-1\right) \left(-1+\mathit{x'}+\mathit{y'}+\mathit{z'}\right) \mu \right) \mathit{x'} \beta}{\mu  \beta},
\\
y'=\frac{-\beta  \mu +\beta +\mu +\left(-\beta  \mathit{y'} \left(\mathit{x'}-\mathit{z'}\right)+\mathit{y'}\right) \mu  \,\beta^{2} \left(\mathit{x'} \left(-1+\mathit{x'}+\mathit{y'}+\mathit{z'}\right) \mu +\lambda  \mathit{y'} \mathit{z'}\right)}{\mu  \,\beta^{2} \left(\mathit{x'} \left(-1+\mathit{x'}+\mathit{y'}+\mathit{z'}\right) \mu +\lambda  \mathit{y'} \mathit{z'}\right)},
\\
z'=0.
\end{array}\right.
\label{fixPchaos11}
\end{eqnarray}
One can calculate that
\begin{eqnarray*}
|DF_{\lambda,\mu,\beta}^2(E')|=
\mathcal{A}_{11}\mathcal{A}_{22}\mathcal{A}_{33}+\mathcal{A}_{13}\mathcal{A}_{21}\mathcal{A}_{32}
-\mathcal{A}_{12}\mathcal{A}_{21}\mathcal{A}_{33}-\mathcal{A}_{11}\mathcal{A}_{23}\mathcal{A}_{32},
\end{eqnarray*}
where
\begin{small}
\begin{eqnarray*}
\!\!\!\!\!\!\!\!\!\!&&\mathcal{A}_{11}=\mu  \left((1+\mu^{2} x' \left(-1+x'+y'+z'\right) \left(\mu  x'^{2}+\left(\left(-1+y'+z'\right) \mu -\beta  y'\right) x'+1+z' \left(\beta -\lambda \right) y'\right)\right)
 \\
\!\!\!\!\!\!\!\!\!\!&&~~~~~~~+\mu  \left(y' \beta^{2} \left(\mu  x'^{2}+\left(-1+y'+z'\right) \mu  x'+\lambda  y' z'\right) \left(x'-z'\right)-\lambda^{2} \beta  y'^{2} \left(x'-z'\right) z'\right)
\\
\!\!\!\!\!\!\!\!\!\!&&~~~~~~~+\mu^{3} x' \left(-1+x'+y'+z'\right) \left(\mu  x'^{2}+\left(\left(-1+y'+z'\right) \mu -\beta  y'\right) x'+1+z' \left(\beta -\lambda \right) y'\right),
\\
\!\!\!\!\!\!\!\!\!\!&&\mathcal{A}_{12}=\mathcal{A}_{13}=\mu^{3} x' \left(-1+x'+y'+z'\right) \left(\mu  x'^{2}+\left(\left(-1+y'+z'\right) \mu -\beta  y'\right) x'+1+z' \left(\beta -\lambda \right) y'\right),
\\
\!\!\!\!\!\!\!\!\!\!&&\mathcal{A}_{21}=-\beta^{3} y' \left(\mu  x'^{2}+\left(-1+y'+z'\right) \mu  x'+\lambda  y' z'\right) \left(x'-z'\right),
\\
\!\!\!\!\!\!\!\!\!\!&&\mathcal{A}_{22}=\beta  \left(-\mu^{2} x' \left(-1+x'+y'+z'\right) \left(\mu  x'^{2}+\left(\left(-1+y'+z'\right) \mu -\beta  y'\right) x'+1+z' \left(\beta -\lambda \right) y'\right)\right)
\\
\!\!\!\!\!\!\!\!\!\!&&~~~~~~~+\beta  \left(-\lambda^{2} \beta  y'^{2} \left(x'-z'\right) z'\right),
\\
\!\!\!\!\!\!\!\!\!\!&&\mathcal{A}_{23}=\beta^{3} y' \left(\mu  x'^{2}+\left(-1+y'+z'\right) \mu  x'+\lambda  y' z'\right) \left(x'-z'\right),
\\
\!\!\!\!\!\!\!\!\!\!&&\mathcal{A}_{31}=0,
\\
\!\!\!\!\!\!\!\!\!\!&&\mathcal{A}_{32}=\lambda^{3} \beta  y'^{2} \left(x'-z'\right) z',
\\
\!\!\!\!\!\!\!\!\!\!&&\mathcal{A}_{33}=-\lambda  y' \beta^{2} \left(\mu  x'^{2}+\left(-1+y'+z'\right) \mu  x'+\lambda  y' z'\right) \left(x'-z'\right).
\end{eqnarray*}
\end{small}
\vspace{-0.8cm}

Given the nature of expanding fixed point $E_2: (x^*,y^*,z^*)$ under mapping $F$, i.e., $|\lambda_{1,2,3}|>1$,we derive the following parameter constraints:
\begin{eqnarray*}
\left\{(\lambda,\mu,\beta)\bigg|
\left|\frac{\lambda(\beta\mu - \beta - \mu)}{\beta\mu}\right| > 1,
\frac{(-4\mu + 4)\beta^2 + 4\beta\mu + \mu^2}{\beta^2} < 0,
\frac{\mu(\beta - 2)}{\beta} > 1\right\}.
\end{eqnarray*}
Employing the theory of complete discrimination system to solve the corresponding semi-algebraic system, we establish the parameter region
\begin{eqnarray*}
\mathfrak{W}=\left\{(\lambda,\mu,\beta)\bigg|
\beta >\frac{9}{4},\frac{\beta}{\beta -2}<\mu <{2 \beta^{2}-2 \beta +2 \sqrt{\beta^{4}-2 \beta^{3}}}\,,\lambda>\frac{\mu  \beta}{\mu  \beta -\mu -\beta}\right\}
\end{eqnarray*}
\vspace{-0.5cm}

Finally, when the conditions in Lemma~\ref{lmchaos} are satisfied, equations \eqref{fixPchaos21}-\eqref{fixPchaos11} indicate that $E', E''\ne E_2$, $E'\in U(E_2)$ and $|DF_{\lambda,\mu,\beta}^2(E')|\ne 0$, implying $E_2$ is a snap-back repeller of $F$. Consequently, the following theorem is obtained.
\begin{thm}
Suppose the conditions in Lemma~\ref{lmchaos} and further the following conditions
\begin{description}
\item[(i)] $(\lambda,\mu,\beta)\in\mathfrak{W}$,
\item[(ii)] The real solutions of equations \eqref{fixPchaos21}-\eqref{fixPchaos11} are different from $E_2$, i.e., $(x',y',z'), (x'',y'',z'')\ne (x^*,y^*,z^*)$. Moreover, $E'(x',y',z')\in U(E_2)$, $(x',y',z')\ne(0,0,0)$ and $|DF_{\lambda,\mu,\beta}^2(E')|\ne 0$
\end{description}
hold. Then $E_2$ is a snap-back repeller of system \eqref{eq2.1} and system \eqref{eq2.1} possesses chaotic behaviors in the sense of Marotto.
\label{E2hd}
\end{thm}
%%%%%%%%%%%%%%%%%%%%%%%%%%%%%%%%%%%%%%%%%%%%%%%%%%%%%%%%%%%%%%%%%%%%%%%%%%%%%%%%%%%%%%%%%%%%%%%%%%%%%%%%%%%%%%%%%
\section{Numerical simulations}

In this section, we will use Matlab R2023a and Maple 2023 to simulate the dynamic properties of system \eqref{eq2.1} to verify our results obtained in Sections 3-6.

From Theorem \ref{flipE1}, when the parameter $\Lambda$ crosses the region $\{\lambda>0,\mu=3,\beta>0\}$, system \eqref{eq2.1} produces a supercritical flip bifurcation.
Therefore, let $\mu$ be the bifurcation parameter. Setting $\lambda=2.9$, $\beta=3.03$ and employing the software Matlab R2023a, we plot the bifurcation diagram of system \eqref{eq2.1} with an initial value $(x_0,y_0,z_0)=(0.665,0.010,0.010)$ in Figure \ref{flipe1} in the $(\mu,x,y)-$space. The corresponding Lyapunov exponents diagram is shown in Figure \ref{lya-E1}.

\begin{figure}[htbp]
\T\T\T\T\T\T\T\T\T\T\T\T\T\T\T\T\T\T\T\T\T\T\T\T\T\T\T\T\T\T\T\T\T\T\T\T
\subfigure[Flip bifurcation diagram at the fixed point $E_1$ when the parameter $\Lambda$ crosses $\{\lambda>0,\beta>1,\mu=3\}$]{
\includegraphics[width=7.5cm]{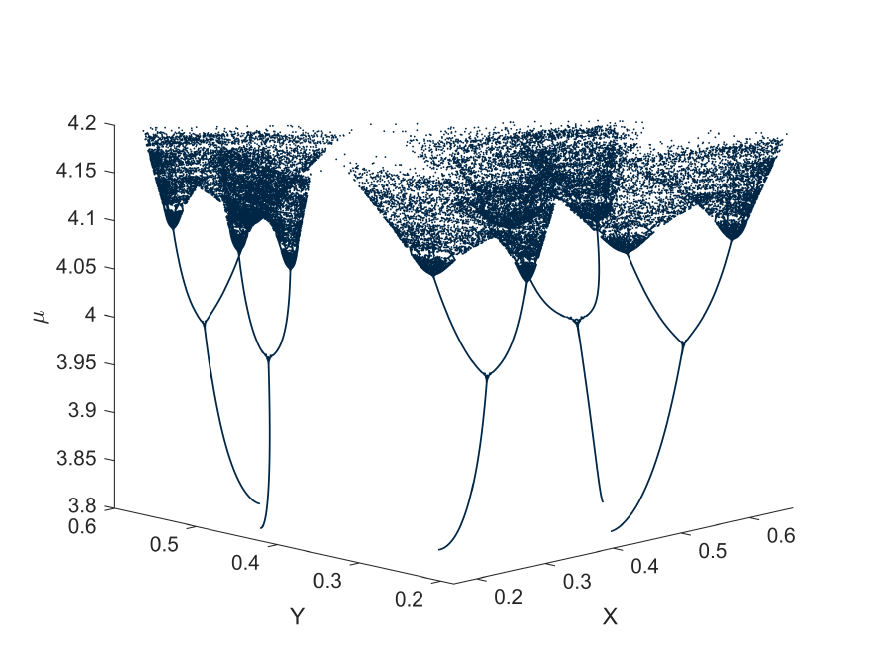}
\label{flipe1}
}
\quad
\subfigure[Lyapunov exponents corresponding to (a)]{
\includegraphics[width=7.5cm]{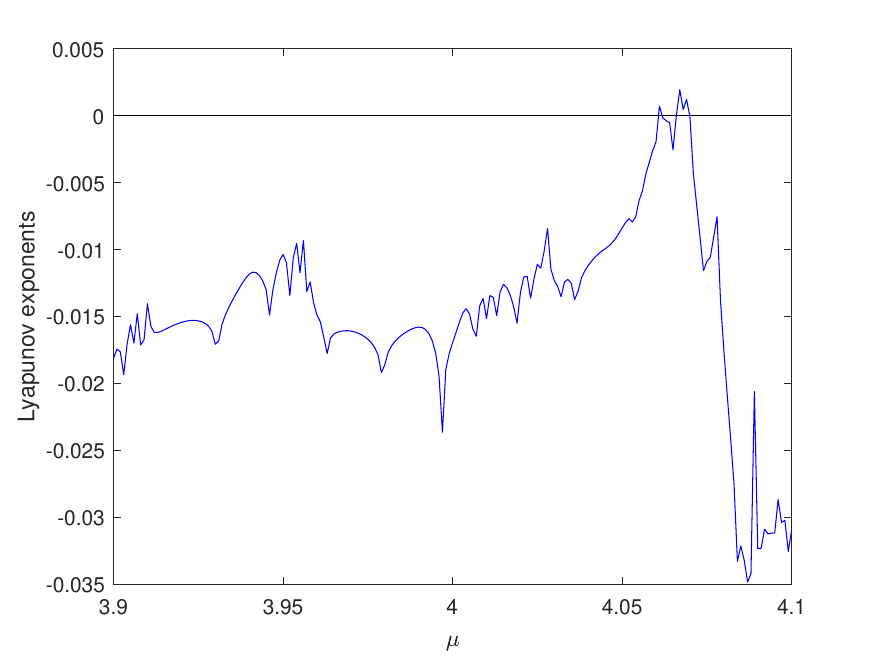}
\label{lya-E1}
85}
\caption{Flip bifurcation diagram and corresponding Lyapunov exponents diagram of system \eqref{eq2.1} at fixed point $E_1$}
\end{figure}
%Figure \ref{flipe1} not only shows that mapping \eqref{eq2.1} indeed occurs flip bifurcation, but also shows that the system generates chaos at $\mu\approx 3.03$.
%It can be seen from Figure \ref{hdE1-1} that the system does begin to produce chaotic phenomena near $\mu=3.03$, and the chaotic phenomena gradually become clear with the change of bifurcation parameters(Figure \ref{hdE1-2}).
Besides, from Theorem \ref{flipE2}, when the parameter $\Lambda$ crosses the region $\{\lambda>0,\mu=-\beta/(\beta-3),\beta=2\}$, system \eqref{eq2.1} can also undergo a subcritical flip bifurcation by taking $\mu$ as the bifurcation parameter. Setting $\lambda=5.4999$, $\beta=2.14$ and using Matlab R2023a, the bifurcation diagram of system \eqref{eq2.1} with an initial value $(x_0,y_0,z_0)=(0.500,0.325,0.005)$ is plotted in Figure \ref{flipe2} in the $(\mu,x,y)-$space. The corresponding Lyapunov exponents diagram is shown in Figure \ref{lya-E2}.
\begin{figure}[htbp]
\T\T\T\T\T\T\T\T\T\T\T\T\T\T\T\T\T\T\T\T\T\T\T\T\T\T\T\T\T\T\T\T\T\T\T\T
\subfigure[Flip bifurcation diagram at the fixed point $E_2$ when the parameter $\Lambda$ crosses $\{\lambda>0,\beta=2,\mu=-3\beta/(\beta-3)\}$]{
\includegraphics[width=7.5cm]{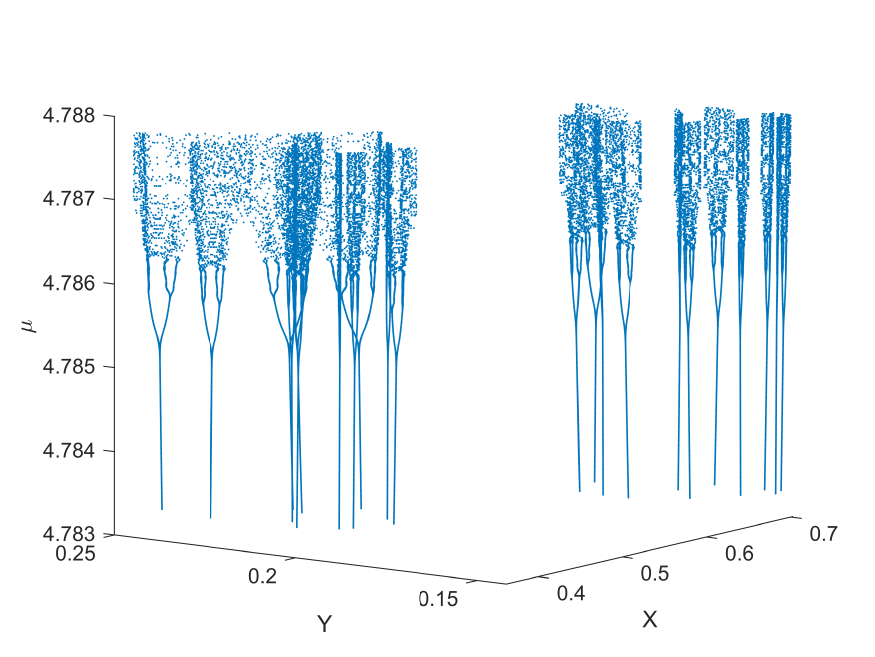}
\label{flipe2}
}
\quad
\subfigure[Lyapunov exponents corresponding to (a)]{
\includegraphics[width=7.5cm]{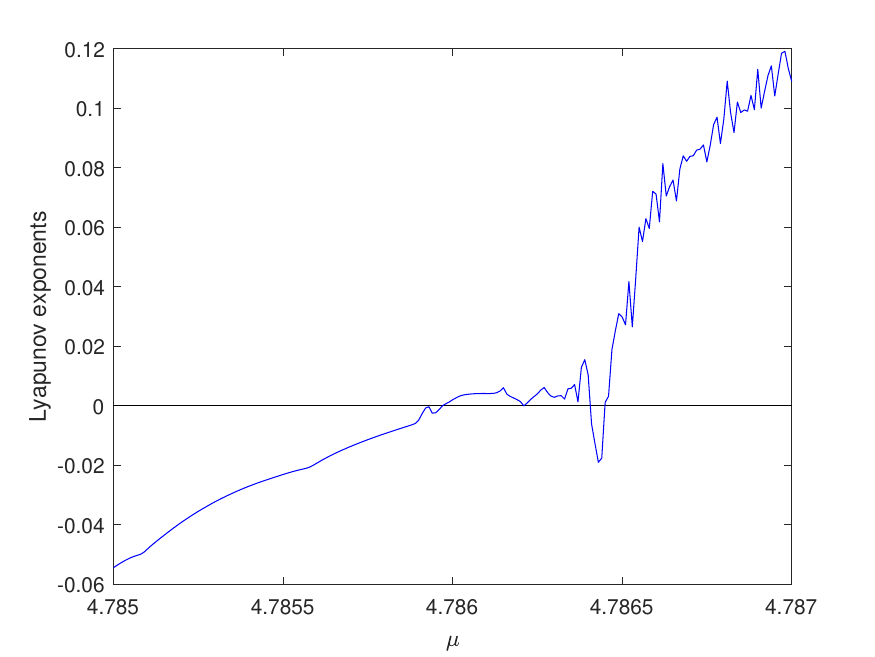}
\label{lya-E2}
}
\caption{Flip bifurcation diagram and corresponding Lyapunov exponents diagram of system \eqref{eq2.1} at fixed point $E_2$}
\end{figure}
%\begin{figure}[htbp]
%\T\T\T\T\T\T\T\T\T\T\T\T\T\T\T\T\T\T\T\T\T\T\T\T\T\T\T\T\T\T\T\T\T\T\T\T
%\subfigure[Fixed point $E_1$ in $\lambda=4,~\beta=3.9,~\mu=3.03$ near begin to produce chaotic phenomena.]{
%\includegraphics[width=7.5cm]{HD-E1-1.eps}
%\label{hdE1-1}
%}
%\quad
%\subfigure[Fixed point $E_1$ in $\lambda=4,~\beta=3.9,~\mu=3.4$ near chaotic phenomena gradually become clear.]{
%\includegraphics[width=7.5cm]{HD-E1-2.eps}
%\label{hdE1-2}
%}
%\caption{Chaotic diagram of system \eqref{eq2.1} with $(\lambda,\mu,\beta)$ at fixed point $E_1$. }
%\label{HDE1}
%\end{figure}

\begin{figure}
\begin{center}
\includegraphics[width=10cm]{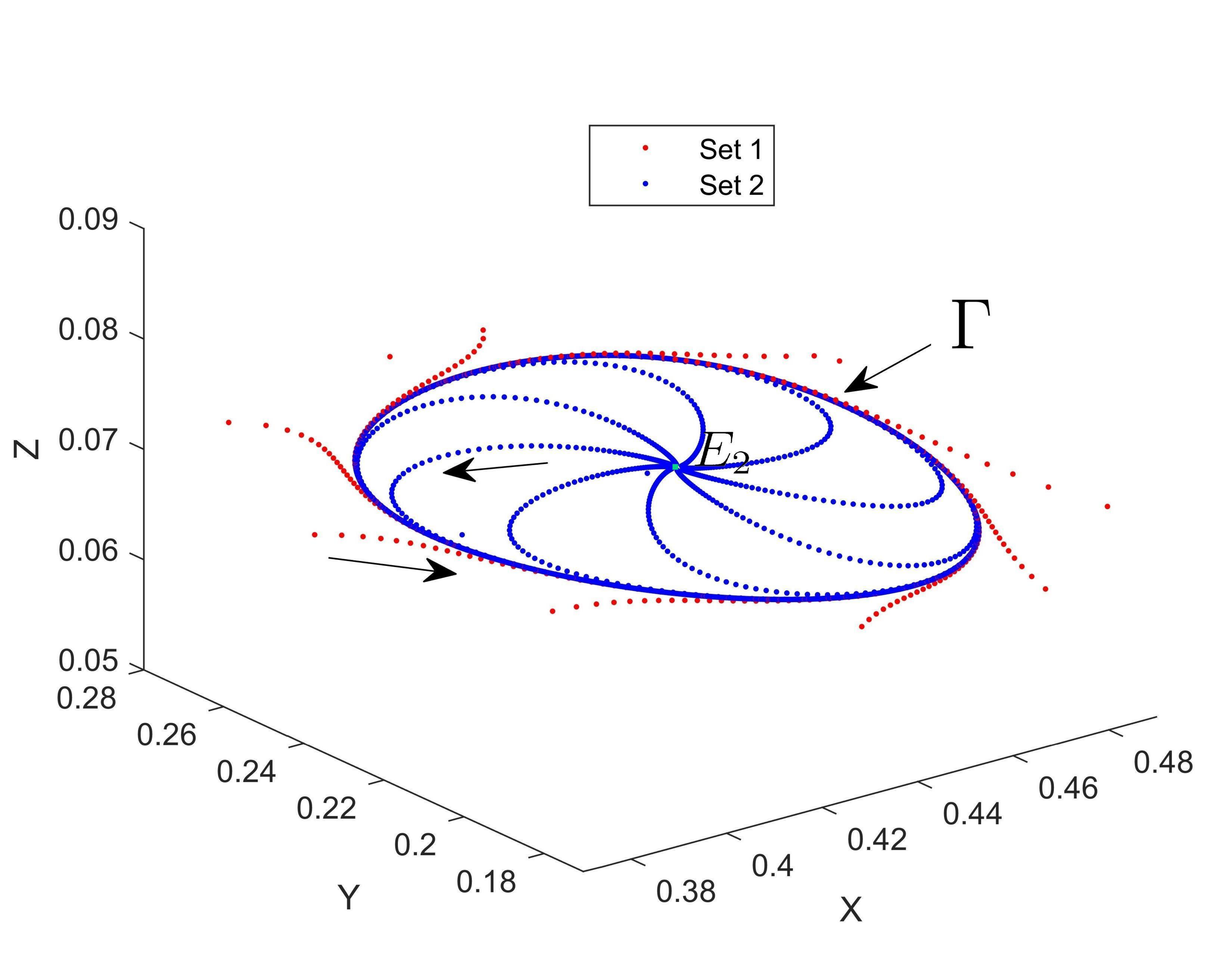}
\caption{An invariant circle $\Gamma$ generated from the Neimark-Sacker bifurcation as the parameter $\Lambda$ crosses $\mathfrak{L}$}
\label{ns}
\end{center}
\end{figure}

According to Theorem~\ref{nsE2}, system \eqref{eq2.1} undergoes a Neimark-Sacker bifurcation near $E_2$ and produces a unique invariant circle when the parameter $\Lambda$ crosses the region $\{\lambda>0,\mu=\beta/\beta-2,\beta>9/4,\beta\neq7/3,~\beta\neq5/2,
\beta\neq9/2+\sqrt{21}/2\}$.
Thus, setting the parameter $(\lambda,\mu,\beta)=(4.444,3.710,2.734)$, and taking two initial values $(x_{01},y_{01},z_{01})=(0.366324,0.364360,0.040000)$ and $(x_{02},y_{02},z_{02})=(0.39123400,0.37930200,0.00000011)$, which are near the fixed point $E_2$ and the invariant circle $\Gamma$ in Figure \ref{ns}.
From Figure \ref{ns}, we see two orbits with different colours, and an invariant circle $\Gamma$ at the junction of two orbits. The blue orbit leaves fixed point $E_2$ and tends to $\Gamma$ and the red one tends to $\Gamma$, implying that system \eqref{eq2.1} undergoes a Neimark-Sacker bifurcation near $E_2$.

%%%%%%%%%%%%%%%%%%%%%%%%%%%%%%%%%%
Next, we will simulate the bifurcation phenomena of the $1:2$, $1:3$ and $1:4$ resonances of system \eqref{eq2.1} to verify Theorems \ref{th1-2}-\ref{th1-4}.

We will utilize Matlab R2023a to simulate the dynamical behaviors of system \eqref{eq2.1} near the 1:2 resonance region. From Theorem~\ref{th1-2}, as $(\lambda,\mu,\beta)\in\{(\lambda,\mu,\beta)|\lambda>0,\beta>9/4,\beta/(\beta-2)+O(|(\beta-9/4)|)^2<\mu<41\beta/(26-7\beta)+O(|(\beta-9/4)|)^2\}$, system \eqref{eq2.1} undergoes a Neimark-Sacker bifurcation and generates a unique invariant circle $\Gamma$, while $\Gamma$ coexists with the unstable period-two orbit $\{Q_{11},Q_{12}\}$ from a subcritical flip bifurcation.
Hence, taking $(\lambda,\mu,\beta)=(1.0000,2.2554,8.9966)$ and the initial values $(x_{11},y_{11},z_{11})=(0.444100,0.445100,0.000001)$ and
$(x_{12},y_{12},z_{12})=(0.417500,0.455100,0.000001)$, after iterating $10^5$ steps, we obtain Figure \ref{2-1}.
Additionally, setting the initial value $(x_{13},y_{13},z_{13})=(0.418300,0.455500,0.000001)$ while maintaining the parameter values and iteration count, we get a blue orbit approaching $\Gamma$ from outside of $\Gamma$ and a red orbit diverging from the fixed point $E_2$ and approaching $\Gamma$ in Figure \ref{2-2}. This indicates that the invariant circle $\Gamma$ coexists with the unstable period-two orbit $\{Q_{11},Q_{12}\}$. In order to verify the disappearance of the invariant circle when the parameter $\Lambda$ crosses the region $H_{2r}$, we select the initial value $(x_{14},y_{14},z_{14})=(0.443600,0.447200,0.000001)$, keep the parameter values and iteration count, then the period-two orbit $\{Q_{11},Q_{12}\}$ is found in $\Gamma$, which looks like a "rowing boat" with vertexes in Figure \ref{2-3}.
Finally, setting $(\lambda,\mu,\beta)=(1.0000,2.2513,8.9050)$ and the initial values $(x_{11},y_{11},z_{11})=(0.444100,0.445100,0.000001)$, $(x_{15},y_{15},z_{15})=(0.405830,0.458650,0.000001)$, it is shown from the red orbit in Figure \ref{2-4} that the invariant cycle disappears.

\begin{figure}[htbp]
\T\T\T\T\T\T\T\T\T\T\T\T\T\T\T\T\T\T\T\T\T\T\T\T\T\T\T\T\T\T\T\T\T\T\T\T
\subfigure[$(\lambda,\mu,\beta)=(1.0000,2.2554,8.9966)$]{
\includegraphics[width=7.5cm]{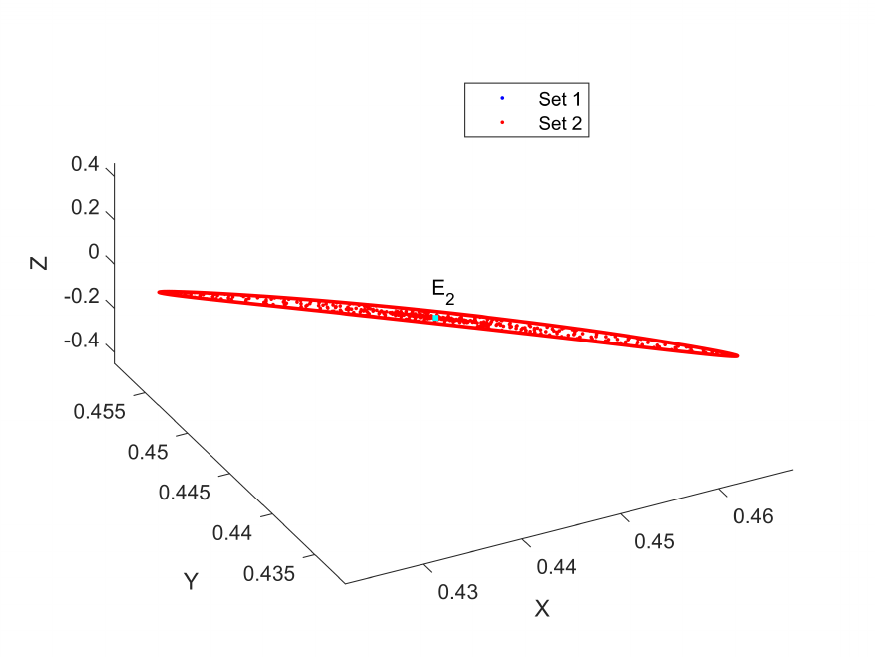}
%\caption{fig1}
\label{2-1}
}
\quad
\subfigure[$(\lambda,\mu,\beta)=(1.0000,2.2554,8.9966)$]{
\includegraphics[width=7.5cm]{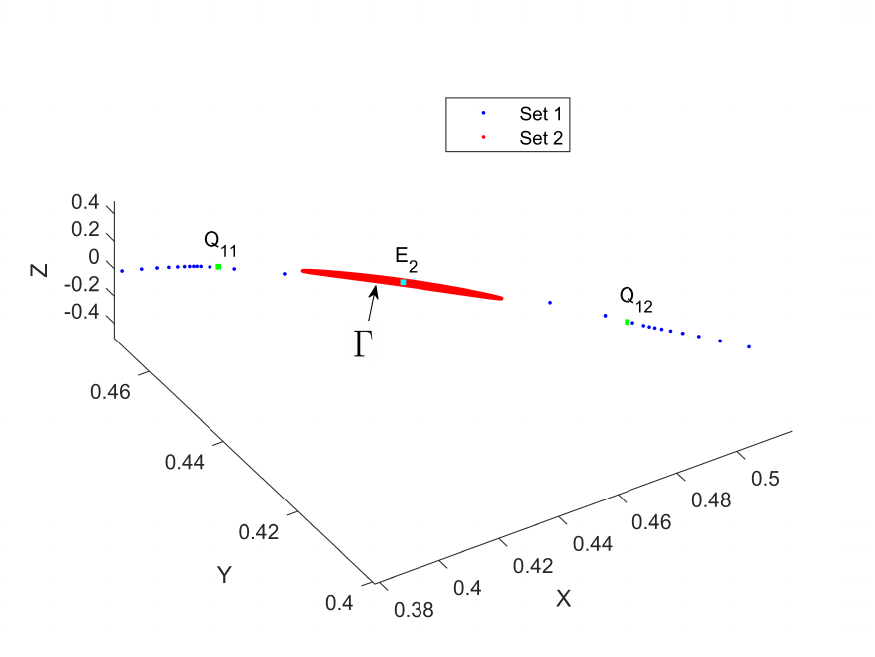}
\label{2-2}
}
\T\T\T\T\T\T\T\T\T\T\T\T\T\T\T\T\T\T\T\T\T\T\T\T\T\T\T\T\T\T\T\T\T\T\T\T
\subfigure[$(\lambda,\mu,\beta)=(1.0000,2.2554,8.9966)$]{
\includegraphics[width=7.5cm]{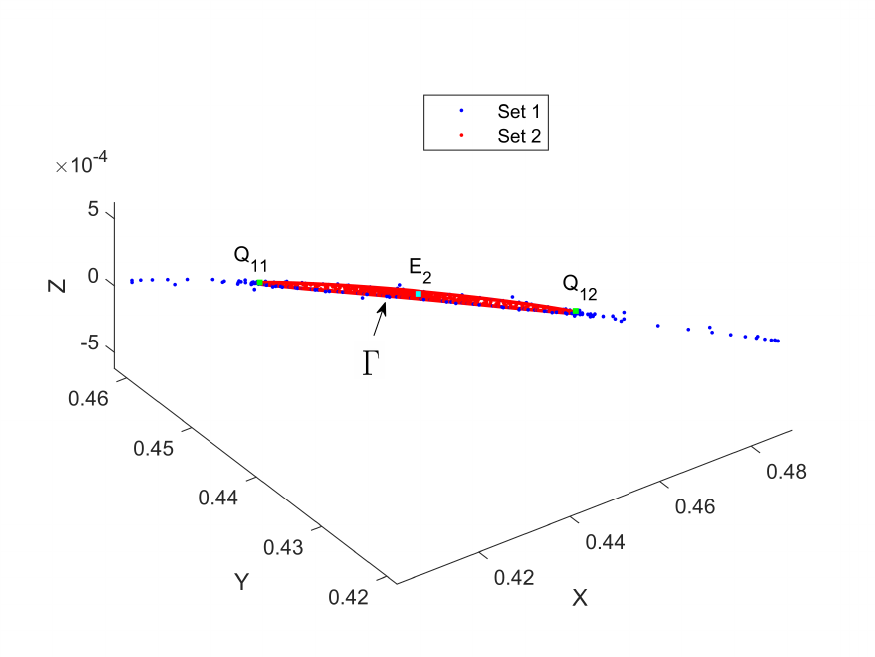}
\label{2-3}
}
\quad
\subfigure[$(\lambda,\mu,\beta)=(1.0000,2.2513,8.9050)$]{
\includegraphics[width=7.5cm]{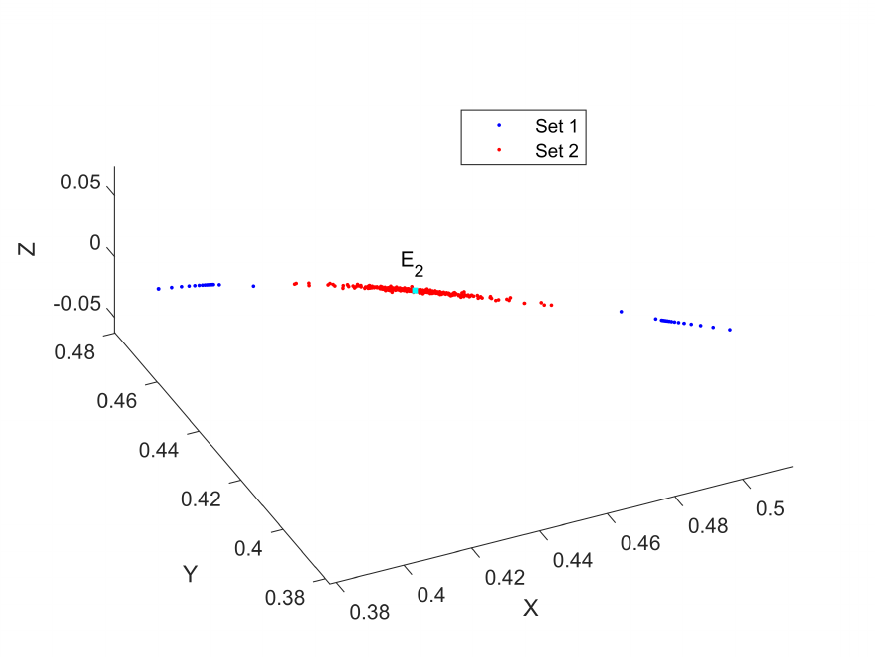}
\label{2-4}
}
\caption{Phase portraits of 1:2 resonance in system \eqref{eq2.1} }
\end{figure}
%%%%%%%%%%%%%%%%%%
In order to verify the 1:3 resonance in system \eqref{eq2.1}, we first set the parameter $(\lambda,\mu,\beta)=(1.00000,2.34027,6.87546)$ and select the initial values $(x_{21},y_{21},z_{21})=(0.4261325,0.429845,0.000010)$, $(x_{22},y_{22},z_{22})=(0.4261325,0.429745,0.000010)$. It is shown that three blue orbits diverge towards infinity, three red orbits converge towards the fixed point $E_2$, and their junction form a period-three saddle point cycle $\{T_{11}, T_{12}, T_{13}\}$ in Figure \ref{3-1}.
According to Theorem \ref{th1-3}, when the parameter $\Lambda$ crosses the region $\mathfrak{L}^{''}_3$, system \eqref{eq2.1} undergoes a supercritical Neimark-Sacker bifurcation and produces a unique invariant circle. Thus, taking $(\lambda,\mu,\beta)=(1.00000,2.34050,6.87546)$ and the initial values $(x_{23},y_{23},z_{23})=(0.42732,0.42734,0.00001)$, $(x_{24},y_{24},z_{24})=(0.4222124,0.4256297,0.0000100)$, we obtain green and red orbits, and the invariant circle $\Gamma$ (Figure \ref{3-2}). From Figure \ref{3-2}, we see that the green orbit originating from $(x_{23},y_{23},z_{23})$ approaches $\Gamma$ from the interior, while the red orbit from $(x_{24},y_{24},z_{24})$ approaches $\Gamma$ from the exterior. Theorem \ref{th1-3} also indicates that the invariant circle $\Gamma$ in system \eqref{eq2.1} coexists with the period-three saddle point cycle $\{T_{11}, T_{12}, T_{13}\}$. By selecting $(x_{25},y_{25},z_{25})=(0.4261325,0.4298450,0.0000100)$ as the initial value, we obtain the blue orbits. Besides, a period-three saddle point cycle $\{T_{11}, T_{12}, T_{13}\}$ exists between the red and blue orbits (Figure \ref{3-2}).
Similarly, setting $(\lambda,\mu,\beta)=(1.00000,2.34060,6.87546)$ and the initial values $(x_{26},y_{26},z_{26})=(0.4261325,0.4248450,0.0000100)$, $(x_{27},y_{27},z_{27})=(0.4237821240,0.4258796297,0.0000100000)$, $(x_{28},y_{28},z_{28})=(0.42732,0.42734,0.00001)$, we get three orbits of different colors: blue, red, and green (Figure \ref{3-3}). Figure \ref{3-3} shows that the invariant circle $\Gamma$ coexists with the period-three saddle point cycle $\{T_{11}, T_{12}, T_{13}\}$.
Moreover, taking $(\lambda,\mu,\beta)=(1.00000,2.34110,6.87546)$, and the initial values $(x_{29},y_{29},z_{29})=(0.4222124,0.4256297,0.0000100)$, $(x_{210},y_{210},z_{210})=(0.42732,0.42734,0.00001)$, we get blue and red orbits (see Figure \ref{3-4}). We see that there exists an invariant circle $\Gamma$ at the junction of the blue and red orbits. Furthermore, a period-three saddle point cycle $\{T_{11}, T_{12}, T_{13}\}$ lies in the invariant circle $\Gamma$.
Subsequently, setting $(\lambda,\mu,\beta)=(1.00000,2.34130,6.87546)$, and the initial values $(x_{211},y_{211},z_{211})=(0.4222124,0.4256297,0.0000100)$, $(x_{212},y_{212},z_{212})=(0.42732,0.42734,0.00001)$, we obtain red and blue orbits, where the red orbit diverged from the fixed point $E_2$ towards infinity. Further, the invariant circle $\Gamma$ starts to disappear, and a period-three saddle point cycle $\{T_{11}, T_{12}, T_{13}\}$ appears at the junction of the blue and red orbits (Figure \ref{3-5}).
Additionally, selecting $(\lambda,\mu,\beta)=(1,2.339,6.94)$ and the initial values $(x_{213},y_{213},z_{213})=(0.4232124,0.4282970,0.0000100)$, $(x_{214},y_{214},z_{214})=(0.42732,0.42834,0.00001)$, we obtain blue and red orbits diverging from the fixed point $E_2$ that toward infinity, respectively (Figure \ref{3-6}). In this case, the invariant circle disappears, but the period-three saddle point cycle $\{T_{11}, T_{12}, T_{13}\}$ exists.
\begin{figure}[htbp]
\T\T\T\T\T\T\T\T\T\T\T\T\T\T\T\T\T\T\T\T\T\T\T\T\T\T\T\T\T\T\T\T\T\T\T\T
\subfigure[$(\lambda,\mu,\beta)=(1.00000,2.34027,6.87546)$]{
\includegraphics[width=7.5cm]{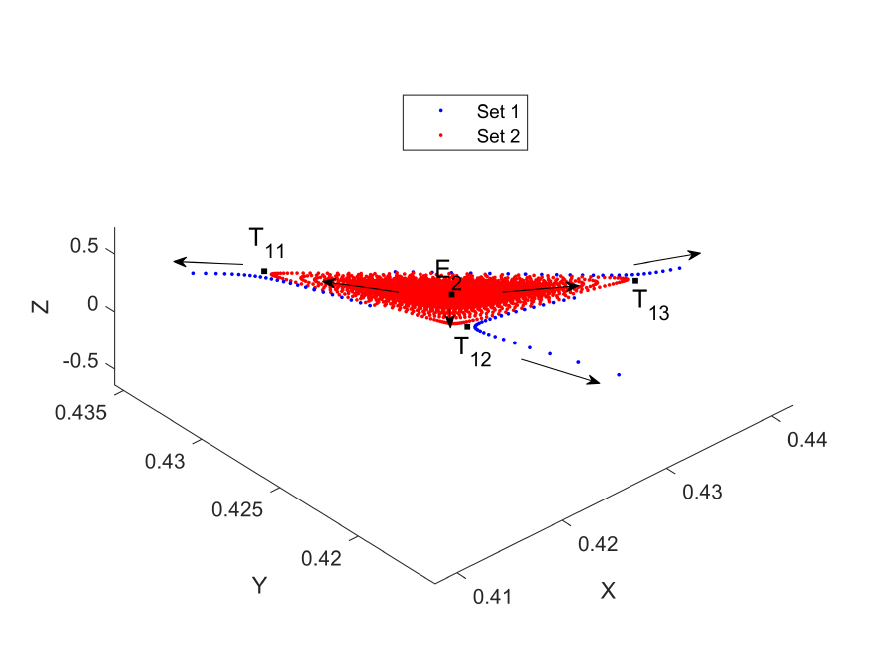}
%\caption{fig1}
\label{3-1}
}
\quad
\subfigure[$(\lambda,\mu,\beta)=(1.00000,2.34050,6.87546)$]{
\includegraphics[width=7.5cm]{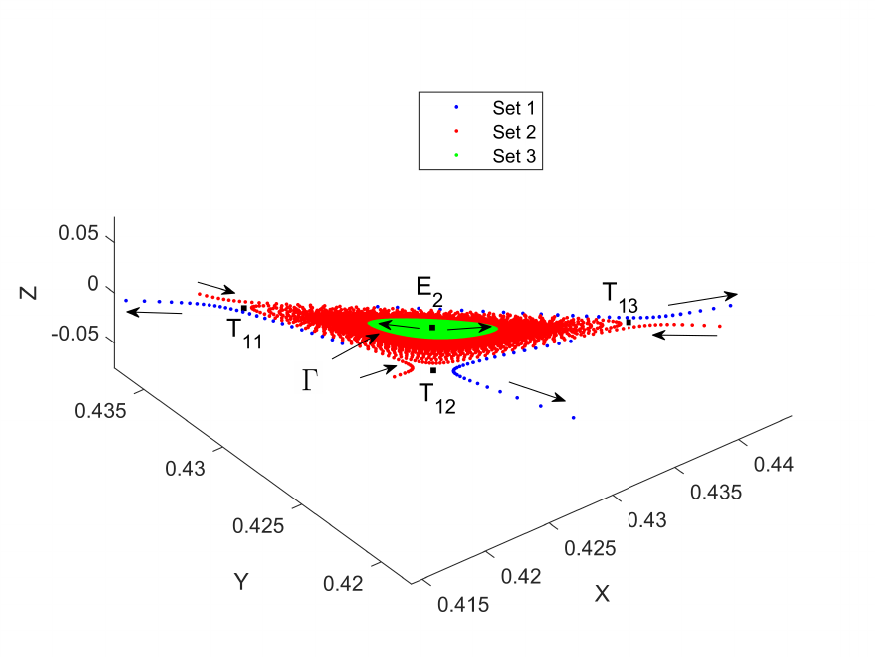}
\label{3-2}
}
\T\T\T\T\T\T\T\T\T\T\T\T\T\T\T\T\T\T\T\T\T\T\T\T\T\T\T\T\T\T\T\T\T\T\T\T
\subfigure[$(\lambda,\mu,\beta)=(1.00000,2.34060,6.87546)$]{
\includegraphics[width=7.5cm]{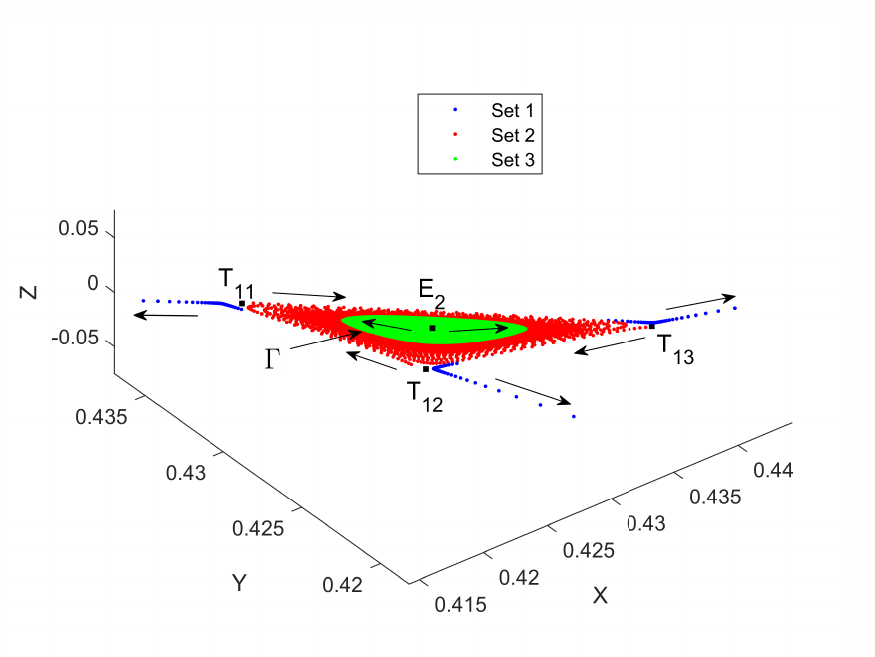}
\label{3-3}
}
\quad
\subfigure[$(\lambda,\mu,\beta)=(1.00000,2.34110,6.87546)$]{
\includegraphics[width=7.5cm]{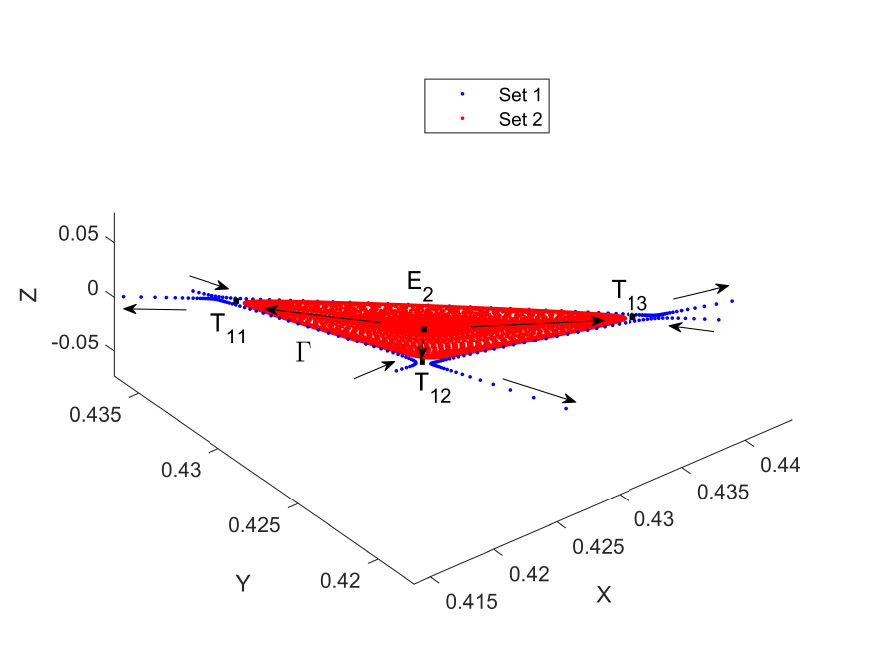}
\label{3-4}
}
\T\T\T\T\T\T\T\T\T\T\T\T\T\T\T\T\T\T\T\T\T\T\T\T\T\T\T\T\T\T\T\T\T\T\T\T
\subfigure[$(\lambda,\mu,\beta)=(1.00000,2.34130,6.87546)$]{
\includegraphics[width=7.5cm]{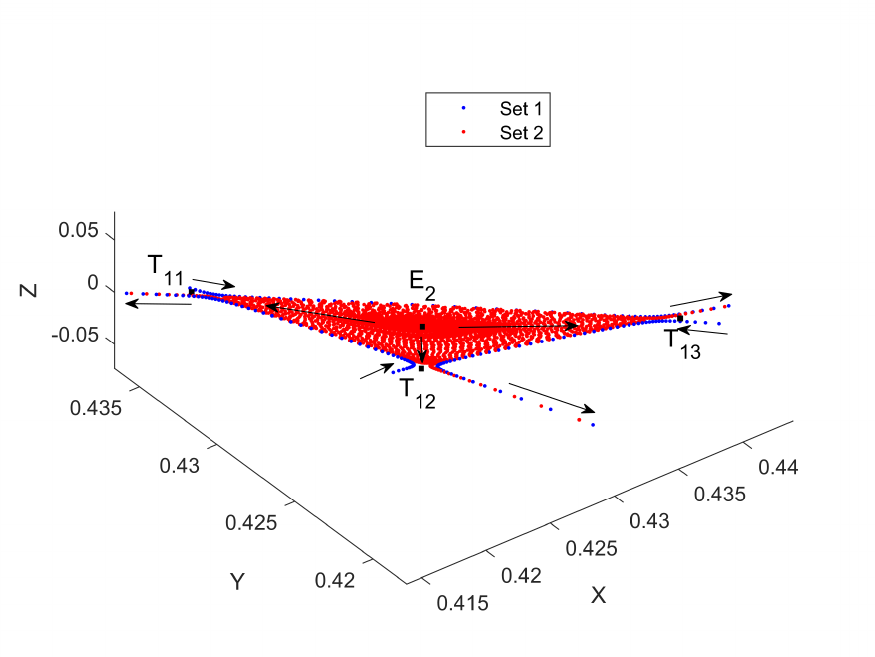}
\label{3-5}
}
\quad
\subfigure[$(\lambda,\mu,\beta)=(1.000,2.339,6.940)$]{
\includegraphics[width=7.5cm]{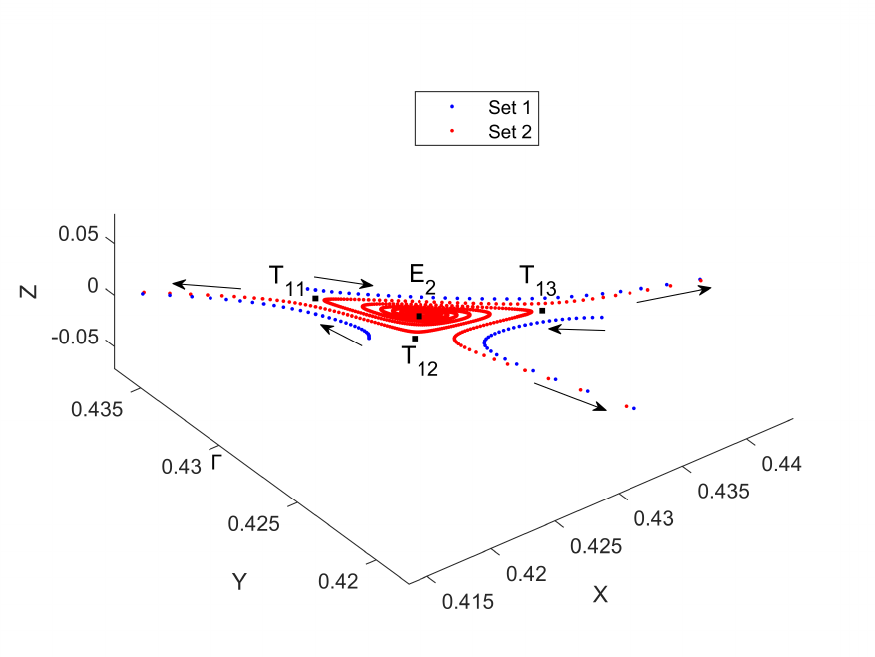}
\label{3-6}
}
\caption{Phase portraits of 1:3 resonance in system \eqref{eq2.1} }
\end{figure}

%%%%%%%%%%%%%%%%%%%%%%
Finally, we consider the 1:4 resonance phenomenon in system \eqref{eq2.1} to illustrate our Theorem \ref{th1-4}. According to Theorem \ref{th1-4}, when the parameter $\Lambda$ crosses $\bar{\mathfrak{L}}_3$, system \eqref{eq2.1} undergoes a supercritical Neimark-Sacker bifurcation and produces a stable invariant circle surrounding $E_2$.
Hence, for a sufficiently small neighborhood within the 1:4 resonance region $\{(\lambda,\mu,\beta)|\lambda>0,\mu=5,\beta=5/2\}$, we select parameter value $(\lambda,\mu,\beta)=(1.000,2.512,4.973)$, and the initial values $(x_{30},y_{30},z_{30})=(0.442520,0.453130,0.000001)$,  $(x_{31},y_{31},z_{31})=(0.398750,0.400950,0.000001)$. The blue and red orbits which leave the fixed point $E_2$ are shown in Figure \ref{4-1}. From Figure \ref{4-1}, we see that there exists a period-four saddle point $\{K_{11},K_{12},K_{13},K_{14}\}$ at the junction of the blue and red orbits.
Besides, selecting parameter $(\lambda,\mu,\beta)=(1.000,2.503,4.990)$ and the initial values $(x_{32},y_{32},z_{32})=(0.3995100,0.4000700,0.0000001)$, $(x_{33},y_{33},z_{33})=(0.4005300,0.4075200,0.0000001)$, we simulate the blue and red orbits in Figure \ref{4-2}. From Figure \ref{4-2}, we observe that there exists an invariant circle at the junction of the blue and red orbits, and the red from the fixed point $E_2$ tends to the invariant circle, while the blue orbit tends to the invariant circle from outside. Further, setting the initial value $(x_{34},y_{34},z_{34})=(0.4005300,0.4175200,0.0000001)$ while keeping the same parameter, we obtain the green orbit in Figure \ref{4-2}. One can see that there exists a period-four saddle point $\{K_{11},K_{12},K_{13},K_{14}\}$ at the junction of the green and blue orbits.
Similarly, selecting parameter $(\lambda,\mu,\beta)=(1.0000,2.5025,4.9900)$ and the initial values $(x_{35},y_{35},z_{35})=(0.3995100,0.4000700,0.0000001)$, $(x_{36},y_{36},z_{36})=(0.4245300,0.3915200,0.0000001)$, $(x_{37},y_{37},z_{37})=(0.4045300,0.4125200,0.0000001)$, we obtain three orbits of different colors: red, blue, and green (Figure \ref{4-3}). From Figure \ref{4-3}, we observe that the invariant circle coexists with the period-four saddle point $\{K_{11},K_{12},K_{13},K_{14}\}$.
Subsequently, by selecting parameter $(\lambda,\mu,\beta)=(1.0000,2.5035,4.9900)$ and the initial values $(x_{38},y_{38},z_{38})=(0.3995100,0.4000700,0.0000001)$, $(x_{39},y_{39},z_{39})=(0.3955300,0.4175200,0.0000001)$, the invariant circle persists, the blue orbit from the invariant circle tends to infinity and the red orbit from the fixed point $E_2$ tends to the invariant circle (Figure \ref{4-4}). Meanwhile, the period-four saddle point $\{K_{11},K_{12},K_{13},K_{14}\}$ lies in the invariant circle, as shown in Figure \ref{4-4}.
Furthermore, setting parameter $(\lambda,\mu,\beta)=(1.0000,2.5035,4.9902)$ and the initial values $(x_{310},y_{310},z_{310})=(0.3995100,0.4000700,0.0000001)$, $(x_{311},y_{311},z_{311})=(0.3945300,0.4175200,0.0000001)$, we obtain the red and blue orbits in Figure \ref{4-5}. From Figure \ref{4-5}, we observe that the period-four saddle point $\{K_{11},K_{12},K_{13},K_{14}\}$ is at the junction of the red and blue orbits, and the invariant circle starts to disappear.
Then, selecting parameter $(\lambda,\mu,\beta)=(1.0000,2.5035,4.9906)$ and the initial values $(x_{312},y_{312},z_{312})=(0.3995100,0.4000700,0.0000001)$, $(x_{313},y_{313},z_{313})=(0.3935300,0.4175200,0.0000001)$, we obtain the red and blue orbits as shown in Figure \ref{4-6}. From Figure \ref{4-6}, we see that the period-four saddle point $\{K_{11},K_{12},K_{13},K_{14}\}$ is at the junction of the red and blue orbits, and the invariant circle disappears completely.

\begin{figure}[htbp]
\T\T\T\T\T\T\T\T\T\T\T\T\T\T\T\T\T\T\T\T\T\T\T\T\T\T\T\T\T\T\T\T\T\T\T\T
\subfigure[$(\lambda,\mu,\beta)=(1.000,2.512,4.973)$]{
\includegraphics[width=7.5cm]{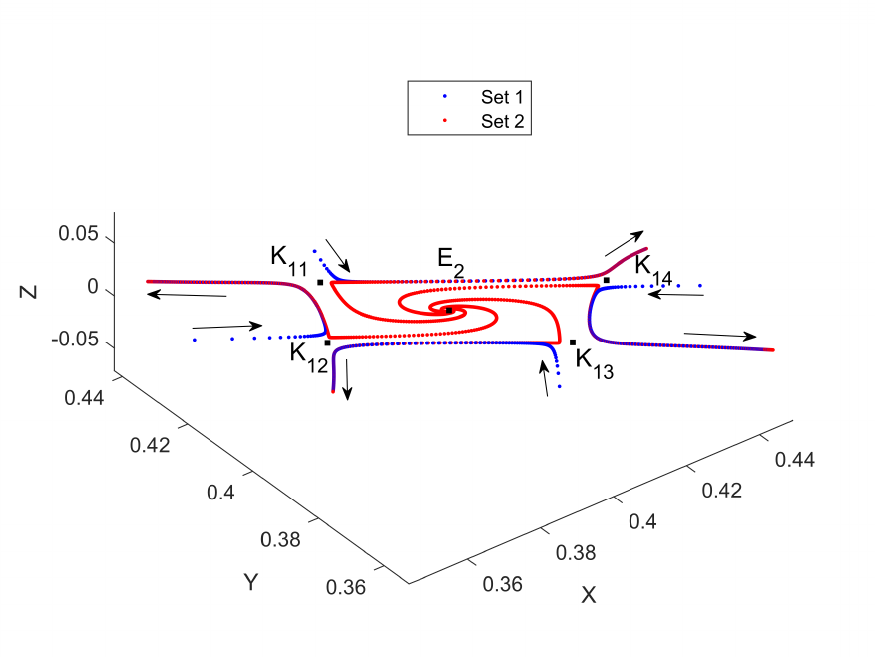}
%\caption{fig1}
\label{4-1}
}
\quad
\subfigure[$(\lambda,\mu,\beta)=(1.000,2.503,4.990)$]{
\includegraphics[width=7.5cm]{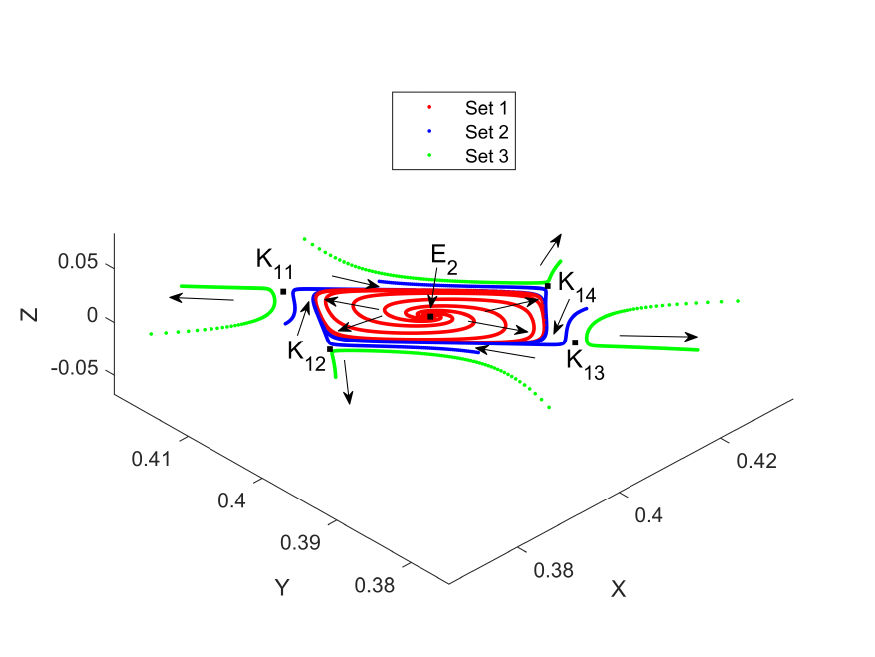}
\label{4-2}
}
\T\T\T\T\T\T\T\T\T\T\T\T\T\T\T\T\T\T\T\T\T\T\T\T\T\T\T\T\T\T\T\T\T\T\T\T
\subfigure[$(\lambda,\mu,\beta)=(1.0000,2.5025,4.9900)$]{
\includegraphics[width=7.5cm]{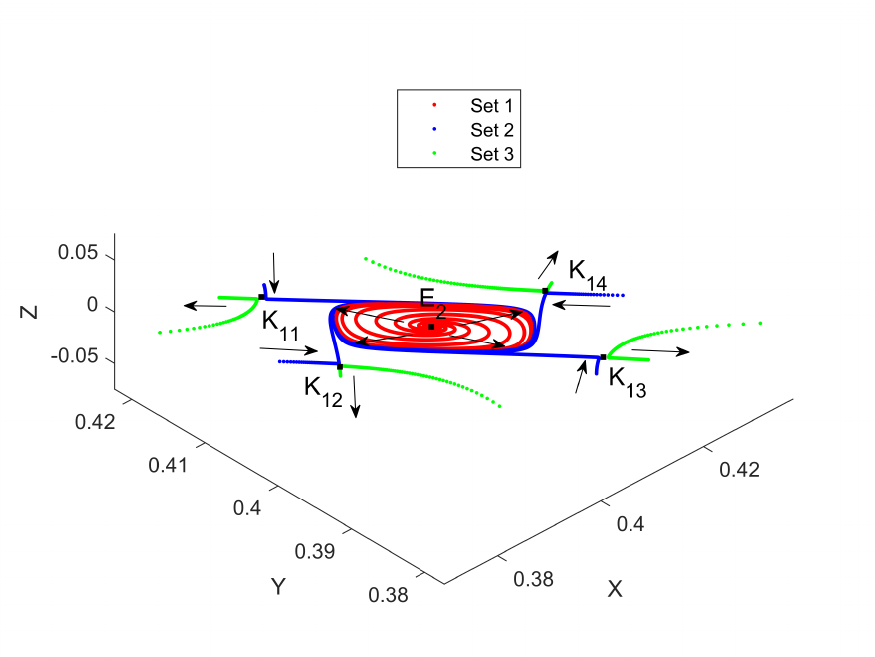}
\label{4-3}
}
\quad
\subfigure[$(\lambda,\mu,\beta)=(1.0000,2.5035,4.9900)$]{
\includegraphics[width=7.5cm]{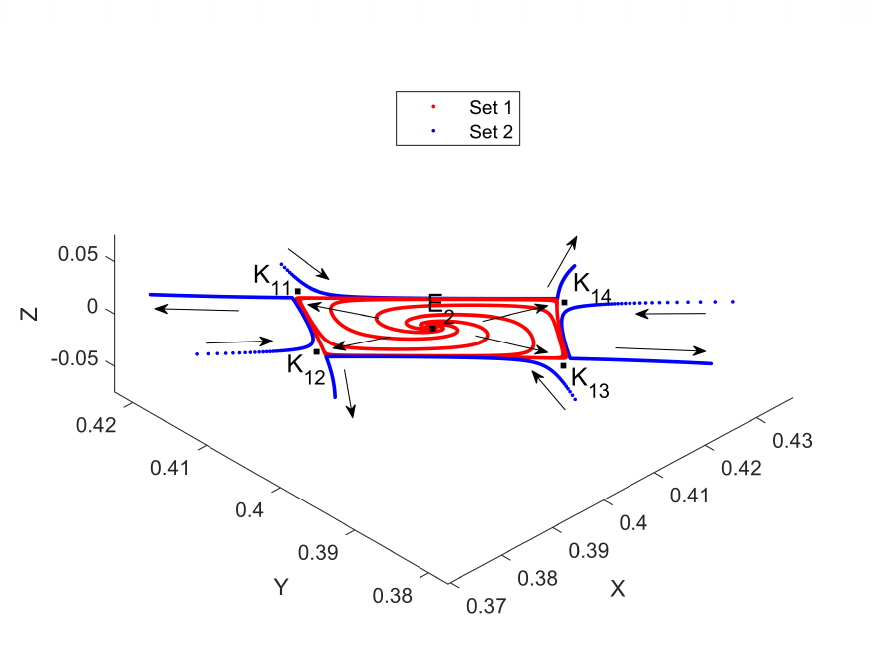}
\label{4-4}
}
\T\T\T\T\T\T\T\T\T\T\T\T\T\T\T\T\T\T\T\T\T\T\T\T\T\T\T\T\T\T\T\T\T\T\T\T
\subfigure[$(\lambda,\mu,\beta)=(1.0000,2.5035,4.9902)$]{
\includegraphics[width=7.5cm]{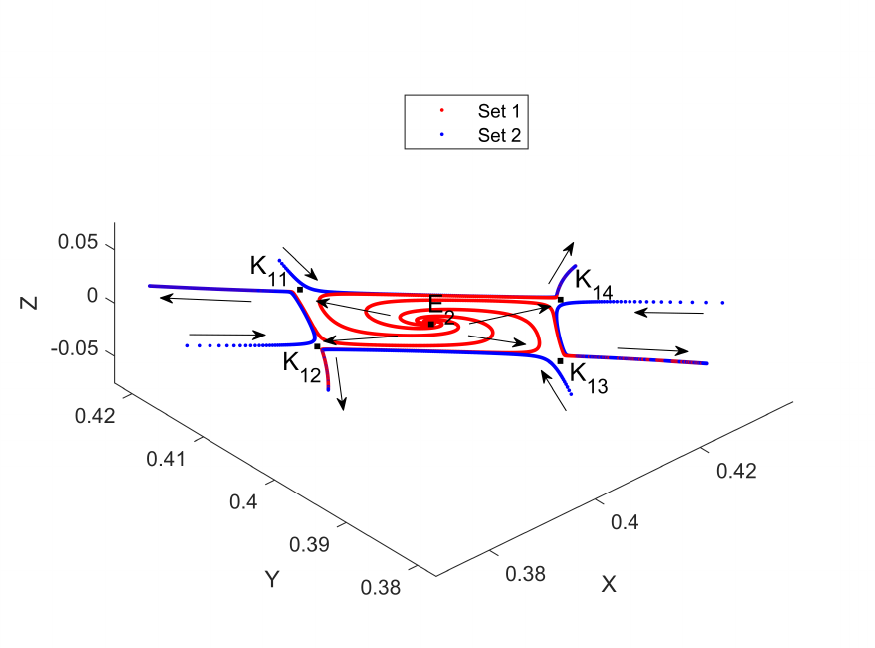}
\label{4-5}
}
\quad
\subfigure[$(\lambda,\mu,\beta)=(1.0000,2.5035,4.9906)$]{
\includegraphics[width=7.5cm]{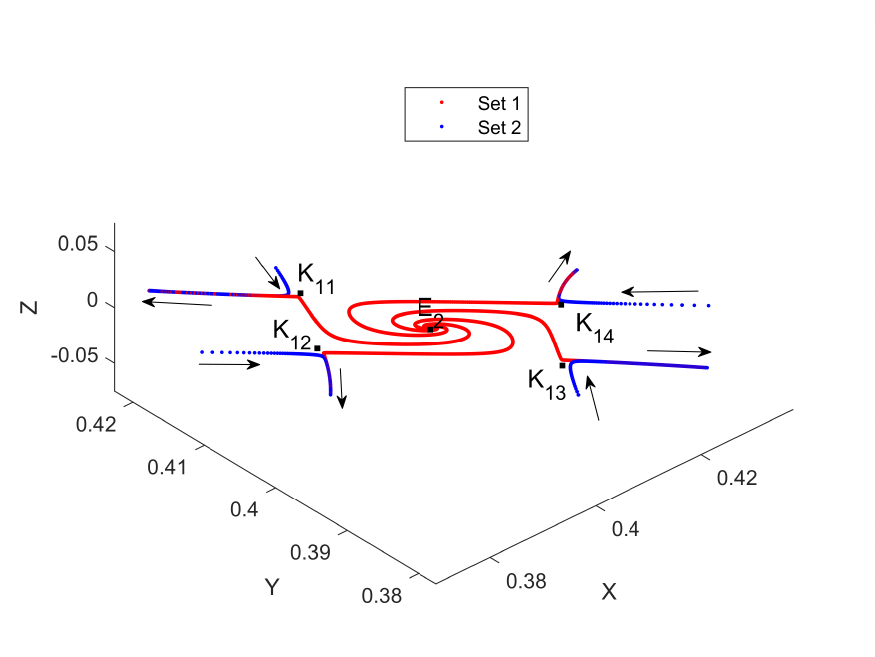}
\label{4-6}
}
\caption{Phase portraits of 1:4 resonance in system \eqref{eq2.1} }
\end{figure}

In what follows, we will give numerical simulations for Arnold tongue $\mathcal{A}_{1/5}$ and period-five orbits on the invariant closed curve generated from the Neimark-Sacker bifurcation. Taking $n=1$ and $m=5$, we get $\beta_*=(4 \cos\! \left(2\,\pi/5\right)-5)/(2 \left(\cos\! \left(2\,\pi/5\right)-1\right))\approx2.723606820$ and $\mu_*=3.763931938$. In this case, mapping \eqref{a2} is written as
\begin{eqnarray*}
&&z\mapsto t_1\,z+\mathbb{A}_{2,1}\,z^{2}\,\bar{z}^{1}+\mathbb{B}\,\bar{z}^{4}+O(|z|^5),
\end{eqnarray*}
where
\begin{eqnarray*}
&&t_1\approx0.3090170156 - 0.9510565096{\bf i},\\
&&\mathbb{A}_{2,1}\approx-1.138610132 + 6.638606465{\bf i},\\
&&\mathbb{B}\approx8.659782235 + 8.309112425{\bf i}.
\end{eqnarray*}
Further, one can check that
\begin{eqnarray*}
&&\check{\varrho}_{3}(0)\approx-6.665539811,~~\check{\varrho}_{2}(0)\approx0.9685596305,~~|\mathbb{B}|\approx12.00138117.
\end{eqnarray*}
By Theorem \ref{th5-1}, the Arnold tongue $\mathcal{A}_{1/5}$ is given by
\begin{eqnarray*}
\mathcal{A}_{1/5}=\left\{(\mu,\beta)\bigg|T_{-}+\frac{2\pi}{5}<arctan\left(\frac{\sqrt{4\beta^{2} \mu-4\beta^{2}-4\beta\mu-\mu^{2}}}{2 \beta-\mu}\right)<T_{+}+\frac{2\pi}{5}\right\},
\end{eqnarray*}
where
\begin{eqnarray*}
T_{\pm}\approx-0.1453085056\,\left(\sqrt{{\frac{\mu  \left(\beta -2\right)}{\beta}}}-1\right)
\pm20.33059665\left(\sqrt{{\frac{\mu  \left(\beta -2\right)}{\beta}}}-1\right)^{\frac{3}{2}}.
\end{eqnarray*}

Choosing a parameter $(\lambda,\mu,\beta)=(2.85000000,2.72622561,3.76422131)$, the Arnold tongue $\mathcal{A}_{1/5}$, i.e., the red region in Figure \ref{Arnold-1}, is simulated in the $(\mu, \beta)$-space by Maple 2023. Besides, using the same parameter in Matlab R2023a, we obtain a stable period-five orbit $\{S_1,S_2,S_3,S_4,S_5\}$ after $5\times10^5$ steps iterating on the invariant circle generated from the Neimark-Sacker bifurcation (see five green points in Figure \ref{Arnold-2}). The orbit of period-five is the limit set of the blue and red orbits with the initial values $(x_{40}, y_{40},z_{40})=(0.3800000,0.2500000,0.0022728)$ near the fixed point $E_2$ and $(x_{41}, y_{42},z_{41})=(0.3480,0.4032,0.0041)$ outside the invariant circle. Further, there exists a period-five saddle point $\{U_1,U_2,U_3,U_4,U_5\}$ near the blue points on the invariant circle.

Finally, by Theorem \ref{E2hd}, we numerically simulate the chaotic behaviors of system \eqref{eq2.1} in the sense of Marotto. Taking $(\lambda,\mu,\beta)=(9.14,2.50,3.36)$, system \eqref{eq2.1} has a fixed point $E_2:(x^*,y^*,z^*)=(0.2976190476,0.3023809524,0)$.
The modulus of all eigenvalues of the Jacobian matrix at $E_2$ are greater than $1$.
From Theorem~\ref{E2hd}, we can find a neighborhood of $E_2$ as $U_{E_2}=\{(x,y,z)|0.233198<x<0.521422,~0.002511<y<0.303401,~0<z<0.207524\}$ such that $E_2$ is expanding. Further, a point $E': (x',y',z')=(0.6834440850,0.2042699899,0)$ is also found, satisfying $F_{\lambda,\mu,\beta}^2(E')=E_2$ and $|DF_{\lambda,\mu,\beta}^2(E')|=2.54737397\ne0$.
Therefore, $E_2$ is a snap-back repeller. Taking the initial value $(x_{0},y_{0},z_{0})=(0.300,0.302,0.001)$, which is near the fixed point $E_2$, system \eqref{eq2.1} has chaotic behaviors in the sense of Marotto, shown in Figure \ref{hdE2}. The corresponding Lyapunov exponents diagram is also shown in Figure \ref{lya-hd}.
\begin{figure}[htbp]
\T\T\T\T\T\T\T\T\T\T\T\T\T\T\T\T\T\T\T\T\T\T\T\T\T\T\T\T\T\T\T\T\T\T\T\T
\subfigure[An Arnold tongue corresponding to the rotation number 1/5]{
\includegraphics[width=7.5cm]{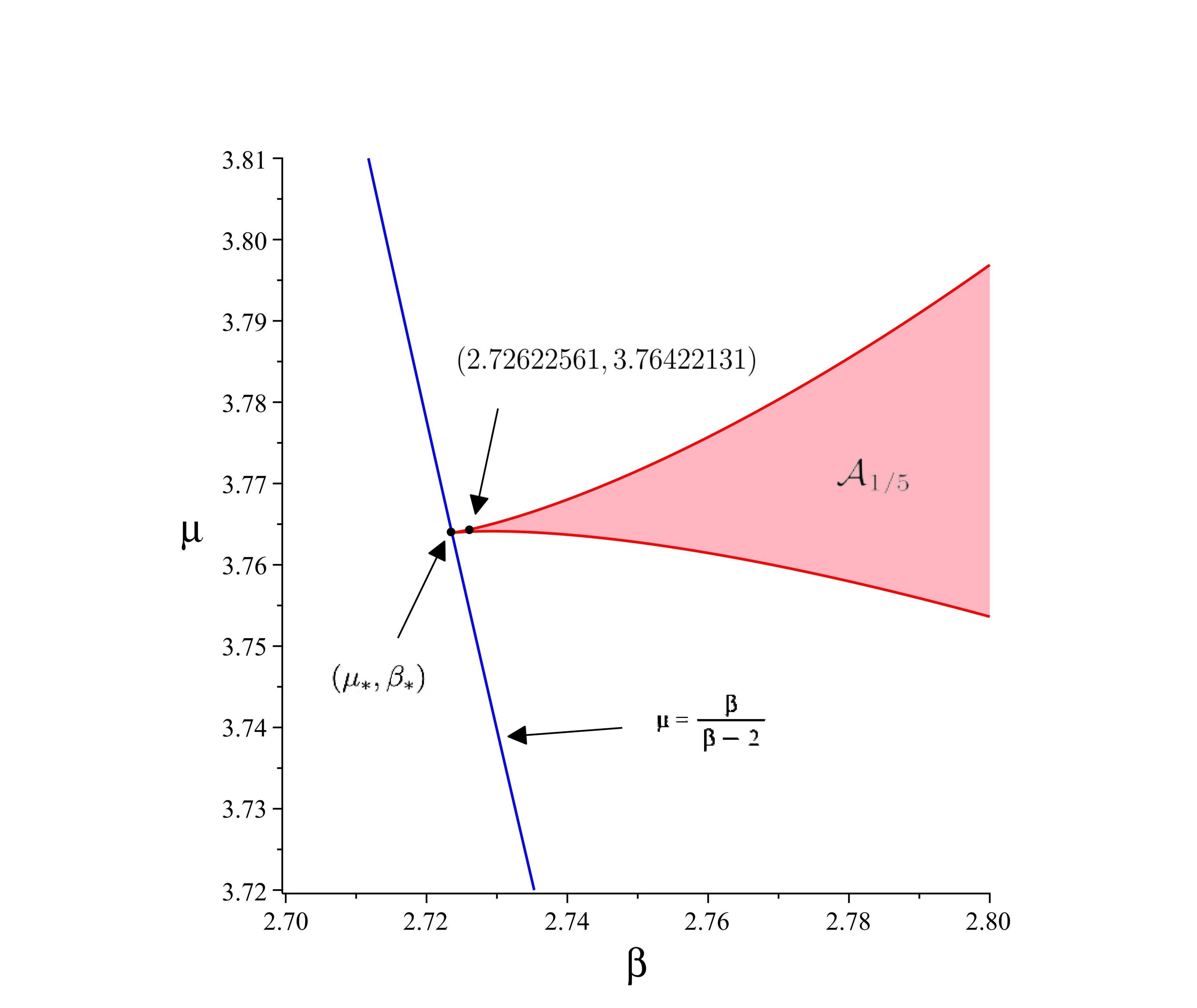}
\label{Arnold-1}
}
\quad
\subfigure[A stable period-five orbit $\{S_1,S_2,S_3,S_4,S_5\}$ on the invariant circle]{
\includegraphics[width=7.5cm]{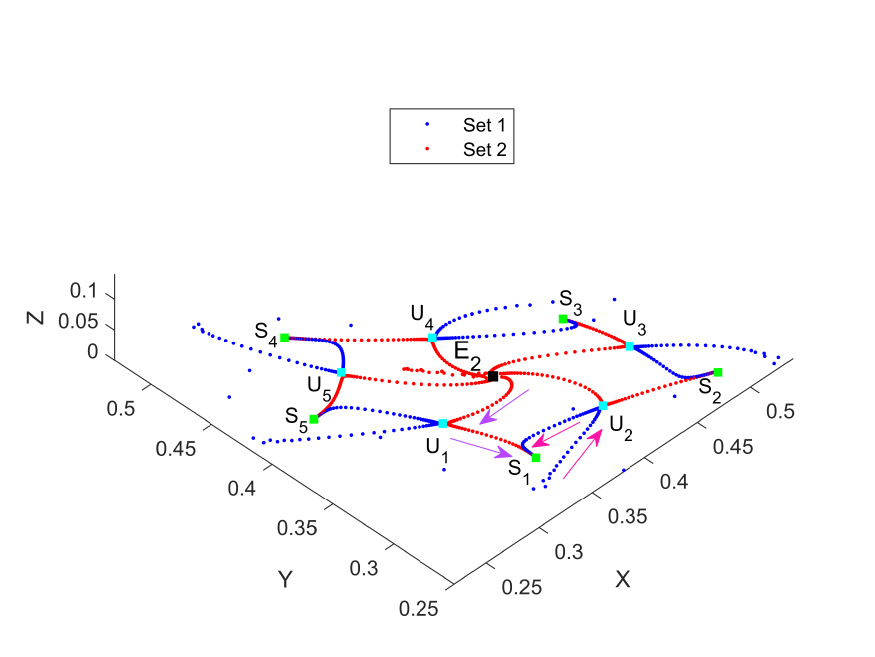}
\label{Arnold-2}
}
\caption{Arnold tongue and stable period-five orbit }
\end{figure}

\begin{figure}[htbp]
\T\T\T\T\T\T\T\T\T\T\T\T\T\T\T\T\T\T\T\T\T\T\T\T\T\T\T\T\T\T\T\T\T\T\T\T
\subfigure[The chaos diagram of system \eqref{eq2.1} in the sense of Marotto]{
\includegraphics[width=7.5cm]{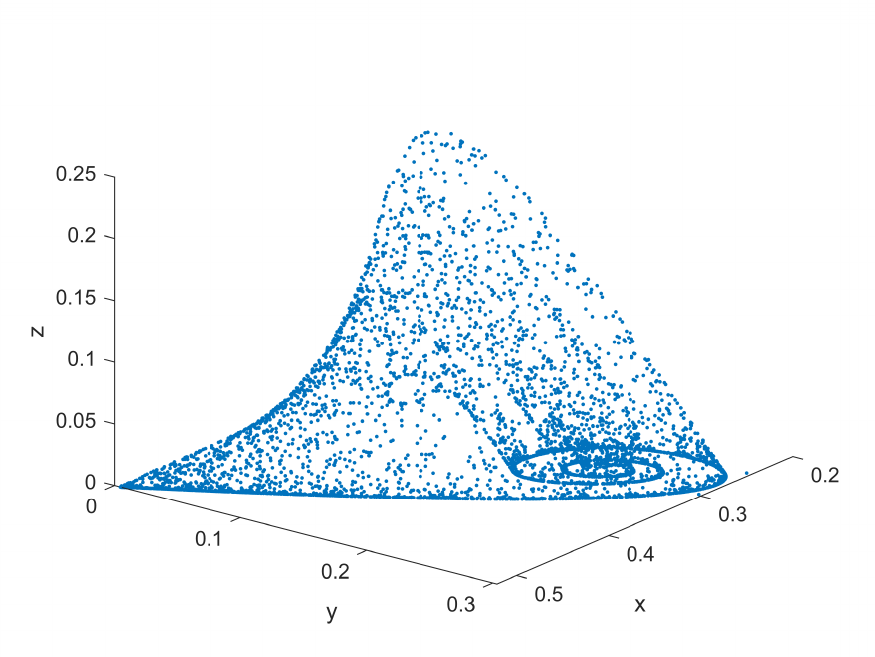}
\label{hdE2}
}
\quad
\subfigure[Lyapunov exponents diagram corresponding to (a)]{
\includegraphics[width=7.5cm]{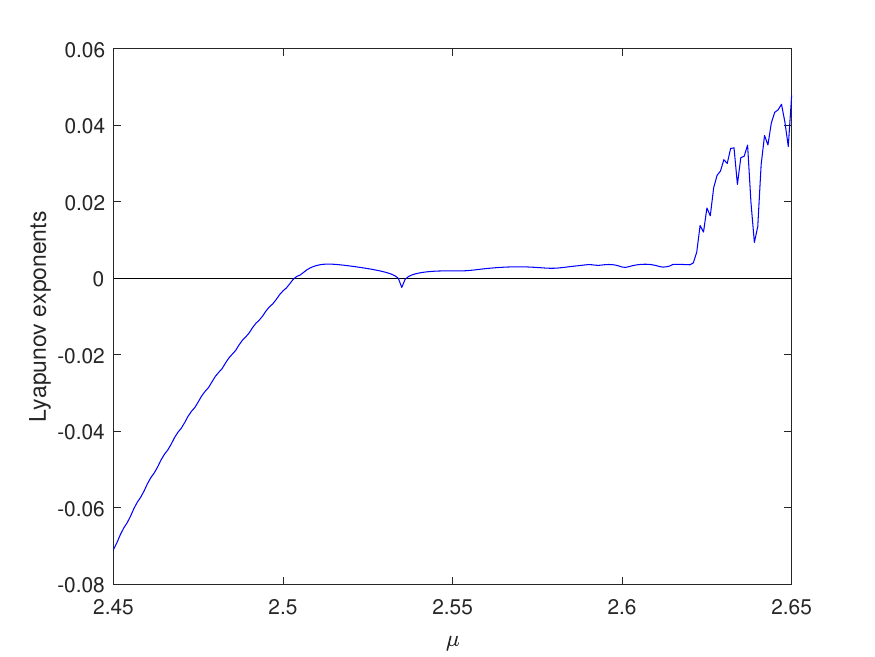}
\label{lya-hd}
}
\caption{\small The chaos diagram and corresponding Lyapunov exponents diagram of
system \eqref{eq2.1} in the sense of Marotto }
\end{figure}

\noindent{\bf Acknowledgements.}~
The authors are very grateful to the anonymous referees for their
useful comments and suggestions.

%%%%%%%%%%%%%%%%%%%%%%%%%%%%%%%%%%%%%%%%%%%%%%%%%%%%%%%%%%%%%%%%%%%%%%%%%%%%%%%%%%%%%%%%%%%%%%%%%%%%%%%%%%%%%%%%%%%%%%%%%%%%%%%%%%%%%%%%%%%%%%%%%%%%%%%%
\end{sloppypar}
\end{document}